\documentclass[11pt,
]{article}
\usepackage[margin=1in,letterpaper]{geometry}

\newcommand{\infdistance}[2]{\|#1-#2\|_\infty}
\newcommand{\fccVR}[1]{\opC_*(\Full(#1))}
\newcommand{\hatdi}{\widehat{d}_\mathrm{I}}
\newcommand{\tildi}{\tilde{d}_\mathrm{I}}
\newcommand{\tildb}{\tilde{d}_\mathrm{B}}
\newcommand{\hatdb}{\widehat{d}_\mathrm{B}}
\newcommand{\dbk}[1]{d_{\mathrm{B},#1}}
\newcommand{\hatdbk}[1]{\widehat{d}_{\mathrm{B},#1}}
\newcommand{\dballk}[2]{\db(\caB_{\Con,k}(#1),\caB_{\Con,k}(#2))}
\newcommand{\pbdbk}[3]{\widehat{d}_{\mathrm{B},#1}\left( #2,#3\right)}
\newcommand{\pbdbvariant}[5]{#2{d}_{\mathrm{B},#1}^{#3}\left( #4,#5\right)}
\newcommand{\pbdi}[2]{\widehat{d}_{\mathrm{I}}\left( #1,#2\right)} 
\newcommand{\pbdivariant}[4]{#1{d}_{\mathrm{I}}^{#2}\left( #3,#4\right)} 

\newcommand{\xoverbrace}[2][\vphantom{\dfrac{A}{A}}]{\overbrace{#1#2}}
\usepackage{caption}
\usepackage{stackengine,xcolor}
\usepackage{tcolorbox}
\newcommand\myboxred[2][]{\tikz[overlay]\node[fill=red!20,inner sep=2pt, anchor=text, rectangle, rounded corners=1mm,#1] {#2};\phantom{#2}}
\newcommand\myboxblue[2][]{\tikz[overlay]\node[fill=blue!20,inner sep=2pt, anchor=text, rectangle, rounded corners=1mm,#1] {#2};\phantom{#2}}

\usepackage[normalem]{ulem}
\usepackage{tikz}
\usepackage{tikz-cd}
\usetikzlibrary{cd}
\usepackage{tkz-euclide}
\usetikzlibrary{decorations.markings, arrows.meta,shapes,backgrounds,decorations.text,patterns,positioning,fit,calc,plotmarks,decorations.markings}
\tikzset{
    marrow/.style={decoration={markings,mark=at position 0.5 with {\arrow{#1}}}, postaction=decorate}
}
\tikzset{
  symbol/.style={
    draw=none,
    every to/.append style={
      edge node={node [sloped, allow upside down, auto=false]{$#1$}}}
  }
}
\usetikzlibrary{plotmarks}
\usepackage{imakeidx}
\makeindex
\usepackage{hyperref}
\hypersetup{
	colorlinks   = true,
	citecolor    = gray
}
\hypersetup{linkcolor=blue}

\usepackage{pgfplots}
\usepgfplotslibrary{polar}
\usepackage{mathtools}

\usepackage{textcomp}
\usepackage{tkz-euclide}
\usepackage{setspace}
\usepackage{amsmath}
\usepackage{amssymb}
\usepackage{mathabx}
\usepackage{bbm}
\usepackage{url}
\usepackage{soul}
\usepackage[all]{xy}
\usepackage[utf8]{inputenc}
\usepackage{amsfonts}
\usepackage{graphicx}
\usepackage{arydshln}
\usepackage{cancel}
\usepackage{enumitem}
\usepackage{extpfeil}
\usepackage[colorinlistoftodos]{todonotes}
\usepackage{blkarray, bigstrut}
\usepackage{parskip}
\usepackage{faktor}
\usepackage{xfrac}
\usepackage{mathrsfs}
\usepackage{framed}
\usepackage{latexsym}
\usepackage{pb-diagram}
\usepackage[mathscr]{euscript}
\usepackage{color}
\usepackage{multicol}
\usepackage{tikz-cd}

\usepackage{authblk}

\usepackage{etoolbox}
\patchcmd{\paragraph}{\itshape}{\bfseries\boldmath}{}{}

\usepackage{amsthm}
\begingroup
    \makeatletter
    \@for\theoremstyle:=definition,remark,plain\do{%
        \expandafter\g@addto@macro\csname th@\theoremstyle\endcsname{%
            \addtolength\thm@preskip\parskip
            }%
        }
\endgroup

\usepackage{thmtools, thm-restate}

\newtheorem{theorem}{Theorem}
\newtheorem{proposition}{Proposition}[section]
\newtheorem{lemma}[proposition]{Lemma}
\newtheorem{corollary}[proposition]{Corollary}
\newtheorem{definition}[proposition]{Definition}
\newtheorem{example}[proposition]{Example}

\newtheorem{remark}[proposition]{Remark}
\newtheorem{fact}[proposition]{Fact}

\newcounter{claimcount}
\newenvironment{claim}{\refstepcounter{claimcount}\par\medskip
   \noindent\emph{Claim \arabic{claimcount}:}}{\vspace{1pt}}
\AtBeginEnvironment{proof}{\setcounter{claimcount}{0}}

\DeclareMathAlphabet{\mymathbb}{U}{BOONDOX-ds}{m}{n}

\newcommand{\face}{\operatorname{face}}

\newcommand{\bbZ}{\mathbb{Z}}
\newcommand{\bbR}{\mathbb{R}}
\newcommand{\bbN}{\mathbb{N}}

\newcommand{\bbV}{\mathbb{V}}
\newcommand{\bbW}{\mathbb{W}}

\newcommand{\frM}{\mathfrak{M}}

\newcommand{\frX}{\mathfrak{X}}

\newcommand{\caB}{\mathcal{B}}
\newcommand{\caC}{\mathcal{C}}
\newcommand{\caE}{\mathcal{E}}

\newcommand{\caH}{\mathcal{H}}

\newcommand{\caX}{\mathcal{X}}
\newcommand{\field}{\mathbb{F}}

\newcommand{\caU}{\mathcal{U}}

\newcommand{\rmS}{\mathrm{S}}

\newcommand{\rmH}{\mathrm{H}}

\newcommand{\card}{\operatorname{card}}

\newcommand{\im}{\operatorname{im}}
\newcommand{\kernel}{\operatorname{ker}}
\newcommand{\rank}{\operatorname{rank}}

\newcommand{\dis}{\operatorname{dis}}
\newcommand{\codis}{\operatorname{codis}}
\newcommand{\cost}{\operatorname{cost}}

\newcommand{\vs}{\operatorname{\mathbf{Vec}}}

\newcommand{\fcc}{\operatorname{\mathbf{FCC}}}

\newcommand{\diam}{\operatorname{diam}}
\newcommand{\Id}{\operatorname{Id}}

\newcommand{\VR}{\operatorname{VR}}
\newcommand{\Full}{\operatorname{VR}}

\newcommand{\opC}{\operatorname{C}}

\newcommand{\Iso}{\operatorname{Iso}}

\newcommand{\supp}{\operatorname{supp}}

\newcommand{\Ver}{\mathrm{Ver}}
\newcommand{\Con}{\mathrm{Con}}

\newcommand{\id}{\operatorname{id}}

\newcommand{\Cor}{\operatorname{Cor}}
\newcommand{\Tri}{\operatorname{Tri}}
\newcommand{\Map}{\operatorname{Map}}

\newcommand{\dm}{d_\mathrm{M}}
\newcommand{\dgh}{d_\mathrm{GH}}

\newcommand{\di}{d_\mathrm{I}}

\newcommand{\dhaus}{d_\mathrm{H}}
\newcommand{\db}{d_\mathrm{B}}

\usepackage{accents}

\definecolor{darkblue}{rgb}{0.0, 0.0, 0.8}
\definecolor{darkred}{rgb}{0.8, 0.0, 0.0}
\definecolor{darkgreen}{rgb}{0.0, 0.8, 0.0}

\newcommand{\revision}[1]  {\color{black} #1 \color{black}}

\author[1]{Facundo~M\'emoli}
\author[2]{Ling~Zhou}

\affil[1]{Department of Mathematics, Rutgers University.
		\thanks{\texttt{facundo.memoli@rutgers.edu}}}
		
\affil[2]{Department of Mathematics, 
		Duke University.
		\thanks{\texttt{ling.zhou@duke.edu}}}

\begin{document}
\title{Ephemeral persistence features and the stability of filtered chain complexes}

\date{\today}

\maketitle
\begin{abstract}
    We strengthen the usual stability theorem for Vietoris-Rips (VR) persistent homology of finite metric spaces by building upon constructions due to Usher and Zhang in the context of filtered chain complexes. 
    The information present at the level of filtered chain complexes includes points with zero persistence which provide additional information to that present at homology level. 
    The resulting invariant, called \emph{verbose barcode}, which has a stronger discriminating power than the usual barcode, is proved to be stable under certain metrics that are sensitive to these ephemeral points. In some situations, we provide ways to compute such metrics between verbose barcodes. We also exhibit several examples of finite metric spaces with identical (standard) VR barcodes yet with different verbose VR barcodes thus confirming that these ephemeral points strengthen the standard VR barcode.
\end{abstract}
	
\tableofcontents


\section{Introduction}

In topological data analysis, \emph{persistent homology} is one of the main tools utilized for extracting and analyzing multiscale geometric and topological information from metric spaces. 

Typically, the \emph{persistent homology pipeline} (as induced by the Vietoris-Rips filtration) is explained via the diagram:
\begin{center}
Metric Spaces $\to$  Simplicial Filtrations $\to$ Persistence Modules
\end{center}
where,  from left to right, the second map is homology with field coefficients. Throughout the paper, we fix a base field $\field$.
We restrict our considerations to finite metric spaces, finite-dimensional simplicial complexes, and finite-dimensional chain complexes. Specifically, in this paper, for any chain complex $(C_*,\partial)$, the total dimension $\dim(C_*)=\sum_{k\geq 0}\dim(C_k)$ is finite.

Pairs of birth and death times of topological features (such as connected components, loops, voids, and so on) give rise to the \emph{barcode}, also called the \emph{persistence diagram}, of a given metric space \cite{edelsbrunner2002topological,carlsson2009topology}.
The so-called \emph{bottleneck distance} $\db$ between the persistent homology barcodes arising from the Vietoris-Rips filtration of metric spaces provides a polynomial time computable lower bound for the \emph{Gromov-Hausdorff distance} $\dgh$ between the underlying metric spaces. However,  this bound is not tight, in general (cf. Example \ref{ex:twopointspace}). A version of this theorem restricted to the case of finite metric spaces states:

\begin{restatable}[{Stability Theorem for $\db$, \cite{CCSGMO09,CSO14}}]{theorem}{dbstab}\label{thm:dgh-classical} 
Let $X$ and $Y$ be two finite metric spaces. Let $\caB_{ k}(X)$ (resp. $\caB_{ k}(Y)$) denote the barcode of the persistence module $\rmH_k\left(\VR_\bullet(X)\right)$ (resp. $\rmH_k\left(\VR_\bullet(Y)\right)$). Then, for any degree $k\in \bbZ_{\geq 0},$
\[\db(\caB_{ k}(X),\caB_{ k}(Y))\leq 2\cdot \dgh(X,Y).\]
\end{restatable}

In this paper, with the goal of refining the standard stability result alluded to above, we concentrate on the, usually implicit but conceptually important, intermediate step which assigns a \emph{filtered chain complex} (FCC) to a given simplicial filtration: 

\begin{center}
Metric Spaces $\to$  Simplicial Filtrations $\to$ \fbox{Filtered Chain Complexes} $\to$ Persistence Modules.
\end{center}

\paragraph{Related work on FCCs.} An FCC is an ascending sequence of chain complexes connected by monomorphisms. For instance, an FCC induced by a simplicial filtration $\{X_t\}_{t\in \bbR}$ can be represented by the following commutative diagram: for any $t\leq t'$,
\[
\begin{tikzcd}
 \opC_*(X_t):
\ar[d]
&
\cdots
\ar[r, "\partial_{k+2}"]
& 
 \opC_{k+1}(X_t)
\ar[r, "\partial_{k+1}"]
\ar[d, hookrightarrow]
& 
 \opC_k(X_t)
\ar[r, "\partial_{k}"]
\ar[d, hookrightarrow]
& \cdots
\\
 \opC_*(X_{t'}):
&
\cdots
\ar[r, "\partial_{k+2}"]
& 
 \opC_{k+1}(X_{t'})
\ar[r, "\partial_{k+1}"]
& 
 \opC_k(X_{t'})
\ar[r, "\partial_{k}"]
& \cdots
\end{tikzcd},
\]
where each $X_t$ is a simplical complex and $\opC_*(X_t)$ denotes the simplical chain complex of $X_t.$ 

Studies of the decomposition of FCCs in several different settings can be found in \cite{usher2016persistent,de2011dualities,meehan2019structural,chacholski2021invariants,chacholski2020decomposing}. 
We follow the convention of Usher and Zhang \cite{usher2016persistent}, where they study a notion of Floer-type complexes as a generalization of FCCs and prove a stability result for the usual bottleneck distance between concise barcodes of Floer-type complexes. 
In particular, they studied FCCs in detail and considered the notion of \emph{verbose barcode} $\caB_{\Ver,k}$ of FCCs, which consists of the standard barcode (which the authors call \emph{concise barcode} and denote as $\caB_{\Con,k}:=\caB_k$
) together with additional \emph{ephemeral} bars, i.e. bars of length $0$. 

They also proved that every FCC decomposes into the direct sum of  indecomposables $\caE(a,a+L,k)$,  which they called  \emph{elementary FCCs}, of the following form (see Definition \ref{def:elementary-f.c.c.}): if $L\in [0,\infty)$ and $a\in \bbR$, then $\caE(a,a+L,k)$ is given by 
\[
\begin{tikzcd}[column sep=5em, row sep=small]
t<a:\hspace{2em}
\cdots\to 0
\ar[r, "\partial_{k+2}=0"]
& 
0
\ar[r, "\partial_{k+1}=0"]
\ar[d, "="]
& 
0
\ar[r, "\partial_{k}=0"]
\ar[d, hookrightarrow]
& 
0 \to \cdots \\
t\in [a,a+L):\hspace{0.5em}
\cdots\to 0
\ar[r, "\partial_{k+2}=0"]
& 
0
\ar[r, "\partial_{k+1}=0"]
\ar[d, hookrightarrow]
& 
\field x
\ar[r, "\partial_{k}=0"]
\ar[d, "="]
&0\to \cdots
\\
t\in [a+L,\infty):\hspace{0.5em}
\cdots\to 0
\ar[r, "\partial_{k+2}=0"]
& 
\field y
\ar[r, "\partial_{k+1}:\,y\,\mapsto\, x"]
& 
\field x
\ar[r, "\partial_{k}=0"]
& 0\to \cdots,
\end{tikzcd}
\]
where $\field x$ denotes the vector space generated by $x.$
If $L=\infty$, then $\caE(a,\infty,k)$ (with the convention that $a+\infty=\infty$) is given by 
\[
\begin{tikzcd}[column sep=5em, row sep=small]
t<a:\hspace{2em}
\cdots\to 0
\ar[r, "\partial_{k+2}=0"]
& 
0
\ar[r, "\partial_{k+1}=0"]
\ar[d, "="]
& 
0
\ar[r, "\partial_{k}=0"]
\ar[d, hookrightarrow]
& 0\to \cdots\\
t\in [a,\infty):\hspace{0.5em}
\cdots\to 0
\ar[r, "\partial_{k+2}=0"]
& 
0
\ar[r, "\partial_{k+1}=0"]
& 
\field x
\ar[r, "\partial_{k}=0"]
& 0\to \cdots
\end{tikzcd}
\]
The degree-$l$ verbose barcode of the elementary FCC $\caE(a,a+L,k)$ is $\left\{  (a,a+L)\right\}$, where each pair $(a,a+L)$ is called a bar and $L$ is its length, for $l=k$ and is empty for $l\neq k$. 

The concise barcode of an FCC is defined as the collection of non-ephemeral bars, i.e. bars corresponding to elementary FCCs with $L\neq 0$ in its decomposition, which agrees with the standard barcode. Indeed, the $k$-th persistent homology of the elementary FCC $\caE(a,a+L,k)$ is the interval persistence module associated with the interval $[a,a+L)$, for $L\in [0,\infty]$. In particular, $\rmH_k(\caE(a,a,k))$ is the trivial persistence module.

In real calculations, barcodes are often computed for simplexwise filtrations first (i.e., simplices are assumed to enter the filtration one at a time), in which case all elementary FCCs correspond to intervals with positive length. This implies that, although not outputted, verbose barcodes are computed in many persistence algorithms. For VR FCCs, we made a small modification of the software Ripser introduced by Bauer (see \cite{bauer2021ripser}) to extract verbose barcodes of finite metric spaces.

In \cite{chacholski2021invariants}, Chach{\'o}lski et al. studied invariants for tame parametrized chain complexes, which are a generalization of filtered chain complexes obtained by allowing maps between chain complexes to be non-injective. In the finite-dimensional case, their notions of \emph{Betti diagram} and \emph{minimal Betti diagram} for filtered chain complexes respectively boil down to the verbose barcode and concise barcodes introduced in \cite{usher2016persistent}. In a subsequent paper \cite{chacholski2020decomposing}, the authors introduced an algorithm for decomposing filtered chain complexes into indecomposables.
Giunti and Landi reported \cite{giunti-landi-22} having independently explored ideas similar to the ones in our paper.

When a filtered chain complex arises from a simplicial filtration, its verbose barcode can also be obtained through the usual matrix reduction procedure applied to the boundary matrix of the underlying simplicial filtration. In this simplicial setting, in \cite{fasy2019faithful,micka2020searching,fasy2022efficient,Schenfisch2023} the authors study problems related to the reconstruction of simplicial complexes embedded in $\mathbb{R}^d$ via verbose barcodes (which they call ``augmented persistence diagrams"). 

\medskip

\paragraph{Overview of our results.} 
One drawback of the bottleneck stability result described in Theorem \ref{thm:dgh-classical} is that one asks for optimal matchings between the concise (i.e. standard) barcodes $\caB_{\Con,k}(X)$ and $\caB_{\Con,k}(Y)$ for each individual degree $k$ \emph{independently}. 


With the goal of finding a \emph{coherent} or \emph{simultaneous} matching of barcodes across all degrees at once, 
we study the interleaving distance $\di$ between FCCs (see Definition \ref{def:interleaving}) and establish an isometry theorem between $\di$ and the \emph{matching distance} $\dm$ between the verbose barcodes (see Definition \ref{def:dist-match}):

\begin{restatable}[Isometry theorem]{theorem}{isothm}\label{thm:dm=di}
For any two filtered chain complexes $(C_*,\partial_C,\ell_C)$ and $(D_*,\partial_D,\ell_D)$,
\[\sup_{k\in \bbZ_{\geq 0}}\dm
    \left(  \caB_{\Ver,k}(C_*), \caB_{\Ver,k}(D_*)\right)   
    =
    \di\left( \left(  C_*,\partial_C,\ell_C\right) ,\left(  D_*,\partial_D,\ell_D \right) \right) .\]
\end{restatable}

To prove that $\dm\leq\di$ (see Section \ref{sec:dm>=di}), we adapt ideas implicit in \cite[Proposition 9.3]{usher2016persistent} which the authors applied to establish the stability of \emph{Floer}-type complexes (on the same underlying chain complex). 
For the other direction, $\dm\geq \di$ (see Section \ref{sec:dm>=di}), we employ an approach similar to the one used to demonstrate that the standard bottleneck distance between concise barcodes is upper bounded by the interleaving distance between persistent modules, cf. \cite[Theorem 3.4]{lesnick2015theory}. 

In contrast to $\db$ between concise barcodes, $\dm$ between verbose barcodes of VR FCCs is not stable under the Gromov-Hausdorff distance between metric spaces. Indeed, $\dm$ is only finite if the two underlying metric spaces have the same cardinality. We remedy this issue in Section \ref{sec:pb stab} by incorporating the notion of tripods as in \cite{memoli2017distance}.

Let $(X,d_X)$ be a metric space, $Z$ a set and $\phi_X: Z\twoheadrightarrow X$ a surjective map. We equip $Z$ with the pullback vectors $\phi_X^*d_X$ of the distance function $d_X$ and call the pair $(Z,\phi_X^*d_X)$ the \emph{pullback (pseudo-metric) space} (induced by $\phi_X$). 
A \emph{tripod} between two sets $X$ and $Y$ is a pair of surjections from a common set $Z$ to $X$ and $Y$ respectively, which will be expressed by a tuple $(Z,\phi_X,\phi_Y)$ or a diagram \label{para:pb m.s.}
$$X\xtwoheadleftarrow{\phi_X}Z\xtwoheadrightarrow{\phi_Y}Y.$$ 
Given a degree $k$, we define the \emph{pullback bottleneck distance (induced by degree-$k$ verbose barcodes)} between two finite metric spaces $X$ and $Y$ to be the infimum of the matching distance between the degree-$k$ verbose barcodes of the VR FCCs induced by the respective pullbacks $(Z,\phi_X^*d_X)$ and $(Z,\phi_Y^*d_Y)$, where the infimum is taken over \emph{finite tripods} $X\xtwoheadleftarrow{\phi_X}Z\xtwoheadrightarrow{\phi_Y}Y$. We denote the resulting quantity by $\hatdbk{k}$; see Definition \ref{def:pb db}. When it is not necessary to specify a particular degree $k$, we will write $\hatdb$ instead of $\hatdbk{k}$.

Similarly, we define the \emph{pullback interleaving distance (induced by VR FCCs)} between metric spaces, and denote it by $\hatdi$ (see Definition \ref{def:pb di}). 

\begin{remark}[Terminology]\label{rem:no-triangle}
We point out the following regarding the use of the term `distance' when referring  to $\hatdbk{k}$ and $\hatdi$: 
\begin{itemize}
    \item [(1)] $\hatdbk{0}$ satisfies the triangle inequality Corollary \ref{cor:hatdb zero triangle}.
    \item [(2)] For $k>0$, the question of whether $\hatdbk{k}$ satisfies the triangle inequality is still open.
    \item [(3)] $\hatdi$ does not satisfy the triangle inequality; see Section \ref{subsubsec:metrize} for details. 
\end{itemize} 
Due to Items (2) and (3), the term `distance' is being abused through the use of the terminology `pullback bottleneck distance' and `pullback interleaving distance'. We do so for consistency with Item (1) and due to the fact that in Section \ref{subsubsec:metrize} we provide a way to modify $\hatdi$ and $\hatdb$ so that they do satisfy the triangle inequality (while still being Gromov-Hausdorff stable), thus making them pseudo-metrics between metric spaces.

In general, the pullback bottleneck distance $\hatdb$ (or the pullback interleaving distance $\hatdi$) depends on the underlying metric spaces, rather than solely on the verbose barcodes (or FCCs). However, we prove that $\hatdbk{0}$ depends only on the barcodes; see Proposition \ref{prop:hatdb degree 0}.
\end{remark}

It follows from Theorem \ref{thm:dm=di} that we immediately have the following:

\begin{restatable}{corollary}{pbisothm}
\label{cor:hatdb-hatdi}
Let $(X,d_X)$ and $(Y,d_Y)$ be two finite metric spaces. Then,
\[
\begin{tikzcd}[ampersand replacement=\&,column sep=small]
\sup\limits_{k} \inf\limits_{(Z,\phi_X,\phi_Y)} \dm   \left(  \caB_{\Ver,k}(Z_X), \caB_{\Ver,k}(Z_Y)\right)
\ar[r,phantom, "\leq" description]
\ar[d, phantom, "\rotatebox{90}{=}" description]
\&
\inf\limits_{(Z,\phi_X,\phi_Y)} \sup\limits_{k} \dm   \left(  \caB_{\Ver,k}(Z_X), \caB_{\Ver,k}(Z_Y)\right)
\ar[d, phantom, "\rotatebox{90}{=}" description]
\\
\sup\limits_{k} \pbdbk{k}{X}{Y} 
\& 
\pbdi{X}{Y}.
\end{tikzcd}
\]
\end{restatable}

In the theorem below, we show that the pullback bottleneck distance $\hatdb$ is stable under the Gromov-Hausdorff distance $\dgh$, and that the bottleneck distance $\db$ between degree-$k$ concise barcodes is never larger than $\hatdbk{k}$. We show in several examples below and in Section \ref{sec:ex for strict} that, between degree-$k$ concise barcodes, $\hatdbk{k}$ can be strictly larger than $\db$. Thus, the stability of $\hatdb$ provides a better lower-bound estimate of $\dgh$, compared with the standard bottleneck distance $\db$ (cf. Theorem \ref{thm:dgh-classical}). See Section \ref{sec:pb stab} for the proof of Theorem \ref{thm:hatdb-dgh stability}.

\begin{restatable}[Pullback stability theorem]{theorem}{pbstabthm}\label{thm:hatdb-dgh stability}
Let $(X,d_X)$ and $(Y,d_Y)$ be two finite metric spaces. Then, for any $k\in \bbZ_{\geq 0}$,
\begin{equation}\label{eq:pb-stab}
    \db \left( \caB_{\Con,k}(X), \caB_{\Con,k}(Y)\right) \leq \pbdbk{k}{X}{Y} 
    \leq\pbdi{X}{Y}
    \leq 2\cdot\dgh(X,Y).
\end{equation}
\end{restatable}

See Figure \ref{fig:3-point-barcode-full} for a pair of 3-point metric spaces for which the bottleneck distance $\db$ between their concise barcodes fails to distinguish them, but the pullback bottleneck distance $\hatdb$ induced by verbose barcodes succeeds at telling them apart. 

\begin{figure}[ht]
    \centering
\renewcommand{\arraystretch}{1.2}
\begin{tabular}{ | c| c| c|  } 
\hline
&
\begin{tikzpicture}[scale=1]
\node (a) at (0,0) {};
\node (b) at (0.7,1.2) {};
\node (c) at (2.6,0) {};
\filldraw (a) [color=green] circle[radius=2pt];
\filldraw (b)[color=green] circle[radius=2pt];
\filldraw (c) [color=green] circle[radius=2pt];
\draw [dashed]  (a)--(b) node[pos=0.5,fill=white]{$a$};
\draw [dashed] (c)--(b)node[pos=0.5,fill=white]{$b$};
\draw [dashed] (a) --(c) node[pos=0.5,fill=white]{$c_1$};
\node at (1.3,-.7) {$X_1$};
\end{tikzpicture} 
&
\begin{tikzpicture}[scale=1]
\node (a) at (0,0) {};
\node (b) at (1,1) {};
\node (c) at (3,0) {};
\filldraw (a) [color=magenta] circle[radius=2pt];
\filldraw (b)[color=magenta] circle[radius=2pt];
\filldraw (c) [color=magenta] circle[radius=2pt];
\draw [dashed]  (a)--(b) node[pos=0.5,fill=white]{$a$};
\draw [dashed] (c)--(b)node[pos=0.5,fill=white]{$b$};
\draw [dashed] (a) --(c) node[pos=0.5,fill=white]{$c_2$};
\node at (1.5,-.7) {$X_2$};
\end{tikzpicture} 
\\ 
\hline
$\caB_{\Ver,0}$ 
& $(0,a),(0,b), (0,\infty)$ 
& $(0,a),(0,b), (0,\infty)$ \\ 
\hline
$\caB_{\Ver,1}$ 
& $(c_1,c_1)$ 
& $(c_2,c_2)$ \\ 
\hline
\end{tabular}

\medskip

\begin{tabular}{ |c| c| c| c| } 
\hline
$\sup_{k\in\bbZ_{\geq 0}}\db\left(  \caB_{\Con,k}(X_1), \caB_{\Con,k}(X_2)\right) $
& $ \sup_{k\in\bbZ_{\geq 0}}\pbdbk{k}{X_1}{X_2} $
&$2\cdot\dgh(X_1,X_2)$
\\ 
\hline
$0$ & $|c_1-c_2|$ 
& $|c_1-c_2|$ \\ 
\hline
\end{tabular}
    \caption{\textbf{First table:} Three-point metric spaces $X_1$ and $X_2$ together with their verbose barcodes. Here $a\leq b\leq c_i$ for $i=1,2.$ \textbf{Second table:} the bottleneck distance between concise barcodes, the pullback bottleneck distance and twice of the Gromov-Hausdorff distance between $X_1$ and $X_2$. See Example \ref{ex:hatdb-3-point space}.}
    \label{fig:3-point-barcode-full}
\end{figure}

In Section \ref{sec:variation of pb distances},
we introduce two variants of the pullback interleaving/bottleneck distance (see Definition \ref{def:pb di R} and \ref{def:pb di M}), which offer advantages in terms of computational efficiency (see Section \ref{subsec:multi vector}).
We refer to all variations of pullback interleaving and bottleneck distances as `\emph{pullback distances}'. 
We show that all pullback distances are stable under the Gromov-Hausdorff distance $\dgh$ between metric spaces and they provide better lower bounds for $\dgh$ than the bottleneck distance between the concise barcodes; see Theorem \ref{thm:all hatdi and hatdb}.

In order to have a more concrete understanding of the pullback bottleneck distance and in order to explore the possibility of computing it, we study the relation between the verbose barcode of a pullback space $(Z,\phi_X^*d_X)$ with the verbose barcode of the original space $X$. 
We conclude that the verbose barcodes of $Z$ and $X$ only differ on some distinguished diagonal points; see Proposition \ref{prop:pullback-barcode} below and its proof in Section \ref{sec:pb barcode proof}. 

We now set up some notation about multisets\footnote{We use the notation $\{\cdot\}$ for multisets as well when its meaning is clear from the content.
}. 
Recall that a \emph{multiset} consists of a set $X$ together with a multiplicity function $m_X:X\to \bbZ_{\geq 0}$. 
The \emph{support} of a multiset is defined as $\supp(X,m_X):=\{x\in X\mid m_X(x)>0\}$.
In this paper, we will adopt the following notation to denote multisets: for an element $x\in X$, and for a non-negative integer $m$, $x^m$ will be understood to mean that $x$ has multiplicity $m$ i.e. $m_X(x) =m$. 
So that, for example,  when we write $\{x_1^4,x_2,x_3^{21}\}$ we mean the multiset where $X=\{x_1,x_2,x_3\}$ and $m_X(x_1)=4$, $m_X(x_2)=1$, $m_X(x_3)=21$. For convenience, for a non-negative integer $m$, by $\{x\}^m$ we will denote the multiset containing exactly $m$ copies of $x$. In other words, $\{x\}^m = \{x^m\}.$ 

For a multiset $A$, we define its \emph{cardinality}, $\card(A)$, as the sum of multiplicities of its elements.
A \emph{sub-multiset} \( A' \) of \( A \), denoted \( A' \subset A \), is a multiset whose support is a subset of the support of \( A \) and whose elements have multiplicities satisfying \( m_{A'}(a) \leq m_{A}(a) \) for all \( a \in A' \).
For any $l\geq 1$, we let 
\begin{equation}\label{eq:P_l(A)}
    P_l(A):=\left\{  A'\subset A:\card(A')=l\right\},
\end{equation}
that is, $P_l(A)$ consists of sub-multisets of $A$ each with cardinality $l$. Let \revision{$P_0(A)=\{\emptyset\}.$}

\begin{restatable}[Initial formula for pullback barcodes]{proposition}{pbbarcode}
\label{prop:pullback-barcode} Let $k\geq 0$ and $m\geq 1$, and let $X$ be a finite pseudo-metric space. For $\left\{x_{j_1},\dots,x_{j_m}\right\}\subset X$ for some $j_1\leq \dots\leq j_m$, consider the multiset $Z=X\sqcup \left\{x_{j_1},\dots,x_{j_m}\right\}$. Then, for $k\geq 0$,
\begin{equation}\label{eq:pullback barcode}
\caB_{\Ver,k}(Z)= \caB_{\Ver,k}(X) \sqcup
\bigsqcup_{i=0}^{m-1}\,
\bigsqcup_{\beta_i\in  P_{k}\left(  (X\setminus \{x_{j_{i+1}}\})\sqcup\left\{  x_{j_{1}},\dots,x_{j_{i}}\right\}\right) }
\left\{ \diam (\{x_{j_{i+1}}\}\sqcup\beta_i)\cdot(1,1)  \right\}.
\end{equation}
In particular, $\caB_{\Ver,0}(Z)= \caB_{\Ver,0}(X)\sqcup \bigsqcup_{i=0}^{m-1} \{\diam(\{
x_{j_{i+1}}\})\cdot (1,1)\}= \caB_{\Ver,0}(X)\sqcup \{(0,0)\}^ m$. 
\end{restatable}
Because the concise barcode can be obtained from the verbose barcode by excluding all diagonal points, the above proposition interestingly implies that $\caB_{\Con,k}(Z)= \caB_{\Con,k}(X)$ for any degree $k.$

To better understand Equation (\ref{eq:pullback barcode})  in the case when $k\geq 1$, we give a graphical explanation in Figure \ref{fig:pb barcode}. 
Let $(X,d_X)$ be a finite metric space with $X=\{x_1,\dots,x_n\}$.
Each finite pullback space $(Z,\phi_X^*d_X)$ of $X$ can be regarded as a multiset $X\sqcup \left\{x_{j_1},\dots,x_{j_m}\right\}$ equipped with the pullback pseudo-metric $\phi_X^*d_X$ induced from $d_X$, for some $m\geq 0$ and $1\leq j_1\leq \dots\leq j_m\leq n$; see Remark \ref{rmk:pullback spaces}. 
In other words, the extra points in $Z$ are `repeats' of the points in $X$. 
We will call each point in $X$ the \emph{parent} of its repeated copies. 
More specifically, for each $z\in Z$, the point $\phi_X(z)\in X$ will be called the parent of $z$.
We identify $Z$ with
\[X\sqcup\Big\{ \underbrace{x_1,\dots, x_1}_{m_1}, \dots, \underbrace{x_n,\dots, x_n}_{m_n} \Big\}
,\] 
where each $m_j\geq 0$ is the multiplicity of the extra copies of $x_j$ in $Z$ and $m_1+\dots+m_n=m$. We call $\vec{m}:=(m_1,\dots,m_n)$ the \emph{pullback vector} associated with $Z$.

\begin{figure}[ht]
    \centering
\begin{align*}
& i=0: & \xoverbrace{x_1,\text{\myboxred{$x_2,\dots,x_n$}}}^{= X},\, \xoverbrace{\text{\myboxblue{$x_1$}}, x_1,\dots,x_1}^{m_1},\, \xoverbrace{x_2,x_2,\dots,x_2}^{m_2},\dots,\, \xoverbrace{x_n,x_n,\dots,x_n}^{m_n}\\
& i=m_1-1: & x_1,\text{\myboxred{$x_2,\dots,x_n,\, x_1,x_1,\dots$}}, \text{\myboxblue{$x_1$}},\, x_2,x_2,\dots,x_2,\dots,\, x_n,x_n,\dots,x_n\\
& i=m_1: & \text{\myboxred{$x_1$}},x_2,\text{\myboxred{$\dots,x_n,\, x_1,x_1,\dots, x_1$}},\, \text{\myboxblue{$x_2$}},x_2,\dots,x_2,\dots,\, x_n,x_n,\dots,x_n\\
& i=m_1+\dots+m_{n-1}: & \text{\myboxred{$x_1,x_2,\dots$}}, x_n,\, \text{\myboxred{$x_1,x_1,\dots,x_1,\, x_2,x_2,\dots,x_2,\dots$}},\, \text{\myboxblue{$x_n$}},x_n,\dots,x_n
\end{align*}
    \caption{Using the notation from Equation (\ref{eq:pullback barcode}), for each \( i \) (i.e., for each row), the point \( x_{j_{i+1}} \) is colored blue. 
    For each $i$, multiset \( \beta_i \) in Equation (\ref{eq:pullback barcode}) ranges over all \( k \)-element sub-multisets of the red-colored multiset. 
    Notice that each red-colored multiset consists of every point before $x_{j_{i+1}}$  (from left to right) 
     excluding the parent  of $x_{j_{i+1}}$.} 
    \label{fig:pb barcode}
\end{figure}

In Section \ref{subsec:pb barcodes explicit}, we prove the following proposition which provides an explicit formula both for the coordinates of the extra diagonal points and for their multiplicity in all degrees (see page \pageref{para:mu_k notation} for the notation $\mu_k(\vec{m}(I_p))$). As above, we let $X=\{x_1,\dots,x_n\}$.
\begin{restatable}[Explicit formula for pullback barcodes]{proposition}{explicit} \label{prop:explicit barc k}
Let $Z:=X\sqcup  \{x_1\}^{m_1}\sqcup\dots\sqcup  \{x_n\}^{m_n},$ 
where each $m_j\geq 0$ is the multiplicity of the extra copies of $x_j$ in $Z$. Then, for any degree $k$,
\begin{align*}
    \caB_{\Ver,k}(Z)&= \caB_{\Ver,k}(X) \sqcup 
        \bigsqcup_{\substack{1\leq p\leq k+1\\
        1\leq i_1<\dots<i_p\leq n}}
            \left\{  \diam(\{ x_{i_1}, x_{i_2},\dots, x_{i_p} \})
        \cdot(1,1)\right\}^{\mu_k(\vec{m}(I_p))}.
\end{align*}
In particular, the multiplicity of $\diam(\{ x_{j} \})
        \cdot(1,1)$ is ${m_j\choose k+1}$, for each $j$.
\end{restatable}

We examine the relationship between $\hatdb$ and $\db$, and obtain an interpretation of $\hatdb$ in terms of matchings of points in the barcodes. 
To compute $\db$, one looks for an optimal matching where points from a barcode can be matched to any points on the diagonal. However, in the computation of $\hatdb$, points are only allowed to be matched to verbose barcodes and a \emph{particular multiset} supported on the diagonal, where the choice of these diagonal points depends on the metric structure of the two underlying metric spaces. 

In degree 0, since the verbose barcode of any pullback space $Z$ of $X$ only differs from the verbose barcode of $X$ in multiple copies of the point $(0,0)$, the distance $\hatdb$ is indeed 
computing an optimal matching between concise barcodes which only allows bars to be matched to other bars or to the origin $(0,0)$ (see Figure \ref{fig:compute pb db}). 
Given that all degree-0 bars originate at $0$, we derive the following explicit formula for computing the distance $\hatdb$ for degree-0 (see Section \ref{subsec: hatdb 0} for the proof):

\begin{restatable}[Pullback bottleneck distance in degree $0$]{proposition}{hatdbfordegreezero}\label{prop:hatdb degree 0}
Let $X$ and $Y$ be two finite metric spaces such that $\card(X)=n\leq n'=\card(Y)$. Suppose the death time of finite-length degree-$0$ bars of $X$ and $Y$ are given by the sequences $a_1\geq\dots\geq  a_{n-1}$ and $b_1\geq \dots\geq b_{n'-1}$, respectively. Then,
\[\pbdbk{0}{X}{Y} 
=  \max\left\{ \max_{1\leq i\leq n-1}|a_i-b_i|, \max_{n\leq i\leq n'-1} b_i \right\}. \]
\end{restatable}

\begin{center}
\begin{tabular}{  c| c| c  }
&
\begin{tikzpicture}[scale=.7]
\node (a) at (0,0) {};
\node (c) at (2,0) {};
\filldraw (a) [color=blue] circle[radius=2pt];
\filldraw (c) [color=blue] circle[radius=2pt];
\draw [dashed] (a) --(c) node[pos=0.5,fill=white]{$a$};
\node at (-1,0) {$X_1:$};
\end{tikzpicture}
&
\begin{tikzpicture}[scale=.7]
\node (a') at (1,-2) {};
\filldraw (a') [color=red] circle[radius=2pt];
\node at (0,-2) {$X_2:$};
\end{tikzpicture} 
\\ 
\hline
$\caB_{\mathrm{Ver},0}=\caB_{\Con,0}$ 
& $\{(0,a),(0,\infty)\}$ 
& $\{(0,\infty)\}$ \\ 
\end{tabular}

\vspace{.5em}

\begin{tikzpicture}[scale=0.5]
    \begin{axis} [ 
    ticklabel style = {font=\large},
    axis lines=middle, 
    axis line style={black, thick,-latex},
    xlabel = {\Large birth},
    every axis x label/.style={
        at = {(ticklabel* cs:1)},
        anchor = west
    },
    ylabel = {\Large death},
    every axis y label/.style={
        at = {(ticklabel* cs:1)},
        anchor = south
    },
    ytick={0,1.6,2.1},
    yticklabels={\text{$(0,0)$},\text{$(0,a)$},\text{$(0,\infty)$}},
    xticklabels=none,
    xtick={0,2.1},
    xmin=0, xmax=2.3,
    ymin=0, ymax=2.3,
    axis equal image,
    yticklabel style={inner ysep=0cm, xshift=-.1em,yshift=0cm},
    ]
    \addplot [mark=none, thick] coordinates {(2.1,2.1) (0,2.1)};
    \addplot [mark=none, thick] coordinates {(2.1,2.1) (2.1,0)};
    \addplot[blue,mark=*] (0,1.6) circle (2pt);
    \addplot[blue,mark=*] (0,2.1) circle (2pt);
    \addplot[red,mark=*] (.05,2.1) circle (2pt);
    \node[mark=none] at (axis cs:0.75,1.3){\LARGE$\db = \frac{a}{2}$};
    \addplot [mark=none, dashed] coordinates {(0,1.6) (.8,.8)};
    \addplot [mark=none, blue] coordinates {(0,0) (.25,.25)};
    \addplot [mark=none, thick, red] coordinates {(.25,.25) (.5,.5)};
    \addplot [mark=none, blue] coordinates {(.5,.5) (.75,.75)};
    \addplot [mark=none, thick, red] coordinates {(.75,.75) (1,1)};
    \addplot [mark=none, blue] coordinates {(1,1) (1.25,1.25)};
    \addplot [mark=none, thick, red] coordinates {(1.25,1.25) (1.5,1.5)};
    \addplot [mark=none, blue] coordinates {(1.5,1.5) (1.75,1.75)};
    \addplot [mark=none, thick, red] coordinates {(1.75,1.75) (2,2)};
    \addplot [mark=none, blue] coordinates {(2,2) (2.1,2.1)};
    \end{axis}
\end{tikzpicture}
\begin{tikzpicture}[scale=0.5]
    \begin{axis} [ 
    ticklabel style = {font=\large},
    axis lines=middle, 
    axis line style={black, thick,-latex},
    xlabel = {\Large birth},
    every axis x label/.style={
        at = {(ticklabel* cs:1)},
        anchor = west
    },
    ylabel = {\Large death},
    every axis y label/.style={
        at = {(ticklabel* cs:1)},
        anchor = south
    },
    ytick={0,1.6,2.1},
    yticklabels={\text{$(0,0)$},\text{$(0,a)$},\text{$(0,\infty)$}},
    xticklabels=none,
    xtick={0,2.1},
    xmin=0, xmax=2.3,
    ymin=0, ymax=2.3,
    axis equal image,
    yticklabel style={inner ysep=0cm, xshift=-.1em,yshift=0cm},
    ]
    \addplot [mark=none, thick] coordinates {(2.1,2.1) (0,2.1)};
    \addplot [mark=none, thick] coordinates {(2.1,2.1) (2.1,0)};
    \addplot[blue,mark=*] (0,1.6) circle (2pt);
    \addplot[blue,mark=*] (0,2.1) circle (2pt);
    \addplot[red,mark=*] (.05,2.1) circle (2pt);
    \node[mark=none] at (axis cs:0.65,1.2){\LARGE$\hatdb = a$};
    \addplot [mark=none, thin] coordinates {(0,0) (2.1,2.1)};
    \addplot[blue,mark=*] (0,0) circle (1pt);
    \addplot[red,mark=*] (.05,0) circle (1pt);
    \end{axis}
    \draw [dashed] (0.05,0) to [out=60,in=300] (0,4);
    \end{tikzpicture}
    \captionof{figure}{\emph{Top}: $X_1$ a two-point space, $X_2$ the one-point space, and their $0$-th verbose (or concise) 
    barcode. \emph{Bottom}: visualization of $\db$ and $\hatdb$, where in both figures the point $(0,\infty)$ is matched with $(0,\infty)$ and the distance between points is measured using the max norm.
    }
    \label{fig:compute pb db}
    \end{center}

For positive degrees, the situation becomes more complicated because, in addition to the point $(0,0)$, other choices of diagonal points need to be considered, as evidenced by the formula for pullback barcodes in Proposition \ref{prop:pullback-barcode}. 
Although we cannot obtain a formula as simple as that for the degree-$0$ case, the pullback distances can be simplified utilizing pullback vectors. See Section \ref{subsec:multi vector} for details, where we also analyze the time complexity for brute-force algorithms for pullback distances. 

In Section \ref{subsub:non triangle}, we present an important example involving certain five-point ultra-metric spaces. This example illustrates both the strictness of some inequalities in Theorem \ref{thm:all hatdi and hatdb} and the failure of the triangle inequality of the pullback interleaving distances (see Corollary \ref{cor:triangle fails}).

\subsection{Organization of the Paper}

In Section \ref{sec:fcc}, we recall the notions of filtered chain complexes, verbose barcodes and concise barcodes. For the case of Vietoris-Rips FCCs, we characterize verbose barcodes of ultra-metric spaces, cf. Theorem \ref{thm:verbose of ultra}, and study the relation between isometry of metric spaces as well as filtered chain isomorphism and filtered homotopy equivalence of FCCs in Section \ref{sec:decomp of fcc}. 
In Section \ref{sec:iso}, we study the interleaving distance between FCCs and the matching distance between verbose barcodes, and we establish an isometry theorem for these two notions of distances, i.e. Theorem \ref{thm:dm=di}.
Starting from Section \ref{sec:VR FCC}, we specifically focus on the case of Vietoris-Rips FCCs and define $\hatdb$ and $\hatdi$ via tripods of metric spaces in Section \ref{sec:pb di and pb db}. 
In addition, we establish a Gromov-Hausdorff stability for them by proving Theorem \ref{thm:hatdb-dgh stability} in Section \ref{sec:pb stab}. Examples are provided in Section \ref{sec:ex for strict} to demonstrate that both inequalities in Theorem \ref{thm:hatdb-dgh stability} can be strict. 
In Section \ref{sec:variation of pb distances}, we introduce two variants of the pullback interleaving/bottleneck distance.
In Section \ref{sec:pb barcode}, we establish relations between the verbose barcodes of the pullback of a metric space and those of the original space, by proving Proposition \ref{prop:pullback-barcode} and Proposition \ref{prop:explicit barc k}.
In Section \ref{subsec:computation}, we study the interpretation and computability of the pullback distances. We prove Proposition \ref{prop:hatdb degree 0} in Section \ref{subsec: hatdb 0}.

\paragraph*{Acknowledgements}
FM and LZ were partially supported by the NSF through grants RI-1901360, CCF-1740761, and CCF-1526513, and DMS-1723003. The work in this paper is part of the second author's PhD dissertation.
\revision{We thank Jeong-hwi Joe for insightful feedback and for suggesting clarifications that helped improve the exposition, especially Lemma~\ref{lem:non-isometric chain complexes}.
We also} thank anonymous reviewers for their valuable advice on improving the presentation and, in particular, their suggestions for correcting some errors in the appendix and simplifying the proofs of Theorem 5 and Proposition 6.2.

\section{Preliminaries}
\label{sec:preliminaries}

In this section, we recall some backgrounds on (pseudo-)metric spaces, Vietoris-Rips complexes and the Gromov-Hausdorff distance. 

Given a set $X$, a \emph{metric} $d_X$ on $X$ is a function $d_X:X\times X\to [0,+\infty)$ such that for any \revision{$x,x',x''\in X$}, the following axioms hold: \label{para:metric}
    \begin{itemize}
        \item $d_X\revision{(x,x')}\geq 0$ and $d_X\revision{(x,x')}=0$ if and only if \revision{$x=x'$};
        \item (Symmetry) $d_X\revision{(x,x')}=d_X\revision{(x',x)}$;
        \item (Triangle inequality) $d_X\revision{(x,x'')}\leq d_X\revision{(x,x')}+d_X\revision{(x',x'')}$.
    \end{itemize}
A \emph{metric space} is a pair $(X,d_X)$ where $X$ is a set and $d_X$ is a metric on $X$. 

An \emph{ultra-metric} $d_X$ on $X$ is a metric $d_X$ on $X$ satisfying the strong triangle inequality: \revision{$d_X(x,x'')\leq \max\{d_X(x,x'),d_X(x',x'')\}$ for all $x,x',x''\in X$}.
A \emph{pseudo-metric} $d_X$ on $X$ is a function $d_X:X\times X\to [0,+\infty)$ satisfying the axioms for a metric, except that in the first axiom different points are allowed to have distance $0$. 
Given two pseudo-metric spaces $(X,d_X)$ and $(Y,d_Y)$, a map $f:(X,d_X)\to (Y,d_Y)$ is said to be \emph{distance-preserving} if $d_X(x,x')=d_Y(f(x),f(x'))$ for all $x,x'\in X$. 
A bijective distance-preserving map is called an \emph{isometry}. 
Two pseudo-metric spaces $X$ and $Y$ are \emph{isometric}, denoted $X\cong Y$, if there exists an isometry between them.
\label{para:isometry}

Given a finite pseudo-metric space $(X,d_X)$ and $\epsilon\geq 0$, the $\epsilon$-\emph{Vietoris–Rips complex} $\VR_{\epsilon }(X)$ is the simplicial complex with vertex set $X$, where 
\begin{center}
    a finite subset $\sigma\subset X$ is a simplex of $\VR_{\epsilon }(X)$ $\iff$ $\diam(\sigma)\leq\epsilon$.
\end{center} 
Here $\diam(\cdot)$ denotes the diameter of a subset of $X$. Let \label{para:VR(X)}
\[\Full(X): = \VR_{\diam(X)}(X),\] 
which is the \emph{full complex} on $X$.
For each $k\in\bbZ_{\geq0}$, we denote by $\opC_{k}\!\left(\Full\left(X\right)\right)$ the free $\field$-vector space generated by $k$-simplices in $\Full(X)$, and let $\fccVR{X}$ be the free simplicial chain complex induced by $\Full(X)$ over coefficients in $\field$, with the standard simplicial boundary operator $\partial^X$.  
Notice that up to homotopy equivalence the simplicial complex $\Full(X)$ only depends on the cardinality of $X$, so does the chain complex $(\fccVR{X},\partial^X)$. 

 The \emph{Hausdorff distance} between two subspaces $X$ and $Y$ of a metric space $Z$ is \label{para:dh}
\[d_\mathrm{H}^Z(X,Y):=\inf\left\{r>0: X\subseteq \bar{B}(Y,r)\text{ and }  Y\subseteq \bar{B}(X,r)\right \}.\]
For metric spaces $(X,d_X)$ and $(Y,d_Y)$, recall from \cite{edwards1975structure,gromov2007metric} that the \emph{Gromov-Hausdorff distance} between them is the infimum of $r>0$ for which there exist a metric space $Z$ and two distance preserving maps $\psi_X:X\to Z$ and $\psi_Y:Y\to Z$ such that $\dhaus^Z(\psi_X(X),\psi_Y(Y))<r$, i.e., \label{para:dgh}
$$\dgh(X,Y):=\inf_{Z,\psi_X,\psi_Y}d_H^Z(\psi_X(X),\psi_Y(Y)).$$

\paragraph{Reformulation of $\dgh$ using maps.}
The \emph{distortion} of a map $\varphi:X\rightarrow Y$ is defined to be \label{para:dis(f)}
$$\dis(\varphi):=\sup_{x,x'\in X}|d_X(x,x')-d_Y(\varphi(x),\varphi(x'))|.$$
For maps $\varphi:X\rightarrow Y$ and $\psi:Y\rightarrow X$, their \emph{co-distortion} is defined to be $$\codis(\varphi,\psi):=\sup_{x\in X,y\in Y}|d_X(x,\psi(y))-d_Y(\varphi(x),y)|.$$
It follows from \cite[Theorem 2.1]{kalton1999distances} that 
\begin{equation}\label{eq:dgh-maps}
    \dgh(X,Y)=\inf_{\substack{\varphi:X\rightarrow Y\\ \psi :Y\rightarrow X}}\tfrac{1}{2}\max\{\dis(\varphi),\dis(\psi),\codis(\varphi,\psi)\}.
\end{equation}

\paragraph{Reformulation of $\dgh$ using correspondences.}
A \emph{correspondence} between $X$ and $Y$ is a subset $R$ of $X\times Y$ such that for any $x\in X$ there exists at least one $y\in Y$ such that $(x,y)\in R$ and for any $y\in Y$ there exists at least one $x\in X$ such that $(x,y)\in R$. The \emph{distortion} of a correspondence $R$ between $X$ and $Y$ is defined to be:	\label{para:correspondence}
$$\dis(R):=\sup_{(x,y),(x',y')\in R}\left|  d_X(x,x')-d_Y(y,y')\right| .$$
Let $\mathfrak{R}(X,Y)$ denote the collection of all correspondences between $X$ and $Y$. It follows from \cite[Theorem 7.3.25]{burago2001course} that
\begin{equation}\label{eq:dgh-tripod}
    \dgh(X,Y)=\tfrac{1}{2}\inf_{R\in \mathfrak{R}(X,Y)}\dis(R).
\end{equation}

\paragraph{Reformulation of $\dgh$ using tripods.}
A \emph{parametrization} of a set $X$ is a set $Z$ together with a surjective map $\phi:Z\twoheadrightarrow X$. A \emph{tripod} between two sets $X$ and $Y$ is a pair of surjections from another set $Z$ to $X$ and $Y$ respectively, expressed by the diagram (cf. \cite{memoli2017distance})\label{para:tripod} $$X\xtwoheadleftarrow{\phi_X}Z\xtwoheadrightarrow{\phi_Y}Y.$$
The \emph{distortion} of a tripod $(Z,\phi_X,\phi_Y)$ between $X$ and $Y$ is defined to be:	
\label{para:dis of tripod}
$$\dis((Z,\phi_X,\phi_Y)):=\sup_{z,z'\in Z}\left|  d_X(\phi_X(z),\phi_X(z'))-d_Y(\phi_Y(z),\phi_Y(z'))\right| .$$
It follows from \cite[Section 7.3.3]{burago2001course} that
\begin{equation}\label{eq:dgh using tripods}
    \dgh(X,Y)=
    \tfrac{1}{2}\inf_{X\xtwoheadleftarrow{\phi_X}Z\xtwoheadrightarrow{\phi_Y}Y}\dis((Z,\phi_X,\phi_Y)).
\end{equation}

\begin{remark}\label{rmk:finite tripod}
    Notice that for finite metric spaces $X$ and $Y$, as given by Equation (\ref{eq:dgh using tripods}), the Gromov-Hausdorff distance $\dgh(X,Y)$ can be computed by only considering finite tripods. To see this, consider a possibly infinite tripod $(Z,\phi_X,\phi_Y)$. Define $Z'=\{(\phi_X(z),\phi_Y(z))\mid z\in Z\}\subset X\times Y$, which is finite given that both $X$ and $Y$ are finite. It is straightforward to verify that $\dis(Z')=\dis(Z).$ Therefore, for computing $\dgh(X,Y)$ via Equation (\ref{eq:dgh using tripods}), any tripod between $X$ and $Y$ can be replaced by a finite tripod whose underlying set has cardinality no greater than $\card(X) \cdot \card(Y)$.
\end{remark}


\section{Filtered Chain Complexes (FCCs)}\label{sec:fcc}
In this section, we recall from \cite{usher2016persistent} the notion of \emph{filtered chain complexes} (in short, FCCs) together with the construction of verbose barcodes and concise barcodes for FCCs. 

\subsection{Filtered Chain Complexes} \label{subsec:fcc}
Let $\field$ be a fixed field. 
A \textbf{non-Archimedean normed vector space} over $\field$ is any pair $(C,\ell)$ where $C$ is a finite-dimensional vector space over $\field$ endowed with a \textbf{filtration function} $\ell:C\to \bbR\sqcup\left\{  -\infty\right\}   $ defined as a map satisfying the following axioms:
\begin{enumerate}[label=(\roman*)]
    \item $\ell(x)=-\infty$ if and only if $x=0$;
    \item For any $0\neq \lambda\in\field$ and $x\in C$, $\ell(\lambda x)=\ell(x);$
    \item For any $x$ and $y$ in $C$, $\ell(x+y)\leq \max\left\{  \ell(x),\ell(y)\right\}  .$
\end{enumerate}
A finite collection $(x_1,\dots,x_r)$ of elements of $C$ is said to be \textbf{orthogonal} if, for all $\lambda_1,\dots,\lambda_r$ in $\field$,
$$\ell\left(  \sum_{i=1}^r \lambda_i x_i\right)=\max_{\lambda_i\neq 0}\ell(x_i) .$$
An \textbf{orthogonalizable $\field$-space} $(C,\ell)$ is a finite-dimensional non-Archimedean normed vector space over $\field$ such that there exists an orthogonal basis for $C$. 
Two subspaces are $V,W$ of $C$ are said to be \textbf{orthogonal} if for all $x\in V$ and $y\in W$, $\ell(x+y)=\max\{\ell(x),\ell(y)\}$.

\revision{We define the norm of a filtration function $\ell : C \to \mathbb{R} \sqcup \{-\infty\}$ as
\[
\|\ell\|_\infty := \sup_{x \in C \setminus \{0\}} |\ell(x)|,
\]
where we exclude the zero element to avoid the ``singularity"  $\ell(0) = -\infty$.}

Below, we introduce a couple of lemmas regarding filtration functions and the orthogonality of subspaces, which will be referenced in later sections.

\begin{lemma} \label{rmk:property of filtration function}
For any $x,y\in C$ such that $\ell(x)\neq \ell(y)$, we have $\ell(x+y)=\max\{\ell(x),\ell(y)\}$. 
\end{lemma}
\begin{proof}
    Given that $\ell(x)=\ell(-x)$, it follows that $\ell(y)=\ell((y+x)+(-x))\leq \max\{\ell(y+x),\ell(x)\}.$ Therefore, if $\ell(y)>\ell(x)$ it must be that $\ell(y)\leq \ell(y+x)$. This implies $\max\{\ell(x),\ell(y)\}\leq \ell(x+y)$. 
    A similar argument applies if $\ell(y)<\ell(x)$.
\end{proof}

\begin{lemma}[Lemma 2.9, \cite{usher2016persistent}] \label{lem:orthogonal}
Let $(C,\ell)$ be a non-Archimedean normed vector space over $\field$. Then,
\begin{itemize}
    \item  For subspaces $U,V$ and $W$ of $C$, if $U$ and $V$ are orthogonal and $U\oplus V$ and $W$ are orthogonal, then $U$ and $V\oplus W$ are orthogonal.
    \item If $U$ and $V$ are orthogonal subspaces of $C$, and if $(u_1,\dots,u_r)$ and $(v_1,\dots,v_s)$ are orthogonal collections of elements of $U$ and $V$, respectively, then $(u_1,\dots,u_r,v_1,\dots,v_s)$ is orthogonal in $U\oplus V.$
\end{itemize}
\end{lemma}

\begin{definition}[Filtered chain complex]
\label{def:fcc}
A \textbf{filtered chain complex (FCC)} over $\field$ is a finite-dimensional chain complex $(C_*=\oplus_{k\in\bbZ}C_k,\partial_C)$ over $\field$ together with a function $\ell_C:C_*\to \bbR\sqcup\left\{  -\infty\right\}   $ such that each $(C_k,\ell_C|_{C_k})$ is an orthogonalizable $\field$-space, and $\ell_C\circ \partial_C\leq\ell_C$.

A \textbf{morphism of FCCs} from $(C_*,\partial_C,\ell_C)$ to $(D_*,\partial_D,\ell_D)$ is a chain map $\Phi_*:C_*\to D_*$ that is \textbf{filtration preserving}, i.e. $\ell_D\circ \Phi_*\leq\ell_C$.
\end{definition}

\begin{example} [Vietoris-Rips FCC]
\label{ex:VR FCC}
For a finite pseudo-metric space $(X,d_X)$, we denote by $(\fccVR{X},\partial^X)$ the chain complex of the simplicial complex $\Full(X)$ (see Section \ref{sec:preliminaries}). 
Define a filtration function $\ell^{X}:\fccVR{X}\to \bbR\sqcup\left\{  -\infty\right\}   $ by 
\[\ell^{X}\left(  \sum_{i=1}^r \lambda_i \sigma_i\right): =\max_{\lambda_i\neq 0} \left\{  \diam(\sigma_i)\right\}   ,\] 
where the $\sigma_i$ are simplices, and $\ell^{X}(0):=-\infty$. Then $\left( \fccVR{X},\partial^X,\ell^{X} \right) $ is an FCC, and the set of simplices is an orthogonal basis for it.
\end{example}

\begin{definition}[Filtered homotopy equivalent
] 
\label{def:f.h.e.}
Two chain maps $\Phi_*,\Psi_*:C_*\to D_*$ are called \textbf{filtered chain homotopic} if they are filtration preserving and there exists a filtration preserving chain map $K:C_*\to D_{*+1}$ such that $\Phi_*-\Psi_*=\partial_C K+K\partial_D$.

We say that $(C_*,\partial_C,\ell_C)$ and $(D_*,\partial_D,\ell_D)$ are \textbf{filtered homotopy equivalent (or f.h.e.)} if there exist filtration preserving chain maps $\Phi_*:C_*\to D_*$ and $\Psi_*:D_*\to C_*$ such that $\Psi_*\circ\Phi_*$ is filtered chain homotopic to the identity $\Id_C$ while $\Phi_*\circ\Psi_*$ is filtered chain homotopic to $\Id_D$.
\end{definition}

\begin{definition}[Filtered chain isomorphism
]
\label{def:f.c.i.}
Two FCCs $(C_*,\partial_C,\ell_C)$ and $(D_*,\partial_D,\ell_D)$ are said to be \textbf{filtered chain isomorphic (or f.c.i.)} if there exists a chain isomorphism 
\begin{center}
    $\Phi_*:(C_*,\partial_C)\xrightarrow{\cong} (D_*,\partial_D)$ such that $\ell_D\circ \Phi_*=\ell_C$,
\end{center} denoted by $(C_*,\partial_C,\ell_C)\cong(D_*,\partial_D,\ell_D)$, or $C_*\cong D_*$ for simplicity.
\end{definition}

\begin{remark}\label{rmk:iso_category}
Let $\fcc$ denote the category whose objects are FCCs and morphisms are given in Definition \ref{def:fcc}. Then the filtered chain isomorphism relation coincides with the isomorphism in the category $\fcc$.
\end{remark}

\paragraph*{Dual of a FCC.} \label{para:dual}
Given a non-Archimedean normed vector space $(C,\ell)$, the dual space $C^*$ becomes a non-Archimedean normed vector space if equipped with the \emph{dual filtration function} $\ell^*:C^*\to\bbR\sqcup\left\{ - \infty\right\}   $ given by
\[\ell^*(\phi):=
\sup\{-\ell(x)\mid x\in C,\phi(x)\neq 0\}.\]

By \cite[Proposition 2.20]{usher2016persistent}, if $(y_1,\dots,y_n)$ is an orthogonal ordered basis for $(C,\ell)$, then the dual basis $(y_1^*,\dots,y_n^*)$ is an orthogonal ordered basis for $(C^*,\ell^*)$ such that $\ell^*(y_i^*)=-\ell(y_i)$ for $i=1,\dots,m$. 

\subsection{Verbose  and Concise Barcodes} \label{sec:barcodes}
In this section, we recall the definition of verbose barcode and concise barcode from \cite{usher2016persistent}. 

\begin{definition}[Singular value decomposition
] \label{def:s.v.d.}
Let $(C,\ell_C)$ and $(D,\ell_D)$ be two orthogonalizable $\field$-spaces, and let $A:C\to D$ be a linear map with rank $r$. A \textbf{(unsorted) singular value decomposition of $A$} is a choice of orthogonal ordered bases $(y_1,\dots,y_n)$ for $C$ and $(x_1,\dots,x_m)$ for $D$ such that:
\begin{itemize}
    \item $(y_{r+1},\dots,y_n)$ is an orthogonal ordered basis for $\kernel  A$;
    \item $(x_{1},\dots,x_r)$ is an orthogonal ordered basis for $\im A$;
    \item $A y_i=x_i$ for $i=1,\dots,r$.
\end{itemize}
If $\left( (y_1,\dots,y_n),(x_1,\dots,x_m)\right) $ is such that $\ell_C(y_1)-\ell_C(x_1)\geq \dots\geq \ell_C(y_r)-\ell_C(x_r)$, we call $((y_1,\dots,y_n),$ $(x_1,\dots,x_m))$ a \textbf{sorted singular value decomposition}.
\end{definition}

The existence of a singular value decomposition for linear maps between finite-dimensional orthogonalizable $\field$-spaces is guaranteed by \cite[Theorem 3.4]{usher2016persistent}. 

\begin{definition}[Verbose barcode and concise barcode
] \label{def:barcode}
Let $(C_*,\partial_C,\ell_C)$ be an FCC over $\field$ and for each $k\in \bbZ$ write $\partial_k=\partial_C|_{C_k}$. Given any $k\in \bbZ$ choose a singular value decomposition $((y_1,\dots,y_n),(x_1,\dots,x_m))$ for the $\field$-linear map $\partial_{k+1}:C_{k+1}\to\kernel \partial_k$ and let $r$ denote the rank of $\partial_{k+1}$. Then the degree-$k$ \textbf{
verbose barcode} of $(C_*,\partial_C,\ell_C)$ is the multiset $\caB_{\Ver,k}$ of elements of $\bbR\times (\bbR\sqcup\{\infty\})$ consisting of 
\begin{enumerate}[label=(\roman*)]
    \item a pair $(\ell(x_i),\ell(y_i))$ for each $i =1,\dots,r=\rank(\partial_{k+1});$ and
   \item a pair $(\ell(x_i),\infty)$ for each $i =r+1,\dots,m=\dim(\kernel \partial_k).$
\end{enumerate}
These pairs are also called \textbf{bars}.
The first (resp. second) entry in a bar is called the \textbf{birth time} (resp. \textbf{death time}) of that bar. The length of a bar, i.e., $\ell(y_i)-\ell(x_i)\geq 0$ or $\infty$, is called its \textbf{life time} (or also called \textbf{persistence}). The \textbf{concise barcode} of $(C_*,\partial_C,\ell_C)$ is the submultiset of the verbose barcode consisting of those elements where $\ell(y_i)-\ell(x_i)> 0$. 
\end{definition}

It is shown in \cite[Theorem 7.1]{usher2016persistent} that each degree-$k$ verbose barcode is independent of the choice of the singular value decomposition of $\partial_{k+1}$. 

\begin{remark}\label{rmk:barcode}
In the case of Vietoris-Rips FCCs (see Example \ref{ex:VR FCC}), the concise barcode is equivalent to the classical persistent homology barcode \cite[page 6]{usher2016persistent}. 
\end{remark}

\begin{remark} Let $X$ be a finite metric space. The degree-$0$ verbose barcode $\caB_{\Ver,0}$ and the degree-$0$ concise barcode $\caB_{\Con,0}$ of the Vietoris-Rips FCC $(\fccVR{X},\partial_X,\ell^X)$ are the same. Notice that this is not necessarily true for pseudo-metric spaces, in which case verbose barcode may contain several copies of $(0,0)$.
\end{remark}

The following theorem states that one can construct an orthogonal ordered basis for the target space when given any orthogonal ordered basis for the source space. This theorem will be applied in later sections to prove the stability of verbose barcodes; see Section \ref{sec:dm<=di}.

\begin{theorem}[Theorem 3.5, \cite{usher2016persistent}] \label{thm:change basis}
Let $(C,\ell_C)$ and $(D,\ell_D)$ be two orthogonalizable spaces, let $A:(C,\ell_C)\to (D,\ell_D)$ be a linear map and let $(y_1,\dots,y_n)$ be an orthogonal ordered basis for $C$. Then one may algorithmically construct an orthogonal ordered basis $(y_1',\dots,y_n')$ for $C$ such that
\begin{itemize}
    \item If $A y_i=0$, then $y_i'=y_i$;
    \item $\ell_C(y_i')=\ell_C(y_i)$ and $\ell_D(A y_i')\leq \ell(A y_i)$, for any $i$;
    \item The set $\left\{  A y_i':A y_i'\neq 0\right\}   $ is orthogonal in $D$.
\end{itemize}
\end{theorem}

\begin{example}[Verbose barcodes of Vietoris-Rips FCCs] 
\label{ex:card of VR verbose barcodes}
Recall from Example \ref{ex:VR FCC} the notion of Vietoris-Rips FCC. Let $X$ be a finite pseudo-metric space of $n$ points. Note that $\Full(X)$ has trivial homology groups $H_{k}(\Full(X))=0$ for each $k\geq 1$, i.e. $\kernel \partial_k=\im\partial_{k+1}$. Thus, the following sequence is exact at each degree except for $0$, where $C_k:=\opC_{k}\!\left(\Full\left(X\right)\right)$ for $k\geq 0$:
\begin{center}
	\begin{tikzcd}
	\ar[r, "\partial_{n+1}=0"]&C_n=0 \ar[r, "\partial_{n}=0"]
	& C_{n-1} \ar[r, "\partial_{n-1}"]
	&\cdots \ar[r]
	& C_1 \ar[r, "\partial_{1}"]
	& C_0 \ar[r, "\partial_{0}"]
	&0
	\end{tikzcd}
\end{center} 
The cardinality of $k$-verbose barcodes (with multiplicity) of 
$(\fccVR{X},\partial_X,\ell^X)$ is
\[\card(\caB_{\Ver,k}(X))=\dim(\kernel \partial_k)= 
\begin{cases}
n,&\mbox{$k= 0$,}\\ 
{n-1 \choose k+1}, &\mbox{for $1\leq k\leq n-2$,}\\
0,&\mbox{for $k\geq n-1$.}
	\end{cases}\]
Indeed, because $\partial_0=0$, we have $\card(\caB_{\Ver,0}(X))=\dim (C_0)=n$. For $1\leq k\leq n-2$, we prove by induction that $\card(\caB_{\Ver,k}(X))= {n-1 \choose k+1}$. First, when $k=1$ we have
\[\dim(\im \partial_{1})=\dim(\kernel \partial_0)-\dim (\rmH_0)
=\dim(C_0)-\dim (\rmH_0)=n-1,\]
and thus, \[ \card(\caB_{\Ver,1}(X))=\dim(C_1)-\dim(\im \partial_{1})={n \choose 2}- (n-1)={n-1 \choose 2}.\]
Suppose that $\card(\caB_{\Ver,k-1}(X))=\dim(\kernel \partial_{k-1})= {n-1 \choose k}$. Then, for degree $k$ we have
\[{n \choose k+1}=\dim(C_k)=\dim(\kernel \partial_k)+\dim(\im \partial_{k})=\card(\caB_{\Ver,k}(X))+\dim(\kernel \partial_{k-1}),\]
implying that
\[\card(\caB_{\Ver,k}(X))={n \choose k+1}-\card(\caB_{\Ver,k-1}(X)) =  {n \choose k+1}-{n-1 \choose k}={n-1 \choose k+1}.\]
\end{example}

\subsubsection{Verbose Barcodes of Ultra-Metric Spaces  }
\label{subsub:ultra}
In Theorem \ref{thm:verbose of ultra} we provide a complete characterization of the verbose barcodes of a finite ultra-metric space $(X,u_X)$. The statement of the theorem uses a special ordering of the points described in \cite[Proposition 4.19]{memoli2019primer} which we now recall.

Let $(X,u_X)$ be an ultra-metric space of $n$ points. We order the points in $X$ following the procedure described in the proof of \cite[Proposition 4.19]{memoli2019primer}. In order to produce one such ordering  $x_1<x_2<\dots<x_n$ of $X$: \label{para:order in ultra}
\begin{itemize}
    \item Pick an arbitrary point $x_1\in X$;
    \item Find $x_2\in X-\left\{  x_1\right\}   $ such that $u_X(x_1,x_2)=\min _{x\in X-\left\{  x_1\right\}   } u_X(x_1,x)$;
    \item Find $x_3\in X-\left\{  x_1,x_2\right\}   $ such that $u_X(x_2,x_3)=\min _{x\in X-\left\{  x_1,x_2\right\}   } u_X(x_2,x)$;\\
    $\dots\dots$
    \item Find $x_i\in X-\left\{  x_1,\dots,x_{i-1}\right\}   $ such that $u_X(x_{i-1},x_i)=\min _{x\in X-\left\{  x_1,\dots,x_{i-1}\right\}   } u_X(x_{i-1},x)$;\\
    $\dots\dots$
    \item Finish when $(n-1)$ points are found, and label the remaining point in $X$ as $x_{n}$.   
\end{itemize}

Note that this ordering is not unique.  We will refer to any such order as a \emph{self-consistent order} on $X$.\footnote{The key property of any such order is that it permits immediately reading off the usual degree-0 VR barcodes from the ultra-metric space structure; see Proposition \ref{prop:ordered 0-barcode} for details.}
\begin{framed}
For the rest of this subsection, we assume that, given an ultra-metric space $(X,u_X)$, the finite set $X$ consists of points $x_1< \cdots < x_n$  ordered as above. 
\end{framed}
\begin{restatable}[Verbose barcodes of ultra-metric spaces]{theorem}{ultrabarc}\label{thm:verbose of ultra}
For any degree $k\geq 1$, we have
\begin{align}
    \caB_{\Ver,k}(X)=&\bigsqcup_{2\leq i_1<i_2<\dots<i_{k+1}\leq n}\left\{ u_X\left(x_{i_1-1},x_{i_{k+1}}\right)\cdot (1,1) \right\} \notag \\
    =&\bigsqcup_{j-i=k+1}^{n-1}\,\bigsqcup_{i=1}^{n-k-1} \left\{ u_X(x_{i},x_{j})\cdot (1,1) \right\}^{{j-i-2 \choose k-1}}. \label{eq:verbose of ultra}
\end{align}
\end{restatable}
We represent the multiplicity of points in $\caB_{\Ver,k}(X)$ via a matrix whose $(i,j)$-th element is the multiplicity of the point $u_X(x_{i},x_{j})\cdot (1,1)$. Then, Equation (\ref{eq:verbose of ultra}) can be expressed as follows: for any $k\geq 1$, the non-zero part of the multiplicity matrix is
\[
\begin{blockarray}{cccccccc}
 &x_1 & \dots & x_{ k+1} & x_{ k+2} & x_{ k+3} & \dots & x_n \\
\begin{block}{c(ccccccc)}
 x_1&   &   &   & { k-1 \choose  k-1} & { k \choose  k-1} & \dots & {n-3\choose  k-1} \\
 x_2 &   &  &   &   & { k-1 \choose  k-1} & \dots & {n-4\choose  k-1} \\
 \dots &   &   &   &   &   & \dots &\dots \\
 x_{n- k-1}&   &  &   &   &   &   &{ k-1 \choose  k-1} \\
 \dots&   &  &   &   &   &   &  \\
 x_n&   &   &   &   &   & &   \\
\end{block}
\end{blockarray}.
 \]

It can be derived from \cite[Corollary 2.13]{lim2021some} that ultra-metric spaces only have non-trivial concise barcodes in degree $0$.
Furthermore, the $0$-th barcode of a finite ultra-metric space is given by the following proposition.
\begin{proposition}[{\cite[Proposition 4.19]{memoli2019primer}}] \label{prop:ordered 0-barcode}  The degree-$0$ verbose (or concise) barcode for the Vietoris-Rips FCC of $(X,u_X)$ is 
\[\left\{  (0,u_X(x_i,x_{i+1})):i=1,\dots,n-1\right\}   \sqcup \left\{  (0,\infty)\right\}  .\]
\end{proposition}

To prove Theorem \ref{thm:verbose of ultra}, we first show the following simple lemma.

\begin{lemma}\label{lem:prop of ordering}
The following hold: 
\begin{itemize} 
    \item[(1)] For any $i<j$, 
$u_X(x_i,x_j)= \max\left\{  u_X(x_i,x_{i+1}),u_X(x_{i+1},x_j)\right\} = \max_{i\leq l\leq j-1}\left\{  u_X(x_l,x_{l+1})\right\}.$
\item [(2)] For any $i_1<i_2<\dots<i_k$, $\diam\left(\left\{x_{i_1},x_{i_2},\dots,x_{i_{k}}\right\}\right)=u_X(x_{i_1},x_{i_{k}})$.
\end{itemize}
\end{lemma}

\begin{proof}
For the first equality of Part (1), the inequality `$\leq$' is true because $u_X$ is an ultra-metric. It remains to show `$\geq$'. 
Recall that  $x_{i+1}\in X-\left\{  x_1,\dots,x_{i}\right\}   $ is such that $u_X(x_{i},x_{i+1})=\min _{x\in X-\left\{  x_1,\dots,x_{i}\right\}   } u_X(x_{i},x)$. 
Since $j>i$, we have $x_j\in X-\left\{  x_1,\dots,x_i\right\}   $, and thus
\[ u_X(x_i,x_{i+1})\leq u_X(x_i,x_j).\]
Since $u_X$ is an ultra-metric, it follows from the above inequality that
\[u_X(x_{i+1},x_j)\leq \max\left\{  u_X(x_i,x_{i+1}), u_X(x_i,x_j)\right\} =u_X(x_i,x_j).\]
Therefore, we have $\max\left\{  u_X(x_i,x_{i+1}), u_X(x_{i+1},x_j)\right\}  = u_X(x_i,x_j)$.

The equality $u_X(x_i,x_j)= \max_{i\leq l\leq j-1}\left\{  u_X(x_l,x_{l+1})\right\}$ can be shown by induction on $j-i.$

Part (2) follows directly from the second equality of Part (1). Indeed, for any $i_1<i_2<\dots<i_k$, applying Part (1) for each pair $i_{l'}<i_{l'+1}$, we obtain 
\[\diam\left(\left\{x_{i_1},x_{i_2},\dots,x_{i_{k}}\right\}\right)= \max_{i_1\leq l\leq i_k-1}\left\{  u_X(x_l,x_{l+1})\right\} =u_X(x_{i_1},x_{i_k}).\qedhere\]
\end{proof}

\begin{remark}
    The ordered multiset $\{u_X(x_i,x_{i+1})\}_{i=1}^{n-1}$ consists of the death times of finite-length bars in degree $0$.
    One immediate consequence of Lemma \ref{lem:prop of ordering} is that one can recover the ultra-metric $u_X$ from $\{u_X(x_i,x_{i+1})\}_{i=1}^{n-1}$.
    Let $\tilde{u}_X:X\times X\to \bbR$ be defined as:
    \[\tilde{u}_X(x_i,x_j)=\begin{cases}
    0,&\mbox{$i=j$}\\
    \max_{i\leq l\leq j-1}\left\{  u_X(x_l,x_{l+1})\right\} ,&\mbox{$i<j $}\\
    \tilde{u}_X(x_j,x_i)  ,&\mbox{$i>j $}.
    \end{cases} \]
    Then, $\tilde{u}_X=u_X$.
\end{remark}

We now prove Theorem \ref{thm:verbose of ultra}.

\begin{proof}[Proof of Theorem \ref{thm:verbose of ultra}] 
Fix a degree $k\geq 1$. For notational simplicity, let $\partial:=\partial_{k+1}^X$ and $\ell:=\ell^X.$ 
We use $[\cdot]$ to denote simplices which are ordered lists of vertices, and we call the first vertex appearing in a simplex its \emph{leading vertex}. Here the order on the vertices is the one described above Theorem \ref{thm:verbose of ultra}. 

For a $k$-simplex $\gamma=[\gamma_0,\dots,\gamma_k]$, we denote its $j$-th face by $\face_j(\gamma)$ for $j=0,\dots,k$. In other words, $\face_j(\gamma)$ is a $(k-1)$-simplex obtained by removing the $j$-th vertex $\gamma_j$ of $\gamma.$

\begin{claim}\label{claim1-ultra} 
For any $(k+1)$-simplex $\gamma= [x_{i_1-1},x_{i_1},x_{i_2},\dots,x_{i_{k+1}}]$ with $2\leq i_1<i_2<\dots<i_{k+1}\leq n$ and any $j=1,\dots, k-1$,
\[\ell( \gamma)= u_X(x_{i_1-1},x_{i_{k+1}})=\ell\left(  \face_j(\gamma)\right) = \ell(\partial\gamma) .\] 
\end{claim}

The first and second equalities follow from Lemma \ref{lem:prop of ordering} Part (2) directly, 
by which we also have $\ell\left( \face_0(\gamma)\right) = u_X(x_{i_1},x_{i_{k+1}})$ and 
$\ell\left( \face_k(\gamma)\right) = u_X(x_{i_1},x_{i_{k}})$. 
Moreover, this implies that 
\[\ell\left( \face_0(\gamma)\right),\ell\left( \face_k(\gamma)\right)\leq \ell\left( \face_j(\gamma)\right)\] 
for every $j=1,\dots, k-1$. Because simplices are orthogonal, we have 
\[\ell(\partial \gamma) 
= \max_{j=0,\dots,k+1} \ell(\face_j(\gamma))
= \ell(\face_j(\gamma)).\]

\begin{claim}\label{claim2-ultra}
Let $A:=\left\{[x_{i_1-1},x_{i_1},x_{i_2},\dots,x_{i_{k+1}}]\mid 2\leq i_1<i_2<\dots<i_{k+1}\leq n\right\}$, whose cardinality is ${n-1 \choose k+1}$. Then, $\partial A$ is orthogonal.
\end{claim}

For any linear combination $c:=\sum_{\gamma\in A} \lambda_{\gamma} \left( \partial \gamma\right)$ of elements in $\partial A$ where the coefficients $\lambda_{\gamma}$ come from the base field $\field$, we want to show that $\ell\left( c\right) =\max_{\lambda_{\gamma}\neq 0 } \ell \left( \partial \gamma\right)$. 
The `$\leq$' follows from the definition of filtration functions. It remains to prove `$\geq $'.

To prove this, consider all simplices that achieve the maximum $\max_{\lambda_{\gamma}\neq 0 } \ell(\partial \gamma)$. Out of these simplices, we select the simplex $\bar{\gamma}=[x_{i_1-1},x_{i_1},x_{i_2},\dots,x_{i_{k+1}}]$ which has the smallest leading vertex according to the given self-consistent order.  
The choice of $\bar{\gamma}$ may not be unique. 

Note that the $1$-st face of $\bar{\gamma}$, denoted as $\face_1(\bar{\gamma})=[x_{i_1-1},x_{i_2},\dots,x_{i_{k+1}}]$, cannot be cancelled out by other terms in the linear combination $\sum_{\gamma\in A} \lambda_{\gamma} \left( \partial \gamma\right).$ 
Consider another $\gamma'$ that also achieves the maximum $\max_{\lambda_{\gamma}\neq 0 } \ell(\partial \gamma)$.
For any $j\geq 2$, the $j$-th face of $\gamma'$ will start with two consecutive vertices, and thus cannot be $[x_{i_1-1},x_{i_2},\dots,x_{i_{k+1}}]$ given that $i_2-(i_1-1)>2$. 
Hence, if $[x_{i_1-1},x_{i_2},\dots,x_{i_{k+1}}]$ were the $j'$-th face of $\gamma'$ for some $j'$, $j'$ can only be $0$ or $1$. Since $\gamma'\neq\bar{\gamma}$, $j'$ cannot be $1$. If $j'=0$, then $\gamma'=[x_{i_1-2},x_{i_1-1},x_{i_2},\dots,x_{i_{k+1}}]$ has a leading vertex smaller than $x_{i_1-1}$. This contradicts the definition of $\bar{\gamma}$ as having the smallest leading vertex. 
Therefore, we have $\ell(c)\geq \ell(\face_1(\bar{\gamma}))$. 
Incorporating Claim \ref{claim1-ultra}, we obtain
\[\ell(c)\geq\ell(\face_1(\bar{\gamma})) =\ell(\bar{\gamma})=\ell(\partial\bar{\gamma})=\max_{\lambda_{\gamma}\neq 0 } \ell(\partial \gamma).\] 
Thus, Claim \ref{claim2-ultra} holds. 

\begin{claim}\label{claim3-ultra} 
Let $B:=\left\{[x_{i_1-1},x_{i_1},x_{i_2},\dots,x_{i_{k}}]\mid 2\leq i_1<i_2<\dots<i_{k}\leq n\right\},$ whose cardinality is ${n-1 \choose k}$.
Then, $\partial A\sqcup B$ is an orthogonal basis for $\opC_{k}\!\left(\Full\left(X\right)\right)$ whose dimension is ${n \choose k+1}.$
\end{claim}

First, notice that the cardinality of $\partial A\sqcup B$ matches the dimension of $\opC_{k}\!\left(\Full\left(X\right)\right)$: 
\[{n-1 \choose k+1}+{n-1 \choose k} = {n \choose k+1}.\]
Thus, to show that $\partial A\sqcup B$ is an orthogonal basis for $\opC_{k}\!\left(\Full\left(X\right)\right)$, it suffices to show $\partial A\sqcup B$ is an orthogonal subset. Since both $ \partial A$ and $B$ are orthogonal subsets, by Lemma \ref{lem:orthogonal}, it remains to show that $ \partial A$ and $B$ are orthogonal to each other. 

Let $c$ and $c'$ be non-zero linear combinations of elements in $\partial A$ and $B$, respectively. We want to prove $\ell(c+c') = \max\{\ell(c),\ell(c')\}.$
When $\ell(c)\neq \ell(c')$, apply Lemma \ref{rmk:property of filtration function}.
When $\ell(c)=\ell(c')$, since `$\leq$' is trivial, we only need to show `$\geq $'.  assume $c =\sum_{\gamma\in A} \lambda_{\gamma} \left( \partial \gamma\right)$ and let $\gamma$ be a simplex from the summands of $c$ that achieves the maximum $\ell(c)=\max_{\lambda_{\gamma}\neq 0 } \ell(\partial \gamma)$. 
By Claim \ref{claim1-ultra}, we know $\ell(c)=\ell(\gamma)=\ell(\face_1(\gamma))$. By noting that $\face_1(\gamma)$ is not in the span of $B$, we conclude that $\ell(c+c')\geq \ell(\face_1(\gamma))=\ell(c)=\ell(c')$.

Thus, Claim \ref{claim3-ultra} holds.\\

In summary, we have proved that the following gives orthogonal bases for the boundary operators:
\label{para:SVD diagram}
\begin{center}
\begin{tikzcd}[column sep=0em]
    C_{k+1}: \ar[d,"\partial_{k+1}"] &\Big( \left\{\partial[x_{i_1-1},x_{i_1}, \dots,x_{i_{k+2}}]\right\}, \ar[mapsto,d,"0"] & \left\{[x_{i_1-1},x_{i_1}, \dots,x_{i_{k+1}}]\right\}\Big) \ar[mapsto,d] \\	
    C_k: \ar[d,"\partial_k"] &  0 &\Big(\left\{\partial[x_{i_1-1},x_{i_1}, \dots,x_{i_{k+1}}]\right\}, \ar[mapsto,d,"0"] & \left\{[x_{i_1-1},x_{i_1}, \dots,x_{i_{k}}]\right\} \Big) \ar[mapsto,d] \\
    C_{k-1}:  &  & 0  & \dots.
\end{tikzcd}
\end{center} 

By the definition of verbose barcodes, we have
\begin{align*}
    \caB_{\Ver,k}(X)
    &= \left\{ \left( \ell\left(\partial\gamma\right), \ell\left(\gamma\right) \right) \right\}_{\gamma\in A}\\
    &= \left\{ \left( \ell\left(\partial[x_{i_1-1},x_{i_1}, \dots,x_{i_{k+1}}]\right), \ell\left([x_{i_1-1},x_{i_1}, \dots,x_{i_{k+1}}]\right) \right) \right\}_{2\leq i_1<i_2<\dots<i_{k+1}\leq n}\\   
    &= \bigsqcup_{2\leq i_1<i_2<\dots<i_{k+1}\leq n}\left\{ u_X\left(x_{i_1-1},x_{i_{k+1}}\right)\cdot (1,1) \right\}\\   
    &=\bigsqcup_{i=2}^{n-k}\, \bigsqcup_{j=i+k}^{n} \left\{ u_X(x_{i-1},x_{j})\cdot (1,1) \right\}^{{j-i-1 \choose k-1}}. \notag \\
    &=\bigsqcup_{j-i=k+1}^{n-1}\,\bigsqcup_{i=1}^{n-k-1} \left\{ u_X(x_{i},x_{j})\cdot (1,1) \right\}^{{j-i-2 \choose k-1}}. \qedhere
\end{align*}
\end{proof}

\subsection{Decomposition of FCCs} \label{sec:decomp of fcc}
In this section, we recall from \cite{usher2016persistent}, that the collection of verbose barcodes is a \emph{complete} invariant of FCCs, and that the collection of concise barcodes is an invariant up to filtered homotopy equivalence. In addition, for the case of Vietoris-Rips FCCs, we verify that isometry implies filtered chain isomorphism while the inverse is not true. 

\begin{definition}[Elementary FCCs
] \label{def:elementary-f.c.c.}
For $a\in\bbR$, $L\in[0,\infty]$ and $k\in\bbZ$, we define the \textbf{elementary FCC}, denoted by $\caE(a,a+L,k)$, to be the FCC $(E_*,\partial_E,\ell_E)$ given as follows:
\begin{itemize}
    \item If $L=\infty$, then $E_m:=\begin{cases}
	\field,&\mbox{$m=k$}\\ 
	0,&\mbox{otherwise}
	\end{cases}$, $\partial_E:=0$ and $\ell(\lambda):=a$ for each $0\neq \lambda\in E_k=\field.$
    \item If $L\in[0,\infty)$, then $E_m:=\begin{cases}
	\field x,&\mbox{$m=k$}\\ 
	\field y,&\mbox{$m=k+1$}\\ 
	0,&\mbox{otherwise}
	\end{cases}$ with $\partial_E:y\mapsto x, x\mapsto 0$, $\ell_E: y\mapsto a+L, x\mapsto a$ such that $\{x,y\}$ is an orthogonal basis.
\end{itemize}
\end{definition}
By noting that $((y),(x))$ forms a singular value decomposition for $\partial_{k+1}$, we conclude that the degree-$k$ verbose barcode of $\caE(a,a+L,k)$ is $\left\{  (a,a+L)\right\}$ with the convention that $a+\infty=\infty.$ For $l\neq k$, it is clear that the degree-$l$ verbose barcode of $\caE(a,a+L,k)$ is empty.
The following proposition shows that each FCC can be decomposed as the direct sum of some elementary FCCs.

\begin{proposition}[Proposition 7.4, \cite{usher2016persistent}]
\label{prop:decomposition}
Let $(C_*,\partial_C,\ell_C)$ be a FCC, and denote by $\mathcal{B}_{\Ver,k}$ the degree-$k$ verbose barcode of $(C_*,\partial_C,\ell_C)$. Then there is a filtered chain isomorphism
\[(C_*,\partial_C,\ell_C)\cong \bigoplus _{k\in \bbZ} \bigoplus _{(a,a+L)\in \mathcal{B}_{\Ver,k}}\mathcal{E}(a,a+L,k).\]
\end{proposition}

\begin{theorem}[Theorem A \& B, \cite{usher2016persistent}] Two FCCs $(C_*,\partial_C,\ell_C)$ and $(D_*,\partial_D,\ell_D)$ are
\begin{enumerate}[label=(\roman*)]
    \item filtered chain isomorphic to each other if and only if they have identical verbose barcodes in all degrees;
    \item filtered homotopy equivalent to each other if and only if they have identical concise barcodes in all degrees.
\end{enumerate}
\end{theorem}

\begin{example}[f.h.e. but not f.c.i.] \label{ex:4-point space}
Let $X$ and $Y$ be (ultra-)metric spaces of $4$ points given in Figure \ref{fig:f.h.e.-not f.c.i.}. The FCCs $\left( \fccVR{X},\partial^X,\ell^{X} \right) $ and $\left( \opC_*(\Full(Y)),\partial^Y,\ell^{Y} \right) $ arising from Vietoris-Rips complexes have the same concise barcodes but different verbose barcodes.

\begin{figure}[ht]
\centering
\begin{tikzpicture}[scale=1]
  \node (a) at (-2,1.5) [circle, fill=blue, inner sep=1.24pt, label=above:\textcolor{blue}{$x_4$}] {};
  \node (b) at (-1,-0.5) [circle, fill=blue, inner sep=1.24pt, label=right:\textcolor{blue}{$x_2$}] {};
  \node (c) at (-3,-0.5) [circle, fill=blue, inner sep=1.24pt, label=left:\textcolor{blue}{$x_1$}] {};
  \node (d) at (-2.2,-1) [circle, fill=blue, inner sep=1.24pt, label=below:\textcolor{blue}{$x_3$}] {};
\draw [dashed]  (a)--(b) node[pos=0.6,fill=white]{\footnotesize{$2$}}; 
\draw [dashed] (a) --(c) node[pos=0.6,fill=white]{\footnotesize{$2$}};
\draw [dashed]  (a)--(d) node[pos=0.6,fill=white]{\footnotesize{$2$}};
\draw [dashed]  (b) --(c)node[above,pos=0.4,fill=white]{\footnotesize{$1$}};
\draw [dashed] (c)--(d)node[below left,pos=0.5,fill=white]{\footnotesize{$1$}};
\draw [dashed]  (b)--(d)node[below,pos=0.5,fill=white]{\footnotesize{$1$}};
\end{tikzpicture} 
\hspace{1em}
\begin{tikzpicture}[scale=0.9, line width=0.8pt, >=stealth, y=2.5em]
\node[blue] at (-0.5,3) {\(x_1\)};
\node[blue] at (-0.5,2) {\(x_2\)};
\node[blue] at (-0.5,1) {\(x_3\)};
\node[blue] at (-0.5,0) {\(x_4\)};
\draw (0,3) -- (2,3);
\draw (0,2) -- (4,2); 
\draw (0,1) -- (2,1); 
\draw (0,0) -- (4,0); 
\draw (4,1) -- (5,1); 
\draw (2,1) -- (2,3); 
\draw (4,2) -- (4,0); 
\node[fill=white] at (2.3,2.3) {\footnotesize{$1$}}; 
\node[fill=white] at (4.3,1.3) {\footnotesize{$2$}}; 
\end{tikzpicture}

\vspace{1em}

\begin{tikzpicture}[scale=1]
\node (e) at (2.5,2) [circle, fill=orange, inner sep=1.24pt, label=above:\textcolor{orange}{$y_4$}] {};
\node (f) at (1,-0.5) [circle, fill=orange, inner sep=1.24pt, label=left:\textcolor{orange}{$y_1$}] {};
\node (g) at (3,-0.5) [circle, fill=orange, inner sep=1.24pt, label=right:\textcolor{orange}{$y_2$}] {};
\node (h) at (1.5,1.5) [circle, fill=orange, inner sep=1.24pt, label=left:\textcolor{orange}{$y_3$}] {};
\draw [dashed]  (f)--(h)node[pos=0.7,fill=white]{\footnotesize{$2$}};
\draw [dashed]  (e)--(f) node[pos=0.6,fill=white]{\footnotesize{$2$}};
\draw [dashed] (g)--(h)node[pos=0.5,fill=white]{\footnotesize{$2$}};
\draw [dashed] (e) --(g) node[pos=0.5,fill=white]{\footnotesize{$2$}};
\draw [dashed] (e)--(h) node[above,pos=0.6,fill=white]{\footnotesize{$1$}};
\draw [dashed]  (f) --(g)node[below,pos=0.4,fill=white]{\footnotesize{$1$}};
    \end{tikzpicture} 
    \hspace{1em}
\begin{tikzpicture}[scale=0.9, line width=0.8pt, >=stealth, y=2.5em]
\node[orange] at (-0.5,3) {\(y_1\)};
\node[orange] at (-0.5,2) {\(y_2\)};
\node[orange] at (-0.5,1) {\(y_3\)};
\node[orange] at (-0.5,0) {\(y_4\)};
\draw (0,3) -- (2,3);
\draw (0,2) -- (2,2); 
\draw (0,1) -- (2,1); 
\draw (0,0) -- (2,0); 
\draw (4,1.5) -- (5,1.5); 
\draw (2,2.5) -- (4,2.5); 
\draw (2,0.5) -- (4,0.5); 
\draw (2,2) -- (2,3); 
\draw (2,0) -- (2,1); 
\draw (4,2.5) -- (4,0.5); 
\node[fill=white] at (2.3,2.8) {\footnotesize{$1$}}; 
\node[fill=white] at (2.3,0.8) {\footnotesize{$1$}}; 
\node[fill=white] at (4.3,1.8) {\footnotesize{$2$}}; 
\end{tikzpicture}
\caption{\textit{Top}: the ultra-metric  space $X$ and its dendrogram representation; \textit{Bottom}: the ultra-metric space $Y$ and its dendrogram representation. In the dendrogram representations, the distance between two points is defined as the first time when the two points are merged together. For example, $d_X(x_1,x_2)=1$ and $d_X(x_1,x_4)=2$.} \label{fig:f.h.e.-not f.c.i.}
\end{figure}

In Figure \ref{fig:f.h.e.-not f.c.i.}, the ultra-metric spaces $X=\{x_1<x_2<x_3<x_4\}$ and $Y=\{y_1<y_2<y_3<y_4\}$ are ordered according to respective self-consistent orders (see page \pageref{para:order in ultra}). 
Then we can apply Theorem \ref{thm:verbose of ultra} to obtain the barcodes of the two metric spaces. 
The diagram below applies to both $X$ and $Y$, so we will use the following compressed notation. First, letting a sequence of indices $i_0i_1\dots i_{k}$ denote the corresponding simplex $[x_{i_0}, x_{i_1},\dots,x_{i_{k}}]$ (or $[y_{i_0}, y_{i_1},\dots,y_{i_{k}}]$, resp.) and applying the diagram on page \pageref{para:SVD diagram}, we have the following singular value decompositions for $\partial_2$ and $\partial_1$: 

\begin{center}
\begin{tikzcd}[column sep=0em]
    C_{2}: \ar[d,"\partial_{2}"] 
    &\Big( \left\{\partial(1234)\right\}, \ar[mapsto,d,"0"] 
    & \left\{123,124,234\right\}\Big) \ar[mapsto,d] \\	
    C_1: \ar[d,"\partial_1"] 
    &  0 
    &\Big(\left\{\partial(123),\partial(124),\partial(234)\right\}, \ar[mapsto,d,"0"] 
    & \left\{12,23,34\right\} \Big) \ar[mapsto,d] \\
    C_0:  &  & 0  & \Big(\left\{\partial(12),\partial(23),\partial(34)\right\} \Big).
\end{tikzcd}
\end{center} 
Then the barcodes of the two metric spaces are given as follows.
\begin{itemize}
    \item The verbose barcodes for $X$ are \label{para:verbose example}
\[\caB_{\Ver,k}(X)=\begin{cases}
	\left\{  (0,1),(0,1),(0,2),(0,\infty)\right\}   ,&\mbox{$k=0$}\\ 
	\left\{  (1,1),(2,2),(2,2)\right\}   ,&\mbox{$k=1$}\\ 
	\left\{  (2,2)\right\}   ,&\mbox{$k=2$}\\ 
	\emptyset,&\mbox{otherwise.}
	\end{cases}\]
\item The verbose barcodes for $Y$ are
\[\caB_{\Ver,k}(Y)=\begin{cases}
	\left\{  (0,1),(0,1),(0,2),(0,\infty)\right\}   ,&\mbox{$k=0$}\\ 
	\left\{  (2,2),(2,2),(2,2)\right\}   ,&\mbox{$k=1$}\\ 
	\left\{  (2,2)\right\}   ,&\mbox{$k=2$}\\ 
	\emptyset,&\mbox{otherwise.}
	\end{cases}\]
	\item The concise barcodes for $X$ and $Y$ are
\[\mathcal{B}_{\Con,k}(X)=\caB_{\Con,k}(Y)=\begin{cases}
	\left\{  (0,1),(0,1),(0,2),(0,\infty)\right\}   ,&\mbox{$k=0$}\\ 
	\emptyset,&\mbox{otherwise.}
	\end{cases}
	\]
\end{itemize}
\end{example}

Recall from Example \ref{ex:VR FCC} the notation for Vietoris-Rips FCCs. Let $(X,d_X)$ and $(Y,d_Y)$ be two finite pseudo-metric spaces. It is clear that the chain complexes $\fccVR{X}$ and $\opC_*(\Full(Y))$ are chain isomorphic if and only if $|X|=|Y|$. In addition, if $|X|=|Y|$ holds, then any bijection $f:X\to Y$ induces a chain isomorphism $$f_*:\fccVR{X}\xrightarrow{\cong}\opC_*(\Full(Y)).$$
Below, we show that the respective Vietoris-Rips FCCs of two isometric pseudo-metric spaces are filtered chain isomorphic.
\begin{proposition}[Isometry implies f.c.i.]\label{thm:iso->f.c.i}
Let $(X,d_X)$ and $(Y,d_Y)$ be two finite pseudo-metric spaces. If $(X,d_X)$ and $(Y,d_Y)$ are isometric, then FCCs $\left( \fccVR{X},\partial^X,\ell^{X} \right) $ and $\left( \opC_*(\Full(Y)),\partial^Y,\ell^{Y} \right) $ are filtered chain isomorphic. \end{proposition}

\begin{proof} Suppose $(X,d_X)$ and $(Y,d_Y)$ are isometric. Then there exists a bijective map $\phi:(X,d_X)\to (Y,d_Y)$ such that $d_X(x,x')=d_Y(\phi(x),\phi(x'))$ for all $x,x'\in X$. Clearly, $\phi$ induces a chain isomorphism $\Phi_*:\fccVR{X}\to \opC_*(\Full(Y))$ such that each $k$-simplex $[x_1,\dots,x_k]$ in $\Full(X)$ is mapped to $[\phi(x_1),\dots,\phi(x_k)]$ in $\Full(Y)$. Since $\phi$ is distance-preserving, we have that $\ell^{Y}\circ\Phi_*=\ell^{X}$. Thus, $\Phi_*$ is a filtered chain isomorphism.
\end{proof}

However, the converse of Proposition \ref{thm:iso->f.c.i} is not true.

\begin{example}[f.c.i. but not isometric] \label{ex:5-point space}
Let $X$ and $Y$ be the ultra-metric spaces  (each consisting of 5 points)  depicted in Figure \ref{fig:f.c.i.-not iso}. These spaces are extensions of those presented in Example \ref{ex:4-point space}, obtained by adding the  points $x_5$ and $y_5$ to $X$ and $Y$, respectively. The distance matrices for $X$ and $Y$ are respectively:
\[\begin{pmatrix}
0 & 1 & 1 & 2 & 2\\
1 & 0 & 1 & 2 & 2\\
1 & 1 & 0 & 2 & 2\\
2 & 2 & 2 & 0 & 0.5\\  
2 & 2 & 2 & 0.5 & 0
\end{pmatrix}
\text{ and }
\begin{pmatrix}
0 & 1 & 2 & 2 & 2\\
1 & 0 & 2 & 2 & 2\\
2 & 2 & 0 & 1 & 1\\
2 & 2 & 1 & 0 & 0.5\\  
2 & 2 & 1 & 0.5 & 0
\end{pmatrix}.\]

\begin{figure}[ht]	  
\centering
\begin{tikzpicture}[scale=1]
\node (a) at (-2,1.5) [circle, fill=blue, inner sep=1.24pt, label=left:\textcolor{blue}{$x_4$}] {};
\node (b) at (-1,-0.5) [circle, fill=blue, inner sep=1.24pt, label=right:\textcolor{blue}{$x_2$}] {};
\node (c) at (-3,-0.5) [circle, fill=blue, inner sep=1.24pt, label=left:\textcolor{blue}{$x_1$}] {};
\node (d) at (-2.2,-1) [circle, fill=blue, inner sep=1.24pt, label=below:\textcolor{blue}{$x_3$}] {};
\node (i) at (-2,2.5) [circle, fill=blue, inner sep=1.24pt, label=above:\textcolor{blue}{$x_5$}] {};
\filldraw (a) [color=blue] circle[radius=2pt];
\filldraw (b)[color=blue] circle[radius=2pt];
\filldraw (c) [color=blue] circle[radius=2pt];
\filldraw (d) [color=blue] circle[radius=2pt];
\filldraw (i) [color=blue] circle[radius=2pt];
\draw [dashed]  (a)--(b) node[pos=0.6,fill=white]{\footnotesize{$2$}}; 
\draw [dashed] (a) --(c) node[pos=0.6,fill=white]{\footnotesize{$2$}};
\draw [dashed]  (a)--(d) node[pos=0.6,fill=white]{\footnotesize{$2$}};
\draw [dashed]  (b) --(c)node[pos=0.25,fill=white]{\footnotesize{$1$}};
\draw [dashed] (c)--(d)node[below left,pos=0.5,fill=white]{\footnotesize{$1$}};
\draw [dashed]  (b)--(d)node[below,pos=0.5,fill=white]{\footnotesize{$1$}};
\draw [dashed]  (i)--(a)node[pos=0.5,fill=white]{\footnotesize{$0.5$}};
\path [dashed,bend right] (i) edge (c);
\path [dashed,bend left] (i) edge (b);
\path [dashed,bend left] (i) edge (d);
\node[fill=white] at (-1,1.5) {\footnotesize{$2$}};
\node[fill=white] at (-3,1.5) {\footnotesize{$2$}};
\node[fill=white] at (-1.5,1.25) {\footnotesize{$2$}};
\end{tikzpicture} 
\hspace{1em}
\begin{tikzpicture}[scale=0.9, line width=0.8pt, >=stealth, y=2.5em]
\node[blue] at (-0.5,3) {\(x_1\)};
\node[blue] at (-0.5,2) {\(x_2\)};
\node[blue] at (-0.5,1) {\(x_3\)};
\node[blue] at (-0.5,0) {\(x_4\)};
\node[blue] at (-0.5,-1) {\(x_5\)};
\draw (0,3) -- (2,3);
\draw (0,2) -- (4,2); 
\draw (0,1) -- (2,1); 
\draw (0,0) -- (4,0); 
\draw (0,-1) -- (1,-1); 
\draw (4,1) -- (5,1); 
\draw (2,1) -- (2,3); 
\draw (4,2) -- (4,0); 
\draw (1,-1) -- (1,0); 
\node[fill=white] at (1.5,-0.3) {\footnotesize{$0.5$}}; 
\node[fill=white] at (2.3,2.3) {\footnotesize{$1$}}; 
\node[fill=white] at (4.3,1.3) {\footnotesize{$2$}}; 
\end{tikzpicture}

\vspace{1em}

\begin{tikzpicture} [scale=1]
\node (e) at (2.5,1.5) [circle, fill=orange, inner sep=1.24pt, label={[label distance=0em]170:\textcolor{orange}{$y_4$}}] {};
\node (f) at (1,-0.5) [circle, fill=orange, inner sep=1.24pt, label=left:\textcolor{orange}{$y_1$}] {};
\node (g) at (3,-0.5) [circle, fill=orange, inner sep=1.24pt, label=right:\textcolor{orange}{$y_2$}] {};
\node (h) at (1.5,1) [circle, fill=orange, inner sep=1.24pt, label=left:\textcolor{orange}{$y_3$}] {};
\node (j) at (2,3) [circle, fill=orange, inner sep=1.24pt, label=left:\textcolor{orange}{$y_5$}] {};
\filldraw (e) [color=orange] circle[radius=2pt];
\filldraw (f)[color=orange] circle[radius=2pt];
\filldraw (g) [color=orange] circle[radius=2pt];
\filldraw (h) [color=orange] circle[radius=2pt]; 
\filldraw (j) [color=orange] circle[radius=2pt]; 
\draw [dashed]  (f)--(h)node[pos=0.7,fill=white]{\footnotesize{$2$}};
\draw [dashed]  (e)--(f) node[below, pos=0.7,fill=white]{\footnotesize{$2$}};
\draw [dashed] (g)--(h)node[pos=0.5,fill=white]{\footnotesize{$2$}};
\draw [dashed] (e) --(g) node[pos=0.5,fill=white]{\footnotesize{$2$}};
\draw [dashed] (e)--(h) node[pos=0.5,fill=white]{\footnotesize{$1$}};
\draw [dashed]  (f) --(g)node[below,pos=0.4,fill=white]{\footnotesize{$1$}};
\draw [dashed]  (j) --(e)node[pos=0.5,fill=white]{\footnotesize{$0.5$}};
\draw [dashed]  (j) --(h)node[pos=0.5,fill=white]{\footnotesize{$1$}};
\path [dashed,bend left] (j) edge (g);
\path [dashed,bend right] (j) edge (f);
\node[fill=white] at (1,1.5) {\footnotesize{$2$}};
\node[fill=white] at (3,1.5) {\footnotesize{$2$}};
\end{tikzpicture} 
\hspace{0.5em}
\begin{tikzpicture}[scale=0.9, line width=0.8pt, >=stealth, y=2.5em]
\node[orange] at (-0.5,3) {\(y_1\)};
\node[orange] at (-0.5,2) {\(y_2\)};
\node[orange] at (-0.5,1) {\(y_3\)};
\node[orange] at (-0.5,0) {\(y_4\)};
\node[orange] at (-0.5,-1) {\(y_5\)};
\draw (0,3) -- (2,3);
\draw (0,2) -- (2,2); 
\draw (0,1) -- (2,1);
\draw (0,0) -- (2,0); 
\draw (4,1.5) -- (5,1.5); 
\draw (2,2.5) -- (4,2.5); 
\draw (2,0.5) -- (4,0.5); 
\draw (0,-1) -- (1,-1); 
\draw (2,2) -- (2,3); 
\draw (2,0) -- (2,1); 
\draw (4,2.5) -- (4,0.5); 
\draw (1,-1) -- (1,0); 
\node[fill=white] at (1.5,-0.3) {\footnotesize{$0.5$}}; 
\node[fill=white] at (2.3,2.8) {\footnotesize{$1$}}; 
\node[fill=white] at (2.3,0.8) {\footnotesize{$1$}}; 
\node[fill=white] at (4.3,1.8) {\footnotesize{$2$}}; 
\end{tikzpicture}
\caption{\textit{Top}: the ultra-metric space $X$ and its dendrogram representation; \textit{Bottom}: the ultra-metric space $Y$ and its dendrogram representation.}  \label{fig:f.c.i.-not iso}
\end{figure}

We apply Theorem \ref{thm:verbose of ultra} to compute the verbose barcodes of the two spaces and see that they are equal:
\[\mathcal{B}_{\Ver,k}(X)=\mathcal{B}_{\Ver,k}(Y)=\begin{cases}
	\left\{  (0,0.5),(0,1),(0,1),(0,2),(0,\infty)\right\}   ,&\mbox{$k=0$}\\ 
	\left\{  (1,1),(2,2),(2,2),(2,2),(2,2),(2,2)\right\}   ,&\mbox{$k=1$}\\ 
	\left\{  (2,2),(2,2),(2,2),(2,2)\right\}   ,&\mbox{$k=2$}\\ 
	\left\{  (2,2)\right\}   ,&\mbox{$k=3$}\\ 
	\emptyset,&\mbox{otherwise.}
	\end{cases}\]
\end{example}

\section{Isometry Theorem (\texorpdfstring{$\di = \dm$}{di=dm})} \label{sec:iso}
In TDA, it is well-known that, under mild conditions (e.g. $q$-tameness, see \cite{CSGO16}), a certain isometry theorem holds: the interleaving distance between persistence modules is equal to the bottleneck distance between their \emph{concise} barcodes (cf. \cite{cohen2007stability,chazal2009proximity,cohen2007stability}). In our notation, this means that for any degree $k$ and any two FCCs $(C_*,\partial_C,\ell_C)$ and $(D_*,\partial_D,\ell_D)$,
\[ \db
    \left(  \caB_{\Con,k}(C_*), \caB_{\Con,k}(D_*)\right)   
    =
    \di\left(  \rmH_k\circ\left( C_*,\partial_C,\ell_C\right) , \rmH_k\circ\left( D_*,\partial_D,\ell_D \right) \right) .\]

In this section, we prove an analogous isometry theorem for the \emph{verbose} barcode: the interleaving distance $\di$ between filtered chain complexes is equal to the matching distance $\dm$ between their respective verbose barcodes: 

\isothm*

\subsection{Interleaving Distance \texorpdfstring{$\di$}{di} between FCCs} \label{sec:inter} 

An FCC $(C_*,\partial_C,\ell_C)$ can be viewed as a functor from the poset category $(\bbR,\leq)$ to the category of chain complexes:
\[t\mapsto \ell_C^{-1}(\{t\}) \text{ and } (t\leq s)\mapsto \left(\ell_C^{-1}(\{t\})\hookrightarrow \ell_C^{-1}(\{s\})\right)\]
where $\ell_C^{-1}(\cdot)$ represents the preimage. 

We review the general notion of interleaving distance given by \cite[Definition 3.1 \& 3.2]{bubenik2014categorification}.
For a given category $\mathcal{C}$, two functors $\bbV,\bbW:(\bbR,\leq)\to \caC$ are said to be \textbf{$\delta$-interleaved} if these exist families of morphisms $\{f_t:V_t\to W_{t+\delta}\}_{t\in \bbR }$ and $\{g_t:W_t\to V_{t+\delta}\}_{t\in \bbR }$ such that the following diagrams commute for all $t\leq t'$:
	\begin{center}
		\begin{tikzcd}[column sep={6em,between origins}]
			V_t \ar[dr, "f_t" below left] 	
			\ar[r,"v_{t,t'}"]
			& 
			V_{t'}			
			\ar[dr,"f_{t'}" ]&
			\\
			&W_{t+\delta}
			\ar[r,"w_{t+\delta,t'+\delta}" below]
			& 
			W_{t'+\delta}
		\end{tikzcd}
\hspace{0.6cm}
	\begin{tikzcd}[column sep={6em,between origins}]
			& V_{t+\delta}
			\ar[r,"v_{t+\delta,t'+\delta}"]
			& 
			V_{t'+\delta}
			\\
			W_t
			\ar[ur, "g_t"] 
			\ar[r,"w_{t,t'}" below]
			& 
			W_{t'}
			\ar[ur,"g_{t'}" below right]&
		\end{tikzcd}
	\end{center} 
and
\begin{center}
		\begin{tikzcd}
			V_t
			\ar[dr, "f_t" below left ] 
			\ar[rr,"v_{t,t+2\delta}"]%
			&& 
			V_{t+2\delta}		
			\\
			& W_{t+\delta}
			\ar[ur,"g_{t+\delta}" below right ]
			& 
		\end{tikzcd}
\hspace{0.6cm}
		\begin{tikzcd}
			& V_{t+\delta}
			\ar[dr,"f_{t+\delta}"]
			& 
			\\
			W_t
			\ar[ur, "g_t"] 
			\ar[rr,"w_{t,t+2\delta}" below]
			&& 
			W_{t+2\delta}.	
		\end{tikzcd}
	\end{center}

\begin{definition}[Interleaving Distance]\label{def:interleaving}
The \textbf{interleaving distance} between two functors $\bbV,\bbW:(\bbR,\leq)\to \caC$ is
$$\di(\bbV,\bbW):=\inf\left\{  \delta\geq0: \bbV \text{ and }\bbW\text{ are }\delta\text{-interleaved}\right\}  .$$
Here we follow the convention that $\inf \emptyset=+\infty.$
\end{definition}

\begin{remark} \label{rmk:general-di}
The concept of $\delta$-interleaving can be reformulated using the following constructions described in \cite{usher2016persistent}. Given an FCC $(C_*,\partial_C,\ell_C)$ and $\lambda\in\bbR$, let $C_*^{\lambda}$ denote the subspace of $C_*$ spanned by $x\in C_*$ such that $\ell_C(x)\leq \lambda$, i.e.\[C_*^{\lambda}:=\ell_C^{-1}([-\infty,\lambda])\subset C_*.\]
Because of the property $\ell_C\circ \partial_C\leq \ell_C$, $C_*^{\lambda}$,  together with the restrictions of $\partial_C$ and $\ell_C$, constitutes an FCC denoted by $(C_*^{\lambda},\partial_C,\ell_C).$ For real numbers $\lambda\leq \lambda'$, the inclusion $i_{\lambda}^{\lambda'}:C_*^{\lambda}\to C_*^{\lambda'}$ naturally gives rise to a chain map from $(C_*^{\lambda},\partial_C,\ell_C)$ to $(C_*^{\lambda'},\partial_C,\ell_C)$ that is filtration preserving. 

For $\delta>0$, a $\delta$-interleaving between two FCCs $(C_*,\partial_C,\ell_C)$ and $(D_*,\partial_D,\ell_D)$ is a pair $(\Phi_*,\Psi_*)$ of chain maps $\Phi_*:C_*\to D_*$ and $\Psi_*:D_*\to C_*$ such that
\begin{itemize}
    \item $\ell_D\circ \Phi_*\leq \ell_C+\delta$;
    \item $\ell_C\circ \Psi_*\leq \ell_D+\delta$;  
    \item For each $\lambda\in \bbR$, the compositions $\Psi_*^{\lambda+\delta}\circ\Phi_*^{\lambda}:C_*^{\lambda}\to C_*^{\lambda+2\delta}$ and $\Phi_*^{\lambda+\delta}\circ\Psi_*^{\lambda}:D_*^{\lambda}\to D_*^{\lambda+2\delta}$ are equal to  the respective inclusions. 
\end{itemize}
This reformulation is inspired by \cite[Definition A.1]{usher2016persistent}, but it differs from the one in that paper. Indeed, our definition is more stringent in that it demands equality between the aforementioned maps whereas the one in \cite[Definition A.1]{usher2016persistent} only requires the compositions $\Phi_*^{\lambda+\delta}\circ\Psi_*^{\lambda}$ and $\Phi_*^{\lambda+\delta}\circ\Psi_*^{\lambda}$, depending on the context, to either be chain homotopic to the inclusions or to induce the same maps on homology as the inclusions.
\end{remark}

\begin{proposition} \label{prop: condition for d_I finite}
Let $(C_*,\partial_C,\ell_C)$ and $(D_*,\partial_D,\ell_D)$ be two FCCs. Then, 
\[\di\left( (C_*,\partial_C,\ell_C),(D_*,\partial_D,\ell_D)\right) <\infty \iff (C_*,\partial_C)\cong (D_*,\partial_D).\]
In particular, if $\di\left( (C_*,\partial_C,\ell_C),(D_*,\partial_D,\ell_D)\right) <\infty$, then any $\delta$-interleaving $(\Phi_*,\Psi_*)$ between $(C_*,\partial_C,\ell_C)$ and $(D_*,\partial_D,\ell_D)$ is such that $\Psi_*\circ\Phi_*=\Id_C$ and $\Phi_*\circ\Psi_*=\Id_D$.
\end{proposition}
\begin{proof} Because $\di\left(  (C_*,\partial_C,\ell_C),(D_*,\partial_D,\ell_D)\right) <\infty$, there is a $\delta$-interleaving $(\Phi_*,\Psi_*)$ between $(C_*,\partial_C,\ell_C)$ and $(D_*,\partial_D,\ell_D)$, for some $\delta>0$. Let $\lambda>0$ be large enough. Then we have $\Psi_*\circ\Phi_*=\Id_C$, as chain maps, because of the following commutative diagram:
\begin{center}
		\begin{tikzcd}[column sep=-1em]
		(C_*,\partial_C,\ell_C)^{\lambda}=(C_*,\partial_C,\ell_C)
			\ar[dr, "\Phi_*" below left ] 
			\ar[rr,"="]%
			&& 
		(C_*,\partial_C,\ell_C)^{\lambda+2\delta}=(C_*,\partial_C,\ell_C)
			\\
			& 
			(D_*,\partial_D,\ell_D)^{\lambda+\delta}=(D_*,\partial_D,\ell_D).
			\ar[ur,"\Psi_*" below right ]
			& 
		\end{tikzcd}
\end{center}
And similarly, $\Phi_*\circ\Psi_*=\Id_D$. Thus, $(C_*,\partial_C)\cong (D_*,\partial_D).$

Conversely, suppose that $(C_*,\partial_C)\cong (D_*,\partial_D)$, via chain maps $\Phi_*:C_*\to D_*$ and $\Psi_*:D_*\to C_*$. It is clear that $(\Phi_*,\Psi_*)$ forms a $\delta$-interleaving between $(C_*,\partial_C,\ell_C)$ and $(D_*,\partial_D,\ell_D)$, for $\delta:=\max\left\{   \|\ell_C-\ell_D\circ \Phi_*\|_{\infty}, \|\ell_D-\ell_C\circ \Psi_*\|_{\infty}\right\}.$
Because both $C_*$ and $D_*$ are finite-dimensional, their filtration functions are bounded above, which implies that $\delta<\infty.$
\end{proof}

\begin{proposition}If two FCCs $(C_*,\partial_C,\ell_C)$ and $(D_*,\partial_D,\ell_D)$ are filtered chain isomorphic (see Definition \ref{def:f.c.i.}), then
$\di\left(  (C_*,\partial_C,\ell_C),(D_*,\partial_D,\ell_D)\right) =0.$
\end{proposition}

Because of Proposition \ref{prop: condition for d_I finite}, the interleaving distance between FCCs is only interesting when we consider the case when two FCCs have the \emph{same} underlying chain complexes. Let $(C_*,\partial_C)$ be a finite-dimensional non-zero chain complex over $\field$, and let $\Iso((C_*,\partial_C))$ be the set of chain isomorphisms on $(C_*,\partial_C)$. 

\begin{theorem}\label{thm:di of same chain} Let $(C_*,\partial_C)$ be a non-zero chain complex over $\field$ and let $\ell_1,\ell_2:C_*\to \bbR\sqcup\left\{  -\infty\right\}   $ be two filtration functions such that both $(C_*,\partial_C,\ell_1)$ and $(C_*,\partial_C,\ell_2)$ are FCCs. Then, 
\begin{equation*}\label{eq:di}
 \di\left( (C_*,\partial_C,\ell_1),(C_*,\partial_C,\ell_2)\right) = \inf_{\Phi_*\in \Iso(C_*,\partial_C)}\|\ell_1-\ell_2\circ \Phi_*\|_{\infty}.
\end{equation*}
Here we follow the convention $(-\infty)-(-\infty)=0$ when computing $\|\ell_1-\ell_2\|_{\infty}$. When $\ell_1$ is the trivial filtration function, we have
$$\di\left( (C_*,\partial_C,\ell_1),(C_*,\partial_C,\ell_2)\right) =\|\ell_2\|_{\infty}.$$
\end{theorem}

\begin{proof} For any $\Phi_*\in \Iso((C_*,\partial_C))$, we check directly that $(\Phi_*,\Phi_*^{-1})$ forms a $\delta$-interleaving between $(C_*,\partial_C,\ell_1)$ and $(C_*,\partial_C,\ell_2)$, for
\[\delta=\max\left\{  \|\ell_1-\ell_2\circ \Phi_*\|_{\infty},\|\ell_1\circ\Phi_*^{-1}-\ell_2\|_{\infty}\right\}   =\|\ell_1-\ell_2\circ \Phi_*\|_{\infty}.\]
Thus,
\[\di\left( \left(  C_*,\partial_C,\ell_1 \right) ,\left(  C_*,\partial_C,\ell_2 \right) \right) 
\leq \inf_{\Phi_*\in \Iso( C_*,\partial_C)}\|\ell_1-\ell_2\circ \Phi_*\|_{\infty}.\]

Conversely, it follows from Proposition \ref{prop: condition for d_I finite} that $ \di\left( (C_*,\partial_C,\ell_1),(C_*,\partial_C,\ell_2)\right) <\infty$, and that each $\delta$-interleaving $(\Phi_*,\Psi_*)$ between $(C_*,\partial_C,\ell_1)$ and $(C_*,\partial_C,\ell_2)$ satisfies $\Phi_*\in \Iso((C_*,\partial_C))$. Moreover, we have $\|\ell_1-\ell_2\circ \Phi_*\|_{\infty}\leq \delta$. 
By minimizing over $\delta$ and $(\Phi_*,\Psi_*)$, 
we obtain
\[\di\left( \left(  C_*,\partial_C,\ell_1 \right) ,\left(  C_*,\partial_C,\ell_2 \right) \right) 
\geq \inf_{\Phi_*\in \Iso( C_*,\partial_C)}\|\ell_1-\ell_2\circ \Phi_*\|_{\infty}. \qedhere\]
\end{proof}

We immediately obtain the following corollary:
\begin{corollary}\label{cor:di of same chain} Let $(C_*,\partial_C)$ be a non-zero chain complex over $\field$ and let $\ell_1,\ell_2:C_*\to \bbR\sqcup\left\{  -\infty\right\}   $ be two filtration functions such that both $(C_*,\partial_C,\ell_1)$ and $(C_*,\partial_C,\ell_2)$ are FCCs. Then, 
\begin{equation*}\label{eq:infty stability}
    \left|\|\ell_1\|_{\infty}-\|\ell_2\|_{\infty}\right|\leq \di\left( (C_*,\partial_C,\ell_1),(C_*,\partial_C,\ell_2)\right) \leq\|\ell_1-\ell_2\|_{\infty}.
\end{equation*}
\end{corollary}

\begin{proof}
    The first inequality follows from $\|\ell_1-\ell_2\circ \Phi_*\|_{\infty}\geq \left|\|\ell_1\|_{\infty}-\|\ell_2\circ \Phi_*\|_{\infty}\right| = \left|\|\ell_1\|_{\infty}-\|\ell_2\|_{\infty}\right|$, for every chain isomorphism $\phi_*$. The second inequality follows by taking $\Phi_*$ to be the identity map.
\end{proof}

\begin{example}[$\di$ between Elementary FCCs]\label{ex:di of e.FCC} 
Recall the notion of elementary FCCs from Definition \ref{def:elementary-f.c.c.}. Let $a$ and $b$ be two real numbers. We claim that
\[\di\left( \caE(a,a,k),\caE(b,b,l)\right) =\begin{cases}
|a-b|,&\mbox{$k=l$}\\ 
\infty,&\mbox{$k\neq l$.}
	\end{cases}\]

When $k=l$, the chain complexes underlying $\caE(a,a,k)$ and $\caE(b,b,k)$ are isomorphic. In $\caE(a,a,k)$, let $x$ and $y$ be the generators of degrees $k$ and $k+1$, respectively. Correspondingly, let $x'$ and $y'$ be the generators for $\caE(b,b,k)$. Any chain isomorphism $\Phi_*$ from $\caE(a,a,k)$ to $\caE(b,b,k)$ can be represented as $\Phi_(x) = \lambda x'$ and $\Phi_(y) = \lambda' y'$ for some non-zero $\lambda,\lambda'\in \field$. Thus, we have
\[ \|\ell_a-\ell_b\circ \Phi_*\|_{\infty}= \max\left\{  |\ell_a(x)-\ell_b(\lambda x)|,|\ell_a(y)-\ell_b(\lambda' y)|\right\}   =|a-b|.\]

For the case $k\neq l$, since the underlying chain complexes of $\caE(a,a,k)$ and $\caE(b,b,l)$ are not isomorphic to each other, we have $\di\left( \caE(a,a,k),\caE(b,b,l)\right) =\infty$.
\end{example}

\subsection{Matching Distance \texorpdfstring{$\dm$}{dm} between Verbose Barcodes}\label{sssec:dm} 
Let $\caH:=\left\{  (p,q):0\leq p< q\leq \infty\right\}   $, and let $\Delta:=\left\{  (r,r):r\in \bbR_{\geq 0}\sqcup \left\{  +\infty\right\}   \right\}   $. We denote $\overline{\caH}:=\caH\sqcup \Delta$ the extended real upper plane. 
Let $d_{\infty}$ be the metric on $\overline{\caH}$ inherited from the $l_{\infty}$-metric: for $p,q,p',q'\in \bbR_{\geq 0}\sqcup \{\infty\}$,
\[d_{\infty}((p,q),(p',q')):=\infdistance{(p,q)}{(p',q')} = 
\begin{cases}
\max\left\{  |p-p'|,|q-q'|\right\},&\mbox{$q,q'<\infty$,}\\ 
|p-p'|,&\mbox{$q=q'=\infty$,}\\
\infty,&\mbox{otherwise.}
	\end{cases}
	\]
It follows that the distance from a point  $(p,q)$ to the diagonal $\Delta$ is $\dfrac{1}{2}(q-p)$.

We will often consider multisets of points in the extended real upper plane. Let $\Delta^{\infty}$, $\caH^\infty$, and $\overline{\caH}^\infty$ represent multisets where each point from $\Delta$, $\caH$, and $\overline{\caH}$, respectively, is included with countably infinite multiplicity. These multisets are equipped with the metric $d^\infty$, inherited from their respective underlying sets. 

Let \( A \) and \( B \) be multisets with supports \(\mathcal{A}\) and \(\mathcal{B}\), and multiplicity functions \( m_A \) and \( m_B \), respectively. 
Define the sets consisting of ``labeled'' elements in \( A \) and \( B \) by
\[
\tilde{A} := \{ (x,i) : x \in \mathcal{A},\ 1 \le i \le m_A(x) \} \quad\text{and}\quad \tilde{B} := \{ (y,j) : y \in \mathcal{B},\ 1 \le j \le m_B(y) \}.
\]
We define the concept of \emph{map} from the multiset \( A \) to the multiset \( B \) as that of a map \(\phi : \tilde{A} \to \tilde{B}\) between the corresponding ``labeled'' sets $\tilde{A}$ and $\tilde{B}$ and, for simplicity, we will then just write $\phi:A\to B.$
If \(\phi\) is bijective, it is called a \emph{bijection} between the corresponding multisets. 

Let \((Z, d_Z)\) be a metric space, and let \( A \) and \( B \) be multisets supported on \( Z \). 
For any map \( \phi: A \to B \), we define its \emph{cost} as  
\[
    \cost_Z(\phi) := \sup_{a \in A} d_Z(a, \phi(a)).
\]

\begin{definition}[The Matching Distance $\dm$] \label{def:dist-match} 
Let $A$ and $B$ be two non-empty multisets supported on $\overline{\caH}$. 
The \textbf{matching distance} between $A$ and $B$ is
\[
    \dm (A,B)
    := \inf\left\{ \cost_{\overline{\caH}}(\phi):A\xrightarrow{\phi} B \text{ a bijection}\right\}  
    = \inf\left\{\sup_{a\in A} \infdistance{a}{\phi(a)}:A\xrightarrow{\phi} B \text{ a bijection}\right\},
\]
where $\dm (A,B)=\infty$ if $\card(A)\neq\card(B)$.
\end{definition}

\begin{definition}[The Bottleneck Distance $\db$] \label{def:bottleneck}
Let $A$ and $B$ be two finite non-empty multisets supported on $\caH$. The \textbf{bottleneck distance} between $A$ and $B$ is
\[\db (A,B):=\dm(A\sqcup \Delta^\infty,B\sqcup \Delta^\infty).\]
\end{definition}

\begin{theorem}[{\cite{chazal2009proximity,lesnick2015theory}}] Let $\bbV$ and $\bbW$ be persistence modules whose vector spaces are finite-dimensional. Then, 
	$$\di(\bbV,\bbW)=\db(\caB(\bbV),\caB(\bbW)),$$
where $\di(\bbV,\bbW)$ is defined  in Definition \ref{def:interleaving} with $\caC=\vs$ and $\caB(\cdot)$ denotes the barcode of a persistence module.
\end{theorem} 

It follows from the following proposition that the matching distance between verbose barcodes is larger than or equal to the bottleneck distance between concise barcodes. 

\begin{proposition}\label{prop:dm-db} 
Let $A$ and $B$ be finite multisets supported on $\caH$. Let $A_0$ and $B_0$ be finite multisets supported on the diagonal $\Delta$ such that $\card(A\cup A_0)=\card(B\cup B_0)$. Then,
\[ \db(A,B)\leq \dm(A\cup A_0,B\cup B_0).\]
\end{proposition}

\begin{proof} 
In this proof, we denote \(\cost_{\overline{\caH}}(\cdot)\) by \(\cost(\cdot)\). 
By the respective definitions of $\db$ and $\dm$, we have
\begin{enumerate}
    \item [(i)] $\db(A,B)=\inf\left\{   
    \cost(\bar{\phi})
    : A\cup \Delta^{\infty}\xrightarrow{\bar{\phi}} B\cup \Delta^{\infty}\text{ a bijection}\right\}  $;
    \item [(ii)] $\dm(A\cup A_0,B\cup B_0)=\inf\left\{   \cost(\phi)
    : A\cup A_0\xrightarrow{\phi} B\cup B_0\text{ a bijection}\right\}  .$
\end{enumerate}

For each  $0\leq r\leq \infty$, consider the multiset $\{(r,r)\}^\infty$ and notice that $$\Delta^\infty = \bigsqcup_{r\in [0,\infty]}  \{(r,r)\}^\infty.$$
Since \(A_0\) and \(B_0\) are finite, for every \( r \) there exists at least one bijection
\[
f_r: \{(r,r)\}^\infty\setminus A_0 \to \{(r,r)\}^\infty\setminus B_0.
\]
Moreover, we have
\[
\cost(f_r) = \|(r,r)- f_r((r,r))\|_\infty = \|(r,r)-(r,r)\|_\infty = 0.
\]
Together, the maps \(\{f_r : r\in[0,\infty]\}\) induce a bijection
\[
f: \Delta^\infty\setminus A_0 \to \Delta^\infty\setminus B_0,
\]
with 
\(
\cost(f) = \sup_{r\in[0,\infty]} \cost(f_r) = 0.
\)

Let \(\phi: A\cup A_0 \to B\cup B_0\) be any bijection. Observe that
\[
A\cup \Delta^\infty = A\cup \left(A_0 \sqcup (\Delta^\infty \setminus A_0)\right) = \left(A\cup A_0\right) \sqcup (\Delta^\infty \setminus A_0)
\]
and, similarly,   
\[
B\cup \Delta^\infty = \left(B\cup B_0\right) \sqcup (\Delta^\infty \setminus B_0).
\]  
Thus, the bijections \(\phi: A\cup A_0 \to B\cup B_0\) and \(f: \Delta^\infty \setminus A_0 \to \Delta^\infty \setminus B_0\) together define a bijection \(\bar{\phi}: A\cup \Delta^\infty \to B\cup \Delta^\infty\).  
Moreover, because $\cost(f)=0$, we have 
\[\dm(A\cup \Delta^{\infty},B\cup \Delta^\infty)\leq \cost(\bar{\phi}) 
=\max\{\cost(\phi), \cost(f)\} = \cost(\phi).\] 
Letting $\phi$ run through all bijections from $A\cup A_0$ to $ B\cup B_0$, we obtain that (i) $\leq$ (ii). 
\end{proof}

The corollary below follows directly from Proposition \ref{prop:dm-db}: 
\begin{corollary}
Given two FCCs $(C_*,\partial_C,\ell_C)$ and $(D_*,\partial_D,\ell_D)$ and any dimension $k$, we have
\[ \db\left(\caB_{\Con,k}(C_*), \caB_{\Con,k}(D_*)\right)\leq \dm\left(\caB_{\Ver,k}(C_*),\caB_{\Ver,k}(D_*)\right).\]
\end{corollary}

At the end of this section, we show that chain isomorphisms induce permutations of verbose barcodes. Given a finite-dimensional FCC $(C_*,\partial_C,\ell_C)$, let $\caB_{\Ver}$ be its verbose barcode. We write $\ell=\ell_C$ and $\caB=\caB_{\Ver}$ for notational simplicity. 

\begin{lemma}\label{lem:change of basis on FCC} Let $(C_*,\partial_C)$ be a non-zero chain complex over $\field$ and let $\ell:C_*\to \bbR\sqcup\left\{  -\infty\right\}   $ be a filtration function such that $(C_*,\partial_C,\ell)$ is an FCC. Then, for any $\Phi_*\in \Iso((C_*,\partial_C))$,
\begin{enumerate}
    \item \label{part:filtration} $(C_*,\partial_C,\ell\circ \Phi_*)$ is a filtered chain complex, cf. Definition \ref{def:fcc}.
    \item \label{part:f.c.i.} $(C_*,\partial_C,\ell)$ and $(C_*,\partial_C,\ell\circ\Phi_*)$ are filtered chain isomorphic, via the chain isomorphism $\Phi_*^{-1}$.
\end{enumerate}
\end{lemma}
\begin{proof} Part (\ref{part:f.c.i.}) follows directly from Part (\ref{part:filtration}), since $\Phi_*^{-1}$ is a chain isomorphism such that \[(\ell\circ \Phi_* )\circ \Phi_*^{-1}=\ell.\] 
It then remains to show in Part (\ref{part:filtration}) that $\ell\circ \Phi_*$ is a filtration function. This holds because for any $x,y\in C_*$ and $0\neq \lambda\in\field$, 
\begin{itemize}
    \item $\ell(\Phi_*(x))=-\infty\iff \Phi_*(x)=0\iff x=0$;
    \item $\ell(\Phi_*(\lambda x))=\ell( \Phi_*(x))$;
    \item $\ell(\Phi_*(x+y))=\ell(\Phi_*(x)+\Phi_*(y))\leq \max\left\{  \ell(\Phi_*(x)),\ell(\Phi_*(y))\right\}   $.
\end{itemize}
In addition, because $\Phi_*$ is chain map and $\ell$ is a filtration function, we have 
\[\ell\circ \Phi_*\circ \partial_C =\ell\circ \partial_C\circ \Phi_*\leq \ell\circ \Phi_*.\qedhere\]
\end{proof}

Because $(C_*,\partial_C,\ell)$ and $(C_*,\partial_C,\ell\circ\Phi_*)$ are filtered chain isomorphic, they have the same verbose barcode $\caB$. Recall from Definition \ref{def:barcode} that for each dimension $k$, the degree-$k$ verbose barcode $\caB_k$ is given by a singular value decomposition $\left(  (y_1,\dots,y_n),(x_1,\dots,x_m) \right) $ of the linear map $\partial_{k+1}:C_{k+1}\to \kernel\partial_k$ (see Definition \ref{def:s.v.d.}). In other words, $(y_1,\dots,y_n)$ and $(x_1,\dots,x_m)$ are orthogonal ordered bases for $C_{k+1}$ and $\kernel\partial_k$, respectively, such that for $r=\rank(\partial_{k+1})$,
\begin{itemize}
    \item $(y_{r+1},\dots,y_n)$ is an orthogonal ordered basis for $\kernel  \partial_{k+1}$;
    \item $(x_{1},\dots,x_r)$ is an orthogonal ordered basis for $\im \partial_{k+1}$;
    \item $\partial_{k+1} y_i=x_i$ for $i=1,\dots,r$.
    \item $\ell(y_1)-\ell(x_1)\geq \dots\geq \ell(y_r)-\ell(x_r).$
\end{itemize}

Since $\Phi_*$ is a chain isomorphism, we have $\Phi_{k+1}(C_{k+1})=C_{k+1}$. In addition,
\revision{
\[y\in \Phi_{k}(\kernel\partial_{k})\iff \Phi_{k}^{-1}(y)\in \kernel\partial_{k}\iff \partial_{k} (\Phi_k^{-1} y) = \Phi_{k-1}^{-1}(\partial_{k} y)=0 \iff \partial_{k} y=0\iff y\in  \kernel\partial_{k}.\]
}
Thus, $\Phi_{k}(\kernel\partial_k)=\kernel\partial_k$, and similarly $\Phi_{k+1}(\kernel\partial_{k+1})=\kernel\partial_{k+1}$. Also, it follows from 
\[y\in \Phi_{k+1}(\im\partial_{k+1})\iff y=\Phi_{k+1}(\partial_{k+1} x)= \partial_{k+1}(\Phi_{k+1} x)\iff y\in \im\partial_{k+1}\]
that $\Phi_{k+1}(\im \partial_{k+1})=\im \partial_{k+1}$.

\begin{proposition}\label{prop:permutation of barcodes} Denote $\tilde{y}_i=\Phi_{k+1}^{-1}(y_i)$ and $\tilde{x}_j=\Phi_{k}^{-1}(x_j)$ for each $i=1,\dots,n$ and $j=1,\dots,m$. Then $\left(  (\tilde{y}_1,\dots,\tilde{y}_n),(\tilde{x}_1,\dots,\tilde{x}_m) \right) $ forms a singular value decomposition of $\partial_{k+1}:C_{k+1}\to \kernel\partial_k$ in the filtered chain complex $(C_*,\partial_C,\ell\circ\Phi_*)$.
\end{proposition} 

\begin{proof} 
The proof, while essentially a straightforward list of validated axioms, is included here to assist readers in understanding the mechanics of singular value decompositions in non-Archimedean normed vector spaces.

First note that $(\tilde{y}_1,\dots,\tilde{y}_n)$ and $(\tilde{x}_1,\dots,\tilde{x}_m)$ are orthogonal ordered bases for $\Phi_{k+1}(C_{k+1})=C_{k+1}$  and $\Phi_k(\kernel\partial_k)=\kernel\partial_k$, respectively. Moreover, for $r=\rank(\partial_{k+1})$, 
\begin{itemize}
    \item $(\tilde{y}_{r+1},\dots,\tilde{y}_n)$ is an orthogonal ordered basis for $\Phi_{k+1}(\kernel  \partial_{k+1})=\kernel  \partial_{k+1}$;
    \item $(\tilde{x}_{1},\dots,\tilde{x}_r)$ is an orthogonal ordered basis for $\Phi_{k+1}(\im \partial_{k+1})=\im \partial_{k+1}$;
    \item $\partial_{k+1} \tilde{y}_i=\tilde{x}_i$ for $i=1,\dots,r$.
    \item $\ell\circ\Phi_*(\tilde{y}_1)-\ell\circ\Phi_*(\tilde{x}_1)\geq \dots\geq \ell\circ\Phi_*(\tilde{y}_r)-\ell\circ\Phi_*(\tilde{x}_r).$
    \qedhere
\end{itemize}
\end{proof}

\paragraph{Comparing \texorpdfstring{$\db$}{db} and \texorpdfstring{$\dm$}{dm} with the Hausdorff Distance \texorpdfstring{$\dhaus$}{dh}.}

Recall the Hausdorff distance $\dhaus$ from page \pageref{para:dh}. 

\begin{proposition}\label{prop:dh and dm}
For any $A,B\subset \overline{\caH}^\infty$,
\begin{equation*}
   \dhaus(A,B)\leq \dm (A,B). 
\end{equation*}
As a result, for two finite metric spaces $X$ and $Y$ and degree $k\geq 0$, we have
\begin{itemize}
    \item $\dhaus(\caB_{\Con,k}(X)\sqcup \Delta^\infty,\caB_{\Con,k}(Y)\sqcup \Delta^\infty)\leq \db(\caB_{\Con,k}(X),\caB_{\Con,k}(Y))$; and
    \item $\dhaus(\caB_{\Ver,k}(X),\caB_{\Ver,k}(Y))\leq \dm(\caB_{\Ver,k}(X),\caB_{\Ver,k}(Y))$.
\end{itemize}
\end{proposition}

\begin{proof}
For any bijiection $\phi:A\to B$ and for any $a\in A$, we have
\[ d^\infty(a,B)\leq  \infdistance{a}{\phi(a)} \quad  \text{and} \quad  d^\infty(\phi(a),A)\leq  \infdistance{a}{\phi(a)}.\]
It then follows that
\[\max\left\{  \max_{a\in A}d^\infty(a,B), \max_{\phi(a)\in B}d^\infty(\phi(a),A)\right\}  \leq\min \left\{   \max_{a\in A}\infdistance{a}{\phi(a)}: A\xrightarrow{\phi} B \text{ a bijiection}\right\}  ,\]
i.e. $\dhaus (A,B)\leq \dm(A,B).$
\end{proof}

We use the notation $\db$, $\dhaus$, and $\dm$ to refer to the distances $\db(\caB_{\Con,k}(X),\caB_{\Con,k}(Y))$, $\dhaus(\caB_{\Ver,k}(X),\caB_{\Ver,k}(Y))$, and $\dm(\caB_{\Ver,k}(X),\caB_{\Ver,k}(Y))$, respectively. While we have shown that both $\db$ and $\dhaus$ provide lower bounds for $\dm$, Example \ref{ex:dhaus v.s. db} illustrates that there is no ordering between $\db$ and $\dhaus$.

\begin{example}\label{ex:dhaus v.s. db}
Consider the metric spaces $X$ and $Y$ from Example \ref{ex:4-point space} together with their verbose barcodes in degree $1$. It is clear that
\[\dhaus(\caB_{\Ver,1}(X),\caB_{\Ver,1}(Y))=1>0 = \db(\caB_{\Con,1}(X),\caB_{\Con,1}(Y)).\]

Consider the same $X$, and let $Y'$ be the metric space obtained by changing the length of the bottom edge of $Y$ in Figure \ref{fig:f.h.e.-not f.c.i.} from $1$ to $2$. One can easily verify that $\caB_{\Con,0}(X)=\{(0,1)^2,(0,2), (0,\infty)\}$ and $\caB_{\Con,0}(Y')=\{(0,1),(0,2)^2, (0,\infty)\}$. In this case, we have 
\[\dhaus(\caB_{\Ver,0}(X),\caB_{\Ver,0}(Y'))=0<1= \db(\caB_{\Con,0}(X),\caB_{\Con,0}(Y')).\]
\end{example}

For any $A=\{a_1,\dots,a_n\}\subset \caH^\infty$ and $\vec{m}\in \bbN^n$, let $A(\vec{m}):=\{a_i^{m_i+1}\}_{i=1}^n\subset\caH^\infty.$

\begin{proposition}\label{prop:dh = sup dm}
For any finite $A=\{a_1,\dots,a_n\},B=\{b_1,\dots,b_{n'}\}\subset \caH^\infty$, we have
\begin{equation*}\label{eq:dh and dm}
   \dhaus(A,B)
   =\inf_{\substack{(\vec{m},\vec{m}')\in \bbN^n\times \bbN^{n'}\\ \|\vec{m}\|_1+n = \|\vec{m}'\|_1+n'}}
   \dm(A(\vec{m}),B(\vec{m}')).
\end{equation*}
\end{proposition}

\begin{proof}
The `$\leq $' direction follows from Proposition \ref{prop:dh and dm} and the fact that $\dhaus$ does not depend on the multiplicities of points: for any $(\vec{m},\vec{m}')\in \bbN^n\times \bbN^{n'}$,
\[\dhaus(A,B)=\dhaus(A(\vec{m}),B(\vec{m}'))\leq \dm(A(\vec{m}),B(\vec{m}')).\]

We now prove the inverse inequality `$\geq $'. For each $a$, let $\phi(a)\in B$ be any one of the closest points in $B$ to $a$. This gives us a bijection between multisets \[\phi:A\to \{\phi(a_1),\dots,\phi(a_n)\} \text{ with } a_i\mapsto \phi(a_i).\] 
Similarly, we have a bijection between multisets $\varphi:B\to \{\varphi(b_1),\dots,\varphi(b_{n'})\}$ such that each $\varphi(b)$ is one of the closest point in $A$ to $b$. 

For each $i$, define $m_i:=\card(\varphi^{-1}(a_i))+1$ and let $\vec{m}=(m_1,\dots, m_n)$. For each $j$, define $m_j':=\card(\phi^{-1}(b_j))+1$ and let $\vec{m}':=(m_1',\dots,m_{n'}')$. Then $\phi\sqcup \varphi^{-1}$ defines a bijection between multisets
\[A(\vec{m})=A\sqcup \{\varphi(b_1),\dots,\varphi(b_{n'})\} \xrightarrow{\phi\sqcup \varphi^{-1}} B(\vec{m}')=\{\phi(a_1),\dots,\phi(a_n)\}\sqcup B.\]
Therefore, we have 
\begin{align*}
\inf_{\substack{(\vec{m},\vec{m}')\in \bbN^n\times \bbN^{n'}\\ \|\vec{m}\|_1+n = \|\vec{m}'\|_1+n'}}
   \dm(A(\vec{m}),B(\vec{m}'))
   \leq &\dm(A(\vec{m}),B(\vec{m}'))\\
   =&
   \max_{a\in A(\vec{m})}d^\infty(a,(\phi\sqcup \varphi^{-1})(a))\\ 
   =& \max\left\{\max_{a\in A}\infdistance{a}{\phi(a)},\max_{a\in \{\varphi(b_1),\dots,\varphi(b_{n'})\}}\infdistance{\varphi(b_i)}{b_i}\right\} \\
   =& \max\left\{\max_{a\in A}d^\infty(a,B),\max_{b\in B}d^\infty(A,b)\right\} \\
   =&\dhaus(A,B). \qedhere
\end{align*}
\end{proof}


\subsection{Proof of the Isometry Theorem}\label{sec:pf of iso}

In this section, we prove the isometry theorem. If two FCCs have non-isomorphic underlying chain complexes, then $\di$ between the two FCCs is $\infty$, and so is the $\dm$ between their verbose barcodes \revision{(see Lemma~\ref{lem:non-isometric chain complexes})}. Thus, it remains to show that when two FCCs have the same underlying chain complexes, $\dm$ between their verbose barcodes is equal to the interleaving distance between the two FCCs.

\isothm*

\subsubsection{The Inequality \texorpdfstring{$\dm\leq \di$}{dm<=di}} 
\label{sec:dm<=di}
In this section, we prove the inequality $\dm\leq \di$. 

First, we state here a lemma proved by Usher and Zhang in \cite{usher2016persistent}, along with a sketch of the proof given in their paper.

\begin{lemma}[Lemma 9.2, \cite{usher2016persistent}]\label{lem:change ell} Let $(C,\ell_C)$ and $(D,\ell_D)$ be two orthogonalizable $\field$-spaces, let $A:(C,\ell_C)\to (D,\ell_D)$ be a linear map with the singular value decomposition $((y_1,\dots,y_n),$ $(x_1,\dots,x_m))$ and let $\ell_D'$ be another filtration function on $D$ such that $(D,\ell_D')$ is orthogonalizable. For any $\delta\geq\|\ell_D-\ell_D'\|_\infty$, there is a singular value decomposition $((y_1',\dots,y_n'),(x_1',\dots,x_m'))$ for the map $A:(C,\ell_C)\to (D,\ell_D')$ such that
\begin{itemize}
    \item $\ell_C(y_i')=\ell_C(y_i)$, for $i=1,\dots,n$;
    \item $|\ell_D'(x_i')-\ell_D(x_i)|\leq \delta$, for $i=1,\dots, r:=\rank(A)$.
\end{itemize}
\end{lemma}

\begin{lemma} \label{lem:infity stability}
Let $(C_*,\partial_C)$ be a finite-dimensional non-zero chain complex over $\field$ and let $\ell_1,\ell_2:C_*\to \bbR\sqcup\left\{  -\infty\right\}   $ be two filtration functions such that both $(C_*,\partial_C,\ell_1)$ and $(C_*,\partial_C,\ell_2)$ are FCCs. Denote by $\caB_{\Ver}^1$ and $\caB_{\Ver}^2$ the verbose barcodes of $(C_*,\partial_C,\ell_1)$ and $(C_*,\partial_C,\ell_2)$, respectively. Then, we have 
\[\dm
    \left(  \caB_{\Ver}^{1}, \caB_{\Ver}^{2}\right)  =
    \sup_{k\in \bbZ_{\geq 0}}\dm
    \left(  \caB_{\Ver,k}^{1}, \caB_{\Ver,k}^{2}\right)   \leq \|\ell_1-\ell_2\|_{\infty}.\]
\end{lemma}
\begin{proof} Even though \cite[Proposition 9.3]{usher2016persistent} states a weaker result, their proof, which we provide for completeness, permits establishing the claim.

Fix an integer $k\in \bbZ_{\geq 0}$ and a $\delta\geq \|\ell_1-\ell_2\|_{\infty}$. We want to show that $\dm
    \left(  \caB_{\Ver,k}^{1}, \caB_{\Ver,k}^{2}\right)   \leq \delta$.
    
Let $r:=\rank(A)$, and let $((y_1,\dots,y_n),(x_1,\dots,x_m))$ be a singular value decomposition for $\partial:=\partial_{k+1}:(C_{k+1},\ell_1)\to (\kernel \partial_k,\ell_1)$. Then we follow the following steps to construct a singular value decomposition for $\partial_{k+1}:(C_{k+1},\ell_2)\to (\kernel \partial_k,\ell_2)$.
\begin{enumerate}
    \item Apply Lemma \ref{lem:change ell} to obtain  $((y_1',\dots,y_n'),(x_1',\dots,x_m'))$, a singular value decomposition for $\partial:(C_{k+1},\ell_1)\to (\kernel \partial_k,\ell_2)$ such that
    \begin{itemize}
        \item $\ell_1(y_i')=\ell_1(y_i)$, for $i=1,\dots,n$;
        \item $|\ell_2(x_i')-\ell_1(x_i)|\leq \delta$, for $i=1,\dots, r$.
    \end{itemize}
    \item The dual elements $((x_1'^*,\dots,x_m'^*),(y_1'^*,\dots,y_n'^*))$ form a singular value decomposition for the adjoint map $\partial^*: ((\kernel \partial_k)^*,\ell_2^*)\to (C_{k+1}^*,\ell_1^*)$, cf. page \pageref{para:dual}.
    \item Apply Lemma \ref{lem:change ell} to obtain  $((\xi_1,\dots,\xi_m),(\eta_1,\dots,\eta_n))$, a singular value decomposition for $\partial^*: ((\kernel \partial_k)^*,\ell_2^*)\to (C_{k+1}^*,\ell_2^*)$ such that
    \begin{itemize}
        \item $\ell_2^*(\xi_i)=\ell_2^*(x_i'^*)$, for $i=1,\dots,m$;
        \item $|\ell_2^*(\eta_i)-\ell_1^*(y_i'^*)|\leq \delta$, for $i=1,\dots, r$.
    \end{itemize}
    \item The dual elements $((\eta_1^*,\dots,\eta_n^*),(\xi_1^*,\dots,\xi_m^*))$ form a singular value decomposition for the map $\partial^{**}:(C_{k+1}^{**},\ell_2^{**})\to ((\kernel \partial_k)^{**},\ell_2^{**})$, i.e. the map $\partial:(C_{k+1},\ell_2)\to (\kernel \partial_k,\ell_2)$.
\end{enumerate}

Next we define a bijection between the finite-length bars in $\caB_{\Ver,k}^1$ and $\caB_{\Ver,k}^2$ by
\[f: (\ell_1(x_i),\ell_1(y_i))\mapsto (\ell_2(\xi_i^*),\ell_2(\eta_i^*)),\forall i=1,\dots,r,\]
and check that $ \max_{1\leq i\leq r} \infdistance{(\ell_1(x_i),\ell_1(y_i))}{(\ell_2(\xi_i^*),\ell_2(\eta_i^*))}  \leq \delta$. Indeed, for $1\le i\leq r$,
\begin{align*}
|\ell_2(\xi_i^*)-\ell_1(x_i)|
&\leq |\ell_2(\xi_i^*)+\ell_2^*(x_i'^*)|+ |-\ell_2^*(x_i'^*)-\ell_1(x_i)| \text{   (by triangle inequality),}\\
&\leq  |-\ell_2^*(\xi_i)+\ell_2^*(x_i'^*)|+ |\ell_2(x_i')-\ell_1(x_i)|  \text{   (by the property of $\ell_2^*$),}\\
&\leq 0+ \delta=\delta
\end{align*}
and similarly,
\[|\ell_2(\eta_i^*)-\ell_1(y_i)|
= |-\ell_2^*(\eta_i)-\ell_1(y_i')|=|-\ell_2^*(\eta_i)+\ell_1^*(y_i'^*)|\leq \delta.\]

Then, it remains to build a bijection $f$ between infinite-length bars in $\caB_{\Ver,k}^1$ and $\caB_{\Ver,k}^2$ such that the difference between the birth time of a bar with the birth time of its image under $f$ is controlled by $\delta$. Let $V_1$ and $V_2$ be an $\ell_1$-orthogonal complement and $\ell_2$-orthogonal complement of $\im\partial_{k+1}$ inside $\kernel \partial_k$, respectively. For $j=1,2$, let $\pi_j:\kernel\partial_k\to V_j$ be the $\ell_j$-orthogonal projection associated with the decomposition $\kernel \partial_k=\im\partial_{k+1}\oplus V_j$.

Notice that $\pi_1|_{V_2}:V_2\to V_1$ is a linear isomorphism, whose inverse is $\pi_2|_{V_1}:V_1\to V_2$. In addition, 
\revision{a singular value decomposition for the map $\pi_2|_{V_1}$ induces a $\ell_1$-orthogonal ordered basis $(x_{r+1},\dots,x_m)$ for $V_1$, and a $\ell_2$-orthogonal ordered basis  $(\pi_2|_{V_1}(x_{r+1}),\dots,\pi_2|_{V_1}(x_m))$ for $V_2$.} Define a bijection between the infinite-length bars in $\caB_{\Ver,k}^1$ and $\caB_{\Ver,k}^2$ as
\[f: (\ell_1(x_i),\infty)\mapsto (\ell_2(\pi_2|_{V_1}(x_{i})),\infty),\forall i=r+1,\dots,m,\]
and check that $ \max_{r+1\leq i\leq m} \infdistance{(\ell_1(x_i),\infty)}{(\ell_2(\pi_2|_{V_1}(x_{i})),\infty)}  \leq \delta$. This holds because, for $r+1\leq i\leq m$,
\[ \ell_2(\pi_2|_{V_1}(x_{i}))\leq \ell_2(x_{i})\leq \ell_1(x_i)+\delta, \]
and 
\[\ell_1(x_i) = \ell_1(\pi_1|_{V_2}(\pi_2|_{V_1}(x_i)))\leq \ell_1(\pi_2|_{V_1}(x_i))\leq \ell_2(\pi_2|_{V_1}(x_{i}))+\delta. \qedhere\]
\end{proof}

\begin{proposition}\label{prop:db-di stability} With the same notation as in Lemma \ref{lem:infity stability}, we have
\[\dm
    \left(  \caB_{\Ver}^{1}, \caB_{\Ver}^{2}\right)  
    \leq 
    \di\left( \left(  C_*,\partial_C,\ell_1\right) ,\left(  C_*,\partial_C,\ell_2 \right) \right) .\]
\end{proposition} 

\begin{proof} Given any $\Phi_*\in \Iso( C_*,\partial_C)$, it follows from Proposition \ref{prop:permutation of barcodes} that $\caB_{\Ver}^{2}=\caB_{\Ver}^{( C_*,\partial_C, \ell_2\circ\Phi_*)}$ the verbose barcodes of $( C_*,\partial_C, \ell_2\circ\Phi_*)$. 
Together with Lemma \ref{lem:infity stability}, we have
\[ \dm
    \left(  \caB_{\Ver}^{1}, \caB_{\Ver}^{2}\right)  = \dm
    \left(  \caB_{\Ver}^{1}, \caB_{\Ver}^{( C_*,\partial_C, \ell_2\circ\Phi_*)}\right)  \leq \| \ell_1- \ell_2\circ \Phi_*\|_{\infty}, \]
for every $\Phi_*\in \Iso( C_*,\partial_C).$ Therefore,
\[\dm
    \left(  \caB_{\Ver}^{1}, \caB_{\Ver}^{2}\right) \leq
    \min_{\Phi_*\in \Iso( C_*,\partial_C)}\| \ell_1- \ell_2\circ\Phi_*\|_\infty=
    \di\left( \left(  C_*,\partial_C, \ell_1\right) ,\left(  C_*,\partial_C, \ell_2 \right) \right) , \]
where the equality follows from Theorem \ref{thm:di of same chain}.
\end{proof}

\subsubsection{The Inequality \texorpdfstring{$\dm\geq \di$}{dm>=di}} 
\label{sec:dm>=di}
Next, we establish the reverse inequality $\dm\geq \di$ via an idea similar to the one employed in demonstrating that the standard bottleneck distance between concise barcodes is at most the interleaving distance between persistent modules, cf. \cite[Theorem 3.4]{lesnick2015theory}. 

\revision{
    \begin{lemma} \label{lem:non-isometric chain complexes}
    If $(C_*,\partial_C,\ell_C)$ and $(D_*,\partial_D,\ell_D)$ are FCCs with non-isomorphic underlying chain complexes, then $\dm \left(  \caB_{\Ver,k}(C_*), \caB_{\Ver,k}(D_*)\right) 
    \geq
    \di\left( \left(  C_*,\partial_C,\ell_C\right) ,\left(  D_*,\partial_D,\ell_D \right) \right) = \infty$, for any degree $k\geq 0$.
    \end{lemma}
}

\revision{
    \begin{proof}[Proof (due to Jeong-hwi Joe)]
        We prove the lemma by contradiction.
        Assume instead that $\dm$ is finite. Then, by the definition of $\dm$, the cardinalities of the verbose barcodes of $(C_*, \partial_C, \ell_C)$ and $(D_*, \partial_D, \ell_D)$ must agree in every degree. 
        Moreover, since the $\ell_\infty$ distance between a finite bar and an infinite bar is $\infty$, it follows that, for each degree $k$,
        \begin{itemize}
        \item $\mathcal{B}_{\Ver, k}(C_*)$ and $\mathcal{B}_{\Ver, k}(D_*)$ have the same number of infinite-length bars;
        \item $\mathcal{B}_{\Ver, k}(C_*)$ and $\mathcal{B}_{\Ver, k}(D_*)$ have the same number of finite-length bars.
        \end{itemize}
        The first condition implies that
        \[
        \dim(\ker \partial_{C,k}) - \dim(\im \partial_{C,k+1}) = \dim(\ker \partial_{D,k}) - \dim(\im \partial_{D,k+1}),
        \]
        and the second implies
        \[
        \dim(\im \partial_{C,k+1}) = \dim(\im \partial_{D,k+1}).
        \]
        By combining the above formulas and applying the rank-nullity theorem, we obtain
        \[
        \dim(\ker \partial_{C,k}) = \dim(\ker \partial_{D,k}),
        \quad \text{and} \quad
        \dim C_k = \dim D_k.
        \]
        Denote these common values by $r_k := \dim(\im \partial_{C,k+1})$, $m_k := \dim(\ker \partial_{C,k})$, and $n_k := \dim C_k$. Note that $n_k = m_k + r_{k-1}$. 
        
        For each $k$, choose subspaces $V_{C,k}, W_{C,k} \subseteq C_k$ such that
        \[
        \ker \partial_{C,k} = \im \partial_{C,k+1} \oplus V_{C,k},
        \quad \text{and} \quad
        C_k = \im \partial_{C,k+1} \oplus V_{C,k} \oplus W_{C,k},
        \]
        and let $\big\{x^{C,k}_{r_k+1}, \dots, x^{C,k}_{m_k}\big\}$ be a basis for $V_{C,k}$ and $\big\{x^{C,k}_{m_k+1}, \dots, x^{C,k}_{n_k}\big\}$ for $W_{C,k}$. Define
        \[
        x^{C,k}_i := \partial_{C,k+1}\Big(x^{C,k+1}_{m_{k+1} + i}\Big) \quad \text{for } i = 1, \dots, r_k,
        \]
        so that $\big\{x^{C,k}_1, \dots, x^{C,k}_{r_k}\big\}$ is a basis for $\im \partial_{C,k+1}$. Then $\big\{x^{C,k}_1, \dots, x^{C,k}_{n_k}\big\}$ forms a basis for $C_k$.
        Repeat this construction for $D_k$, producing subspaces $V_{D,k}, W_{D,k}$ and a basis $\{y^{D,k}_1, \dots, y^{D,k}_{n_k}\}$ for $D_k$ defined analogously.
        
        Now define a linear map $\Phi_k : C_k \to D_k$ by $\Phi_k\big(x^{C,k}_i\big) := y^{D,k}_i$ for all $i$, and extend linearly. Then $\Phi_k$ is a vector space isomorphism. Define $\Phi_* := \{\Phi_k\}_k$.
        We verify that $\Phi_*$ is a chain map. For $i = 1, \dots, m_k$,
        \[
        \partial_{D,k}\big(\Phi_k\big(x^{C,k}_i\big)\big) = \partial_{D,k}\big(y^{D,k}_i\big) = 0 = \Phi_{k-1}(0) = \Phi_{k-1}\big(\partial_{C,k}\big(x^{C,k}_i\big)\big).
        \]
        For $i = m_k + 1, \dots, n_k$,
        \begin{align*}
        \partial_{D,k}(\Phi_k\big(x^{C,k}_i\big)\big) &= \partial_{D,k}\big(y^{D,k}_i\big) = y^{D,k-1}_{i - m_k} = \Phi_{k-1}(x^{C,k-1}_{i - m_k}) = \Phi_{k-1}\big(\partial_{C,k}\big(x^{C,k}_i\big)\big).
        \end{align*}
        Therefore, $\Phi_*$ is a chain map. A symmetric argument shows that $\Phi_*^{-1}$ is also a chain map. Hence, $C_*$ and $D_*$ are isomorphic as chain complexes, contradicting the assumption that they are not.
    \end{proof}
}

\begin{proof} [Proof of Theorem \ref{thm:dm=di} ``$\dm\geq \di$'']
    \revision{The case where $(C_*,\partial_C,\ell_C)$ and $(D_*,\partial_D,\ell_D)$ have non-isomorphic underlying chain complexes follows from Lemma~\ref{lem:non-isometric chain complexes}.
    
    We then} consider the case when the chain complexes $(C_*,\partial_C)$ and $(D_*,\partial_D)$ are isomorphic, and we assume without loss of generality that $(D_*,\partial_D)=(C_*,\partial_C)$ and write $\ell_1:=\ell_C,\ell_2:=\ell_D$. 

    Take some $\delta>0$ such that $\delta>\dm\left(  \caB_{\Ver,k}^{1}, \caB_{\Ver,k}^{2}\right) $ for all degree $k\in \bbZ_{\geq 0}$. Recall that
    \[\dm\left(  \caB_{\Ver,k}^{1}, \caB_{\Ver,k}^{2}\right) =\inf\left\{   \max_{a\in \caB_{\Ver,k}^{1}} \infdistance{a}{f_k(a)}\mid\caB_{\Ver,k}^{1}\xrightarrow{f_k} \caB_{\Ver,k}^{2} \text{ a bijection }\right\}  .\]
    Thus, for each $k$, there is a bijection $f_k:\caB_{\Ver,k}^{1}\to \caB_{\Ver,k}^{2}$ such that 
    \begin{equation}\label{eq:cost of f}
        \max_{a\in \caB_{\Ver,k}^{1}} \infdistance{a}{f_k(a)}\leq \delta.
    \end{equation}

    For $a\in \caB_{\Ver,k}^{1}\subset\overline{\caH}^\infty$, we assume that $a=(a_1,a_2)$. Also, we write $b=f_k(a)$ and assume that $b=(b_1,b_2)$. Next, we construct an isomorphism between the following elementary FCCs (see Definition \ref{def:elementary-f.c.c.}): \[h_k:\caE(a_1,a_2,k)\to\caE(b_1,b_2,k).\] Notice that $a_2$ and $b_2$ are either both finite or both infinite, otherwise the left-hand side of Equation (\ref{eq:cost of f}) is equal to $\infty$, which contradicts with $\delta<\infty$. 

    Case (1): $a_2=b_2=\infty$. Then $\caE(a_1,a_2,k)$ and $\caE(b_1,b_2,k)$ have the same underlying chain complex:
    \[\begin{tikzcd}
        \dots \ar[r]
        & 0 \ar[r]
        & \field x_k \ar[r, "\partial_{k}=0"]
        & 0\ar[r]
        &\dots,
        \end{tikzcd}\]
    and the filtration functions are given by $\ell_1(x_k)=a_1$ and $\ell_2(x_k)=b_1$, respectively. The following defines a chain isomorphism
    \[h_k:\caE(a_1,\infty,k)\to\caE(b_1,\infty,k)\text{  with  }x_k\mapsto x_k.\]

    Case (2): $a_2,b_2<\infty$. Then $\caE(a_1,a_2,k)$ and $\caE(b_1,b_2,k)$ have the same underlying chain complex:
    \[\begin{tikzcd}
        \dots \ar[r]
        & 0 \ar[r]
        & \field y_{k+1} \ar[rrr, "\partial_{k+1}(y_{k+1})=x_k"]
        &&
        & \field x_k \ar[r, "\partial_{k}=0"]
        & 0\ar[r]
        &\dots,
        \end{tikzcd}\]
    and the filtration functions are given by $\ell_1(x_k)=a_1$, $\ell_1(y_{k+1})=a_2$ and $\ell_2(x_k)=b_1$, $\ell_2(y_{k+1})=b_2$, respectively. The following defines a chain isomorphism
    \[h_k:\caE(a_1,a_2,k)\to\caE(b_1,b_2,k)\text{  with  }x_k\mapsto x_k, y_{k+1}\mapsto y_{k+1}.\]

    In either case, it is straightforward to check that $h_k$ satisfies the following condition
    \[ \|\ell_1-\ell_2\circ h_k\|_\infty\leq \max\left\{  |a_1-b_2,a_1-b_2|\right\}   = \infdistance{a}{f(a)}\leq \delta.\]
    We write $h_{k,a}$ whenever it is necessary to emphasize that $h_k$ depends on $a$.

    Recall from Proposition \ref{prop:decomposition} that we have the following decomposition of FCCs
    \[\left(  C_*,\partial_C,\ell_1\right) \cong \bigoplus_{k\in \bbZ_{\geq 0}}\bigoplus_{a\in \caB_{\Ver,k}^{1}} \caE(a_1,a_2,k)  \quad  \text{and} \quad 
    \left(  C_*,\partial_C,\ell_2\right) \cong \bigoplus_{k\in \bbZ_{\geq 0}}\bigoplus_{b\in \caB_{\Ver,k}^{2}} \caE(b_1,b_2,k).\]
    Let $h:=\bigoplus_{k\in \bbZ_{\geq 0}}\bigoplus_{a\in \caB_{\Ver,k}^{1}} h_{k,a}:\left(  C_*,\partial_C,\ell_1\right)  \to \left(  C_*,\partial_C,\ell_2\right) $, which is then a chain isomorphism such that
    \[ \|\ell_1-\ell_2\circ h\|_\infty
    =\max_{k\in \bbZ_{\geq 0}}\max_{a\in \caB_{\Ver,k}^{1}} \|\ell_1-\ell_2\circ h_{k,a}\|_\infty \leq \delta.\]
    It then follows from Theorem \ref{thm:di of same chain} that
    \[ \di\left( \left(  C_*,\partial_C, \ell_1\right) ,\left(  C_*,\partial_C, \ell_2 \right) \right) = 
    \min_{\Phi_*\in \Iso( C_*,\partial_C)}\|
    \ell_1- \ell_2\circ\Phi_*\|_\infty \leq \|\ell_1-\ell_2\circ h\|_\infty\leq \delta. \]
    Since $\delta$ is arbitrary, we obtain the desired inequality $\di\leq \dm$.
\end{proof}

\section{Vietoris-Rips FCCs and an Improved Stability Result} \label{sec:VR FCC}

In this section, we study the Vietoris-Rips FCC of metric spaces. Recall from Example \ref{ex:VR FCC} that given a finite pseudo-metric space $(X,d_X)$, $\left( \fccVR{X},\partial^X,\ell^{X} \right) $ denotes the filtered chain complex arising from Vietoris-Rips complexes of $X$. 
To simplify notation, we will omit the differential map $\partial^X$ and the filtration function. We use $\opC_*(\Full(X))$ to refer to the Vietoris-Rips filtered chain complex $(\opC_*(\Full(X)),\partial^X,\ell^{X})$.

The matching distance between the verbose barcodes of two Vietoris-Rips FCCs of finite metric spaces is infinite if the underlying metric spaces have different cardinality. As a consequence, the matching distance between verbose barcodes of Vietoris-Rips FCCs is not stable under the Gromov-Hausdorff distance $\dgh$, since $\dgh$ between any two bounded metric spaces is always finite. 

We overcome the above problem by incorporating the notion of tripods. Recall from Section \ref{sec:preliminaries} the distortion $\dis(R)$ of a tripod $R$ and how the Gromov-Hausdorff distance can be obtained via finding the infimum of $\dis(R)$ over all tripods $R$.

In Section \ref{sec:pb di and pb db}, through tripods, we can pull back two metric spaces $X$ and $Y$ (with possibly different cardinalities) into a common space $Z$, and then compare (via the matching distance) the barcodes of the FCCs induced by the respective pullbacks; see Definition \ref{def:pb db}. For each degree $k,$ we call the resulting distance the \emph{pullback bottleneck distance} and denote it by $\pbdbk{k}{X}{Y}$. We apply the same strategy to define what we call the \emph{pullback interleaving distance}, written as $\pbdi{X}{Y}$: use tripods to pull back spaces to a common space and compare the interleaving distance between the FCCs induced by the respective pullbacks; see Definition \ref{def:pb di}. 

In Section \ref{sec:pb stab}, we prove the following stability results to show that the pullback bottleneck distance is stable under $\dgh$, and that its stability improves the standard stability result of the bottleneck distance between concise barcodes (cf. Theorem \ref{thm:dgh-classical}):

\pbstabthm*

In Section \ref{sec:ex for strict}, we present examples to demonstrate that both inequalities in Theorem \ref{thm:hatdb-dgh stability} can be tight and strict. In Section \ref{sec:variation of pb distances}, we study two variants of the pullback interleaving/bottleneck distance.

\subsection{Pullback Interleaving Distance and Pullback Bottleneck Distance} \label{sec:pb di and pb db}

In this section, we introduce our construction of the pullback interleaving distance and the pullback bottleneck distance between metric spaces, and study some basic properties of these two notions.

Let $(X,d_X)$ be a finite metric space, and let $\phi:Z\twoheadrightarrow X$ be a finite parametrization of $X$. We denote by $\phi^*d_X$ the pullback pseudo-metric\footnote{The map $\phi$ does not need to be surjective to define the pullback pseudo-metric.} 
on $Z$ induced by $\phi$ given as follows: for any $z,z'\in Z$, 
\[\phi^*d_Z(z,z'):=d_X(\phi(z),\phi(z')).\]
For brevity, we often write the pulled-back pseudo-metric space as  
\begin{equation}  \label{eq:Z_X}
    Z_X := (Z, \phi^* d_X).  
\end{equation}

Given a simplex $\sigma=[z_0,\dots,z_n]$ in $\opC_*\!\left(\Full\left(Z\right)\right)$, we write 
\[\phi(\sigma):=\begin{cases}
[\phi(z_0),\dots,\phi(z_n)],&\mbox{if $\phi(z_i)\neq \phi(z_j)$ for any $i\neq j$,}\\ 
0,&\mbox{otherwise.}
	\end{cases}\]
Let $\fccVR{Z_X} $ denote the Vietoris-Rips FCC of the pseudo-metric space $Z_X$. It is not hard to see that $\phi$ induces a surjective chain map
\[\phi:\fccVR{Z_X} \twoheadrightarrow\fccVR{X} .\]

\begin{remark} We have $\ell^{X}\circ \phi\leq \ell^{Z_X},$ where the inequality can be strict in general. Indeed, if $\phi$ is not injective, there exist $z_1,z_2\in Z$ such that $\phi(z_1)=\phi(z_2)=x\in X$. Then $\phi([z_1,z_2])=0$ and $$\ell^{X}\circ \phi ([z_1,z_2])=-\infty<0=\ell^{Z_X}([z_1,z_2]).$$ 
On the other hand, we always have $\|\ell^{Z_X}\|_{\infty}=\diam(Z,\phi^*d_X)=\diam(X,d_X)=\|\ell^{X}\|_{\infty}.$
\end{remark}

Via the notions of tripod and Vietoris-Rips filtered chain complexes, we construct the pullback interleaving distance as follows: 
\begin{definition} [Pullback interleaving distance] 
\label{def:pb di}
For two finite metric spaces $X$ and $Y$, we define the \textbf{pullback interleaving distance (induced by the Vietoris-Rips FCCs)} between $X$ and $Y$ to be
\begin{align*}
    \pbdi{X}{Y}  :=\inf\left\{   \di\left(  \fccVR{Z_X} ,\fccVR{Z_Y}  \right)  \mid X\xtwoheadleftarrow{\phi_X}Z\xtwoheadrightarrow{\phi_Y}Y\text{ a finite tripod}  \right\},
\end{align*}
where $Z_X:=(Z,\phi_X^*d_X)$ and $Z_Y:=(Z,\phi_Y^*d_Y)$.
\end{definition}

With a similar idea and again invoking tripods, we refine the standard bottleneck distance and introduce a new notion of distance between verbose barcodes:

\begin{definition} [Pullback bottleneck distance] 
\label{def:pb db}
Let $k\in \bbZ_{\geq 0}$. For two finite metric spaces $X$ and $Y$, the \textbf{pullback bottleneck distance (induced by the degree-$k$ verbose barcodes)} between $X$ and $Y$ is defined to be
\begin{align*}
    \pbdbk{k}{X}{Y}  :=\inf\left\{   \dm
    \left(  \caB_{\Ver,k}(Z_X), \caB_{\Ver,k}(Z_Y)\right)  \mid X\xtwoheadleftarrow{\phi_X}Z\xtwoheadrightarrow{\phi_Y}Y\text{ a finite tripod}  \right\}   ,
\end{align*}
where $Z_X:=(Z,\phi_X^*d_X)$ and $Z_Y:=(Z,\phi_Y^*d_Y)$.
\end{definition}

\begin{remark}[Infima are minima in Definition \ref{def:pb di} and \ref{def:pb db}]
\label{rmk:inf=min}
    Applying Proposition \ref{prop:pullback-barcode}, we observe that for any finite tripod $X\xtwoheadleftarrow{\phi_X}Z\xtwoheadrightarrow{\phi_Y}Y$,  $\dm
    \left(  \caB_{\Ver,k}(Z_X), \caB_{\Ver,k}(Z_Y)\right)$ takes values in the finite set $\{|a-b|\mid a\in \im d_X, b\in \im d_Y\}\sqcup\{\infty\}$. 
    In other words,  $\dm
    \left(  \caB_{\Ver,k}(Z_X), \caB_{\Ver,k}(Z_Y)\right)$ is a finite-set valued function as a function defined on (finite) tripods. 
    Consequently, the infimum in the definition of $\pbdbk{k}{X}{Y}$ is indeed a minimum. 

    A similar argument applies to the pullback interleaving distance, implying that the infimum in the definition of $\pbdi{X}{Y}$ is indeed a minimum. 
\end{remark}

We have the following relation between the pullback interleaving distance and the pullback bottleneck distance, which is an immediate consequence of Theorem \ref{thm:dm=di}:

\pbisothm*

Similarly to Remark \ref{rmk:inf=min}, one can check that the suprema in the LHS and RHS above are both maxima.

\begin{proof}[Proof of Proposition \ref{cor:hatdb-hatdi}] By the isometry between $\dm$ and $\di$, we have
\begin{align*}
&\sup_{k\in \bbZ_{\geq 0}} \pbdbk{k}{X}{Y}\\
=& \sup_{k\in \bbZ_{\geq 0}} \inf\left\{   \dm
    \left(  \caB_{\Ver,k}(Z_X), \caB_{\Ver,k}(Z_Y)\right)  \mid X\xtwoheadleftarrow{\phi_X}Z\xtwoheadrightarrow{\phi_Y}Y\text{ a finite tripod}  \right\}\\
\leq& \inf\left\{   \sup_{k\in \bbZ_{\geq 0}} \dm
    \left(  \caB_{\Ver,k}(Z_X), \caB_{\Ver,k}(Z_Y)\right)  \mid X\xtwoheadleftarrow{\phi_X}Z\xtwoheadrightarrow{\phi_Y}Y\text{ a finite tripod}  \right\}\\
=& \inf\left\{   \di
    \left(  \fccVR{Z_X}, \fccVR{Z_Y}\right)  \mid X\xtwoheadleftarrow{\phi_X}Z\xtwoheadrightarrow{\phi_Y}Y\text{ a finite tripod}  \right\}\\
=&\pbdi{X}{Y}. \qedhere
\end{align*}
\end{proof}

\begin{remark}[$\hatdb\leq \dm$] Let two finite metric spaces $X$ and $Y$ have the same cardinality (in which case the $\dm$ of verbose barcodes is finite). Considering any surjective set map $X\xtwoheadrightarrow{f} Y$ and the resulting tripod $X\xtwoheadleftarrow{\id_X}X\xtwoheadrightarrow{f}Y$, we conclude that for any degree $k$,
\[\pbdbk{k}{X}{Y}\leq \dm(\caB_{\Ver,k}(X),\caB_{\Ver,k}(Y)).\]
Thus, $\sup_{k\in\bbZ_{\geq 0}}\pbdbk{k}{X}{Y}\leq  \sup_{k\in\bbZ_{\geq 0}}\dm(\caB_{\Ver,k}(X),\caB_{\Ver,k}(Y))$.
Note that this inequality can be strict. For instance, consider the four-point metric spaces $X$ and $Y$ given in Example \ref{ex:4-point space}, we claim that
\[\sup_{k\in\bbZ_{\geq 0}}\pbdbk{k}{X}{Y}=0<1= \sup_{k\in\bbZ_{\geq 0}}\dm(\caB_{\Ver,k}(X),\caB_{\Ver,k}(Y)).\]
The non-trivial part is $\sup_{k\in\bbZ_{\geq 0}}\pbdbk{k}{X}{Y}=0$. Consider the tripod $X\xtwoheadleftarrow{\phi_X}Z\xtwoheadrightarrow{\phi_Y}Y$ given by $Z=\{z_1,\dots,z_5\}$,
\[\phi_X(z_i):= \begin{cases}
    x_i,&\mbox{$1\leq i\leq 4$,}\\
    x_4,&\mbox{$i=5$.}
    \end{cases}
    \hspace{2em}
    \text{   and   }
    \hspace{2em}
\phi_Y(z_i):= \begin{cases}
    y_i,&\mbox{$1\leq i\leq 4$,}\\
    y_4,&\mbox{$i=5$.}
    \end{cases}
    \]
Let $Z_X:=(Z,\phi_X^*d_X)$ and $Z_Y:=(Z,\phi_Y^*d_Y)$. 
It follows from Proposition \ref{prop:pullback_barcode_1} which will be proved in Section \ref{sec:pb barcode} and the verbose barcodes of $X$ and $Y$ computed on page \pageref{para:verbose example} that 
\begin{itemize}
    \item $\caB_{\Ver,0}(Z_X)= \{(0,0), (0,2), (0,\infty)\}\sqcup \{(0,1)\}^ 2 = \caB_{\Ver,0}(Z_Y)$;
    \item $\caB_{\Ver,1}(Z_X) = \left\{  (1,1)\right\}\sqcup \{(2,2)\}^ 5= \caB_{\Ver,1}(Z_Y)$;
    \item $\caB_{\Ver,2}(Z_X) = \left\{ (2,2)\right\}^ 4= \caB_{\Ver,2}(Z_Y)$;
    \item $\caB_{\Ver,3}(Z_X) = \left\{ (2,2)\right\}= \caB_{\Ver,3}(Z_Y)$;
    \item $\caB_{\Ver,k}(Z_X) = \emptyset= \caB_{\Ver,k}(Z_Y)$ for $k\geq 4.$
\end{itemize}
Because $\dm(\caB_{\Ver,k}(Z_X), \caB_{\Ver,k}(Z_Y)) = 0$ for all $k$, we obtain $\pbdi{X}{Y}=0$. This implies that $\sup_{k\in\bbZ_{\geq 0}}\pbdbk{k}{X}{Y}=0$.
\end{remark}


\subsection{Pullback Stability Theorem} \label{sec:pb stab}

In this section, we prove that the pullback interleaving distance $\hatdi$ and the pullback bottleneck distance $\hatdb$ are stable under the Gromov-Hausdorff distance $\dgh$ (cf. Theorem \ref{thm:hatdb-dgh stability}) and that they provide a better lower-bound estimate of $\dgh$ in comparison with the standard bottleneck distance $\db$ (cf. Theorem \ref{thm:dgh-classical}). 

The following proposition establishes that $\hatdi$ is stable. See page \pageref{pf:hatdi-dgh stability} for the proof.

\begin{proposition} [Stability of Pullback Interleaving Distance] \label{prop:hatdi-dgh stability}
Let $(X,d_X)$ and $(Y,d_Y)$ be two finite metric spaces. Then,
\[\pbdi{X}{Y} \leq 2\cdot\dgh(X,Y).\]
\end{proposition}

Corollary \ref{cor:hatdb-hatdi} and Proposition \ref{prop:hatdi-dgh stability} together yield the stability of $\hatdb$. In addition, we prove that $\hatdb$ is an improvement over $d_B$, when both are regarded as lower bounds of the $\dgh$ between metric spaces:

\pbstabthm*

\begin{proof} 
We only need to prove $\db \left( \caB_{\Con,k}(X), \caB_{\Con,k}(Y)\right) \leq  \pbdbk{k}{X}{Y} $. For any tripod $X\xtwoheadleftarrow{\phi_X}Z\xtwoheadrightarrow{\phi_Y}Y$, let $Z_X:=(Z,\phi_X^*d_X)$ and $Z_Y:=(Z,\phi_Y^*d_Y)$. By Proposition \ref{prop:pullback-barcode} and the fact that concise barcodes can be obtained from the corresponding verbose barcode minus all the diagonal points, we have that $\caB_{\Con,k}(X)=\caB_{\Con,k}(Z_X)$ and $\caB_{\Con,k}(Y)=\caB_{\Con,k}(Z_Y)$. 
Incorporating Proposition \ref{prop:dm-db}, we deduce:
\[\db \left( \caB_{\Con,k}(X), \caB_{\Con,k}(Y)\right)=\db \left( \caB_{\Con,k}(Z_X), \caB_{\Con,k}(Z_Y)\right) \leq  \dm\left( \caB_{\Ver,k}(Z_X), \caB_{\Ver,k}(Z_Y)\right). \qedhere \]
\end{proof}

\paragraph{The proof of Proposition \ref{prop:hatdi-dgh stability}.} 
In order to prove Proposition \ref{prop:hatdi-dgh stability}, we first establish the stability of the interleaving distance $\di$ between Vietoris-Rips FCCs by showing that $\di$ is stable under the max norm between the two distance functions over the same underlying set. Recall the definition of the distortion of a map from page \pageref{para:dis(f)}.  

\begin{proposition}\label{prop:infty-stability}
Let $X$ be a finite set of cardinality $n$. Let $d_1$ and $d_2$ be two distance functions on $X$, and let $\ell_1$ and $\ell_2$ be the filtration functions induced by $d_1$ and $d_2$ respectively. Then,
\[|\|d_1\|_{\infty}-\|d_2\|_{\infty}|\leq 
\di\left(  (\opC_*(\Full(X)),\ell_1),(\opC_*(\Full(X)),\ell_2) \right) \leq \min_{\mathrm{bijection }f:X\to X}\dis(f)\leq \|d_1-d_2\|_{\infty}.\]
\end{proposition}

\begin{proof}
\begin{claim}\label{lem:l norm = d norm}
$\|\ell_1-\ell_2\|_{\infty}=\|d_1-d_2\|_{\infty}$ and $\|\ell_1\|_{\infty}=\|d_1\|_{\infty}$.
\end{claim}

When $X$ is an one-point space, this is trivial to prove. Now assume that $X$ has at least two points and suppose that $\|d_1-d_2\|_{\infty}=|d_1(x_1,x_2)-d_2(x_1,x_2)|$ for some $x_1,x_2\in X$. Then,
\[\|\ell_1-\ell_2\|_{\infty}\geq |\ell_1([x_1,x_2])-\ell_2([x_1,x_2])|=|d_1(x_1,x_2)-d_2(x_1,x_2)| = \|d_1-d_2\|_{\infty}.\]
\revision{Conversely, for any chain $\sigma = \sum_i \lambda_i \sigma_i$, expressed as a linear combination of simplices $\sigma_i$ in $X$ (not necessarily all of the same degree), we have
\begin{align*}
|\ell_1(\sigma)-\ell_2(\sigma)|
=& \Big|\ell_1\Big(\sum_{i}\lambda_i \sigma_i\Big)-\ell_2\Big(\sum_{i}\lambda_i \sigma_i\Big)\Big|=\Big|\max_{\lambda_i \neq 0}\ell_1(\sigma_i)-\max_{\lambda_i \neq 0}\ell_2(\sigma_i)\Big|\\
\leq & \max_{\lambda_i \neq 0} |\ell_1(\sigma_i)-\ell_2(\sigma_i)|=\max_{\lambda_i \neq 0} |\diam_1(\sigma_i)-\diam_2(\sigma_i)| \leq \|d_1-d_2\|_{\infty}.
\end{align*}
Taking $d_2=0$, the trivial distance function, we obtain $\|\ell_1\|_{\infty}=\|d_1\|_{\infty}$.}

\begin{claim}\label{lem:filtration under bij}
Consider any bijection $f:X \rightarrow X$, and let $d$ be a metric on $X$. Define $d\circ f$ as the composition $d\circ(f,f)$, and denote $\tilde{X} := (X, d\circ f)$. Let $\ell$ and $\tilde{\ell}$ represent the filtration functions induced by $d$ and $d\circ f$, respectively. Then $\tilde{\ell}=\ell\circ f.$
\end{claim}

Clearly, $f$ induces a chain isomorphism $f:\opC_*(\Full(X))\xrightarrow{\cong}\opC_*(\Full(X))$. Then $\tilde{X}$ is a metric space, whose filtration function for the Vietoris-Rips FCC is given by
\[ \tilde{\ell} \left(  \sum_{i=1}^r \lambda_i\sigma_i \right)  = \max_{\lambda_i\neq 0}\diam_{\tilde{X}}(\sigma_i)=\max_{\lambda_i\neq 0}\diam_{X}(f(\sigma_i)),\]
where $\sigma_1,\dots,\sigma_r$ are simplices. Since $f(\sigma_i)=\sigma_{j_i}$ for some simplex $\sigma_{j_i}$, we have
\[\ell\circ f \left(  \sum_{i=1}^r \lambda_i\sigma_i \right) =\ell\left(  \sum_{i=1}^r \lambda_i \sigma_{j_i} \right) =\max_{\lambda_i\neq 0} \diam_X(\sigma_{j_i})=\max_{\lambda_i\neq 0}\diam_{X}(f(\sigma_i))=\tilde{\ell} \left(  \sum_{i=1}^r \lambda_i\sigma_i \right).\]
\smallskip

Let $f:X\to X$ be any bijection. By Claim \ref{lem:l norm = d norm} and Claim \ref{lem:filtration under bij}, we have
\[\|\ell_1-\ell_2\circ f\|_{\infty}=\|\ell_1-\ell^{(X,d_2\circ f)}\|_{\infty}= \|d_1-d_2\circ f\|_{\infty}=\dis(f).\]
Therefore, by Theorem \ref{thm:di of same chain}, we have
\begin{align*}
    \di\left(  (\opC_*(\Full(X)),\ell_1),(\opC_*(\Full(X)),\ell_2) \right) 
    &= \inf_{f\in \Iso(\opC_*(\Full(X)))} \|\ell_1-\ell_2\circ f\|_{\infty}\\
    &\leq  \min_{f:X\xrightarrow{\text{bij.}} X} \|d_1-d_2\circ f\|_{\infty}\\
    & = \min_{f:X\xrightarrow{\text{bij.}} X}\dis(f)\leq \|d_1-d_2\|_{\infty}.
\end{align*}

On the other hand, for any $f\in \Iso(\opC_*(\Full(X))) $, we have $\|\ell_2\circ f\|_{\infty} = \|\ell_2\|_{\infty}$. Thus, by Claim \ref{lem:l norm = d norm},
\begin{align*}
    \di\left(  (\opC_*(\Full(X)),\ell_1),(\opC_*(\Full(X)),\ell_2) \right) 
    &\,=\, \inf_{f\in \Iso(\opC_*(\Full(X)))} \|\ell_1-\ell_2\circ f\|_{\infty}\\
    &\,\geq\,  \big|\|\ell_1\|_{\infty}-\|\ell_2\|_{\infty}\big|= |\|d_1\|_{\infty}-\|d_2\|_{\infty}|. \qedhere
\end{align*}
\end{proof}

\begin{proof}[Proof of Proposition \ref{prop:hatdi-dgh stability}]
\label{pf:hatdi-dgh stability}
Suppose 
$R:X\xtwoheadleftarrow{\phi_X}Z\xtwoheadrightarrow{\phi_Y}Y$
is a finite tripod between $X$ and $Y$. By Proposition \ref{prop:infty-stability}, we obtain
\[\di\left(  \left( \opC_*\!\left(\Full\left(Z\right)\right),\ell^{Z_X}\right) ,\left( \opC_*\!\left(\Full\left(Z\right)\right),\ell^{Z_Y}\right)  \right) \leq \|\phi_X^*d_X-\phi_Y^*d_Y\|_{\infty}=\dis(R).
\]
We finish the proof, by taking infimum over all finite tripods $R$ in the above inequality and applying the fact that $2\cdot \dgh(X,Y)=\inf_{\text{finite }R} \,\dis(R)$, by Remark \ref{rmk:finite tripod}.
\end{proof}

\begin{example} [Strictness of the inequalities in Proposition \ref{prop:infty-stability}] Let $X=\left\{  a,b,c\right\}   $. Consider pseudo-metrics $d_0,d_1$ and $d_2$ on $X$, given in Figure \ref{fig:3-point}. We calculate the following quantities for $d_0$ and $d_1$: 
\begin{enumerate}
    \item $|\|d_0\|_{\infty}-\|d_1\|_{\infty}|=1;$
    \item $\di\left(  (\fccVR{X,d_0},\fccVR{X,d_1} \right) =1$.
    \item $\min_{\mathrm{bijection }f:(X,d_0)\to (X,d_1)}\dis(f)=1;$
    \item $\|d_0-d_1\|_{\infty}=2.$
\end{enumerate}
Similarly, we compute for $d_1$ and $d_2$: 
\begin{enumerate}
    \item $|\|d_1\|_{\infty}-\|d_2\|_{\infty}|=0;$
    \item $\di\left(  (\fccVR{X,d_1},\fccVR{X,d_2} \right) =1$;
    \item $\min_{\mathrm{bijection }f:(X,d_1)\to (X,d_2)}\dis(f)=1;$
    \item $\|d_1-d_2\|_{\infty}=1.$
\end{enumerate}

\begin{figure}[ht]
\centering
\begin{tikzpicture}[scale=1]
\node (a) [label=above:$a$] at (1.5,1) {};
\node (b) [label=left:$b$] at (1,-0.5) {};
\node (c) [label=right:$c$] at (3,-0.5) {};
\filldraw (a) [color=orange] circle[radius=2pt];
\filldraw (b)[color=orange] circle[radius=2pt];
\filldraw (c) [color=orange] circle[radius=2pt];
\draw [dashed] (a) --(b) node[pos=0.5,fill=white]{\textcolor{red}{$0$}};
\draw [dashed] (a) --(c) node[pos=0.5,fill=white]{\textcolor{red}{$1$}};
\draw [dashed]  (b) --(c)node[pos=0.5,fill=white]{\textcolor{red}{$1$}};
\end{tikzpicture} 
\hspace{2em}
\begin{tikzpicture}[scale=1]
\node (a) [label=above:$a$] at (0,2) {};
\node (b) [label=left:$b$] at (-1,-0.5) {};
\node (c) [label=right:$c$] at (,-0.5) {};

\filldraw (a) [color=blue] circle[radius=2pt];
\filldraw (c)[color=blue] circle[radius=2pt];
\filldraw (b) [color=blue] circle[radius=2pt];
\draw [dashed]  (a)--(c) node[pos=0.6,fill=white]{\textcolor{red}{$2$}}; 
\draw [dashed] (a) --(b) node[pos=0.6,fill=white]{\textcolor{red}{$2$}};
\draw [dashed]  (c) --(b)node[pos=0.5,fill=white]{\textcolor{red}{$1$}};
\end{tikzpicture} 
\hspace{2em}
\begin{tikzpicture}[scale=1]
\node (a) [label=above:$a$] at (0,2) {};
\node (b) [label=left:$b$] at (-1,-0.5) {};
\node (c) [label=right:$c$] at (1,-0.5) {};

\filldraw (a) [color=green] circle[radius=2pt];
\filldraw (c)[color=green] circle[radius=2pt];
\filldraw (b) [color=green] circle[radius=2pt];
\draw [dashed]  (a)--(c) node[pos=0.6,fill=white]{\textcolor{red}{$2$}}; 
\draw [dashed] (a) --(b) node[pos=0.6,fill=white]{\textcolor{red}{$2$}};
\draw [dashed]  (c) --(b)node[pos=0.5,fill=white]{\textcolor{red}{$0$}};
\end{tikzpicture} 
    
\caption{Pseudo-metrics $d_0,d_1,d_2$.} 
\label{fig:3-point}
\end{figure}
\end{example}


\subsection{Tightness and Strictness of the Pullback Stability Theorem}\label{sec:ex for strict}

We study some examples for Theorem \ref{thm:hatdb-dgh stability} and see that both inequalities in this theorem are tight, and they can be strict too.

\begin{example} \label{ex:hatdb-3-point space}
Recall the $3$-point metric spaces $X_1$ and $X_2$ from Figure \ref{fig:3-point-barcode-full}, and assume that $a\leq b\leq c_i$ for $i=1,2.$ 
Computing each of the distances appearing in Theorem \ref{thm:hatdb-dgh stability}, we obtain:
\begin{center}
\renewcommand{\arraystretch}{1.4}
\begin{tabular}{ |c| c| c| c| } 
\hline
$\sup\limits_{k\in\bbZ_{\geq 0}}\db\left(  \caB_{\Con,k}(X_1), \caB_{\Con,k}(X_2)\right) $
& $ \sup\limits_{k\in\bbZ_{\geq 0}}\pbdbk{k}{X_1}{X_2} $
&$2\cdot\dgh(X_1,X_2)$
\\ 
\hline
$0$ & $|c_1-c_2|$ 
& $|c_1-c_2|$ \\ 
\hline
\end{tabular}
\end{center}
The first and third columns in the above table arise from straightforward calculations. For the second column, notice that for any tripod $X_1\xtwoheadleftarrow{\phi_1}Z\xtwoheadrightarrow{\phi_2}X_2$ with $\card(Z)=m+3$ for some non-negative integer $m$, we have
\[\caB_{\Ver,m+1}(Z_1)=\{(c_1,c_1)\} \quad  \text{and} \quad \caB_{\Ver,m+1}(Z_2)=\{(c_2,c_2)\},\]
where $Z_1:=(Z,\phi_1^*d_{X_1})$ and $Z_2:=(Z,\phi_2^*d_{X_2})$. In particular, $\pbdbk{m+1}{Z_1}{Z_2} = |c_1-c_2|$, for any tripod $X_1\xtwoheadleftarrow{\phi_1}Z\xtwoheadrightarrow{\phi_2}X_2$. Thus, $\sup_{k\in\bbZ_{\geq 0}}\pbdbk{k}{X_1}{X_2} = |c_1-c_2|.$ 
\end{example}

\begin{example} \label{ex:hatdb-4-point space}
Let $X,Y,Z$, and $W$ be metric spaces each consisting of $4$ points where the respective metric are depicted in Figure \ref{fig:verbose-4-point}, together with the verbose barcodes of each of these metric spaces. 
Among these spaces, $X$ and $Y$ have been studied in Example \ref{ex:4-point space}. 
The space $Z$ is the complete graph on $4$ vertices with edge length $1$, and $W$ is the cycle graph on $4$ vertices with edge length $1$. Both $Z$ and $W$ are equipped with the graph distance.

\begin{figure}[ht]
    \centering
\renewcommand{\arraystretch}{1.2}
\begin{tabular}{ | c| c| c| c| c| } 
\hline

&
\begin{tikzpicture}[scale=1]
\node (a) at (-2,1.8) {};
\node (b) at (-1,-0.5) {};
\node (c) at (-3,-0.5) {};
\node (d) at (-2.2,-1) {};
\filldraw (a) [color=blue] circle[radius=2pt];
\filldraw (b)[color=blue] circle[radius=2pt];
\filldraw (c) [color=blue] circle[radius=2pt];
\filldraw (d) [color=blue] circle[radius=2pt];
\draw [dashed]  (a)--(b) node[pos=0.6,fill=white]{\textcolor{red}{$2$}}; 
\draw [dashed] (a) --(c) node[pos=0.6,fill=white]{\textcolor{red}{$2$}};
\draw [dashed]  (a)--(d) node[pos=0.6,fill=white]{\textcolor{red}{$2$}};
\draw [dashed]  (b) --(c)node[above,pos=0.4,fill=white]{\textcolor{red}{$1$}};
\draw [dashed] (c)--(d)node[below left,pos=0.5,fill=white]{\textcolor{red}{$1$}};
\draw [dashed]  (b)--(d)node[below,pos=0.5,fill=white]{\textcolor{red}{$1$}};
\node at (-2,-1.7) {$X$};
\end{tikzpicture} 
& 
\begin{tikzpicture}[scale=1]
\node (e) at (2.5,2) {};
\node (f) at (1,-0.5) {};
\node (g) at (3,-0.5) {};
\node (h) at (1.5,1.5) {};
\filldraw (e) [color=orange] circle[radius=2pt];
\filldraw (f)[color=orange] circle[radius=2pt];
\filldraw (g) [color=orange] circle[radius=2pt];
\filldraw (h) [color=orange] circle[radius=2pt]; 
\draw [dashed]  (f)--(h)node[pos=0.7,fill=white]{\textcolor{red}{$2$}};
\draw [dashed]  (e)--(f) node[pos=0.6,fill=white]{\textcolor{red}{$2$}};
\draw [dashed] (g)--(h)node[pos=0.5,fill=white]{\textcolor{red}{$2$}};
\draw [dashed] (e) --(g) node[pos=0.5,fill=white]{\textcolor{red}{$2$}};
\draw [dashed] (e)--(h) node[above,pos=0.6,fill=white]{\textcolor{red}{$1$}};
\draw [dashed]  (1,-0.5) --(g)node[pos=0.5,fill=white]{\textcolor{red}{$1$}};
\node at (2,-1.2) {$Y$};
\end{tikzpicture} 
&
\begin{tikzpicture}[scale=1]
\node (a) at (1.35,1) {};
\filldraw (2.5,1.5) [color=green] circle[radius=2pt];
\filldraw (1,-0.5)[color=green] circle[radius=2pt];
\filldraw (3,-0.5) [color=green] circle[radius=2pt];
\filldraw (a) [color=green] circle[radius=2pt]; 
\draw [dashed]  (f)--(a)node[pos=0.7,fill=white]{\textcolor{red}{$1$}};
\draw [dashed]  (2.5,1.5)--(1,-0.5) node[pos=0.6,fill=white]{\textcolor{red}{$1$}};
\draw [dashed] (3,-0.5)--(a)node[pos=0.5,fill=white]{\textcolor{red}{$1$}};
\draw [dashed] (2.5,1.5) --(3,-0.5) node[pos=0.5,fill=white]{\textcolor{red}{$1$}};
\draw [dashed] (2.5,1.5)--(a) node[above,pos=0.6,fill=white]{\textcolor{red}{$1$}};
\draw [dashed]  (1,-0.5) --(3,-0.5)node[pos=0.5,fill=white]{\textcolor{red}{$1$}};
\node at (2,-1.2) {$Z$};
\end{tikzpicture} 
&
\begin{tikzpicture}[scale=1]
\node (a) at (1,1) {};
\node (b) at (3,1) {};
\node (c) at (1,3) {};
\node (d) at (3,3) {};
\filldraw (a) [color=magenta] circle[radius=2pt];
\filldraw (b)[color=magenta] circle[radius=2pt];
\filldraw (c) [color=magenta] circle[radius=2pt];
\filldraw (d) [color=magenta] circle[radius=2pt]; 
\draw [dashed]  (a)--(b) node[pos=0.5,fill=white]{\textcolor{red}{$1$}};
\draw [dashed]  (d)--(b) node[pos=0.5,fill=white]{\textcolor{red}{$1$}};
\draw [dashed] (c)--(d)node[pos=0.5,fill=white]{\textcolor{red}{$1$}};
\draw [dashed] (a) --(c) node[pos=0.5,fill=white]{\textcolor{red}{$1$}};
\draw [dashed] (a) --(d) node[pos=0.5,fill=white]{\textcolor{red}{$2$}};
\draw [dashed] (b) --(c) node[pos=0.5,fill=white]{\textcolor{red}{$2$}};
\node at (2,0.3) {$W$};
\end{tikzpicture} 
\\ 
\hline
$\caB_{\Ver,0}$ 
& $(0,1)^2,(0,2), (0,\infty)$ 
& $(0,1)^2,(0,2), (0,\infty)$ 
& $(0,1)^3, (0,\infty)$ 
& $(0,1)^3, (0,\infty)$ \\ 
\hline
$\caB_{\Ver,1}$ 
& $(1,1), (2,2)^2$ 
& $(2,2)^3$ 
& $(1,1)^3$ 
& $(1,2), (2,2)^2$\\ 
\hline
$\caB_{\Ver,2}$ 
& $(2,2)$ 
& $(2,2)$ 
& $(1,1)$ 
& $(2,2)$ \\ 
\hline
\end{tabular}
    \caption{Verbose barcodes of $4$-point metric spaces $X,Y,Z$ and $W$.}
    \label{fig:verbose-4-point}
\end{figure}

\begin{figure}[ht]
    \centering
\begin{tabular}{  c| c c c c } 
$\db(\caB_{\Con,0}(\cdot),\caB_{\Con,0}(\cdot))$
&
$X$
& 
$Y$
&
$Z$
&
$W$
\\ 
\hline
$X$ 
& $0$ 
& $0$ 
& $1$ 
& $1$ \\ 
$Y$ 
&  
& $0$ 
& $1$ 
& $1$ \\ 
$Z$ 
&  
&  
& $0$ 
& $0$ \\ 
$W$ 
&  
&  
&  
& $0$ \\ 
\end{tabular}
\hspace{1.2cm}
\begin{tabular}{  c| c c c c } 
$\pbdbk{0}{\cdot}{\cdot}$
&
$X$
& 
$Y$
&
$Z$
&
$W$
\\ 
\hline
$X$ 
& $0$ 
& $0$ 
& $1$ 
& $1$ \\ 
$Y$ 
&  
& $0$ 
& $1$ 
& $1$ \\ 
$Z$ 
&  
&  
& $0$ 
& $0$ \\ 
$W$ 
&  
&  
&  
& $0$ \\ 
\end{tabular}\\
\vspace{0.5cm}
\begin{tabular}{  c| c c c c } 
$\db(\caB_{\Con,1}(\cdot),\caB_{\Con,1}(\cdot))$
&
$X$
& 
$Y$
&
$Z$
&
$W$
\\ 
\hline
$X$ 
& $0$ 
& $0$ 
& $0$ 
& $\frac{1}{2}$ \\ 
$Y$ 
&  
& $0$ 
& $0$  
& $\frac{1}{2}$ \\ 
$Z$ 
&  
&  
& $0$ 
& $\frac{1}{2}$  \\ 
$W$ 
&  
&  
&  
& $0$ \\ 
\end{tabular}
\hspace{1.2cm}
\begin{tabular}{  c| c c c c } 
$\pbdbk{1}{\cdot}{\cdot}$
&
$X$
& 
$Y$
&
$Z$
&
$W$
\\ 
\hline
$X$ 
& $0$ 
& $0$ 
& $1$ 
& $1$ \\ 
$Y$ 
&  
& $0$ 
& $1$ 
& $1$ \\ 
$Z$ 
&  
&  
& $0$ 
& $1$ \\ 
$W$ 
&  
&  
&  
& $0$ \\ 
\end{tabular}\\
\vspace{0.5cm}
\begin{tabular}{  c| c c c c } 
$2\cdot \dgh(\cdot, \cdot)$
&
$X$
& 
$Y$
&
$Z$
&
$W$
\\ 
\hline
$X$ 
& $0$ 
& $1$ 
& $1$ 
& $1$ \\ 
$Y$ 
&  
& $0$ 
& $1$ 
& $1$ \\ 
$Z$ 
&  
&  
& $0$ 
& $1$ \\ 
$W$ 
&  
&  
&  
& $0$ \\ 
\end{tabular}
    \caption{The bottleneck distance $\db$ between concise barcodes, the pullback bottleneck distance $\hatdb$, the pullback interleaving distance $\hatdi$ and the Gromov-Hausdorff distance between spaces.}
    \label{fig:metric-4-point}
\end{figure}

From Figure \ref{fig:metric-4-point}, we notice that the pair of metric spaces $(X,Y)$ is such that
\[\sup_{k\in\bbZ_{\geq 0}}\dballk{X}{Y}=\sup_{k\in\bbZ_{\geq 0}}\pbdbk{k}{X}{Y}
=0<1= 2\cdot\dgh(X,Y),\]
which establishes that $\hatdb$ between non-isometric spaces can be zero. 
To see that $\pbdbk{1}{X}{Y}=0$, consider pullback spaces $Z_X:=X\sqcup\{x_0\}$ where $x_0$ is a duplicate of the top vertex in $X$ and $Z_Y=Y\sqcup\{y_0\}$ where $y_0$ is a duplicate of an arbitrary point from $Y$, 
and verify that $\caB_{\Ver,1}(Z_X)=\caB_{\Ver,1}(Z_Y) = \{(1,1),(2,2)^5\}$. The pair $(X,Y)$ shows the tightness of $\db\leq \hatdb.$

The pair $(Z,W)$ is such that 
\[\sup_{k\in\bbZ_{\geq 0}}\dballk{Z}{W}=\frac{1}{2} < 1 =\sup_{k\in\bbZ_{\geq 0}}\pbdbk{k}{Z}{W}
= 2\cdot\dgh(Z,W),\]
which gives another example of $\hatdb$ and $\hatdi$ providing better bounds for $\dgh$ in comparison with the standard bottleneck distance $\db$.
\end{example}

Below is another example in which the stability of $\hatdb$ improves that of $\db$:
\begin{example} \label{ex: one-point space} Let $X$ be the one-point metric space. Let $Y=\Delta_n(\epsilon)$ be the $n$-point metric space where all points are at distance $\epsilon>0$ from each other, for $n\geq 2$. Then,
\[2\cdot\dgh(X,Y)=\epsilon.\]
For any tripod 
$X\xtwoheadleftarrow{\phi_X}Z\xtwoheadrightarrow{\phi_Y}Y,$
we have
\begin{itemize}
    \item $\caB_{\Ver}(Z_X)$ consists of only copies of $(0,0)$ in all degrees; 
    \item $\caB_{\Ver}(Z_Y)$ consists of copies of $(0,0)$, $(0,\epsilon)$ and $(\epsilon,\epsilon)$, and $\caB_{\Ver,0}(Z_Y)$ contains copies of $(0,\epsilon)$,
\end{itemize} 
It is not hard to verify that
\[\sup_{k\in\bbZ_{\geq 0}}\dm\left(  \caB_{\Ver,k}(Z_X), \caB_{\Ver,k}(Z_Y)\right) =\epsilon\]
for any tripod. Thus, $\sup_{k\in\bbZ_{\geq 0}} \pbdbk{k}{X}{Y} =\epsilon.$ In addition, we have the following table
\begin{center}
\renewcommand{\arraystretch}{1.4}
\begin{tabular}{ |c| c| c| c| } 
\hline
$\sup\limits_{k\in\bbZ_{\geq 0}}\db\left(  \caB_{\Con,k}(X), \caB_{\Con,k}(Y)\right) $
& $ \sup\limits_{k\in\bbZ_{\geq 0}} \pbdbk{k}{X}{Y} $
&$2\cdot\dgh(X,Y)$
\\ 
\hline
$\frac{\epsilon}{2}$ & $\epsilon$ 
& $\epsilon$ \\ 
\hline
\end{tabular}
\end{center}

\end{example}

Via a similar argument and by invoking the fact that $\diam(Y)\cdot(1,1)\in \caB_{\Ver,\card(Y)-2}(Z_Y)$ for any pullback space $Z_Y$ of $Y$, we generalize Example \ref{ex: one-point space} to the following proposition:
\begin{proposition}\label{prop: one-point space} Let $X$ be the one-point metric space, and $Y$ be any finite metric space. Then, \[\sup_{k\in\bbZ_{\geq 0}} \pbdbk{k}{X}{Y} =\pbdi{X}{Y} = \diam (Y) = 2\cdot \dgh(X,Y).\]
\end{proposition}

\subsection{Variations of the Pullback Interleaving/Bottleneck Distance} \label{sec:variation of pb distances}

In previous subsections, we introduced the pullback interleaving distance $\hatdi$ and the pullback bottleneck distance $\hatdb$ based on the notion of tripod. To highlight the role of tripods and facilitate comparisons with other variants, we sometimes write 
\[\hatdi^{\Tri} := \hatdi\text{ and }\hatdb^{\Tri} := \hatdb,\] respectively. Given a degree $k$, we will use $\dbk{k}$ to denote $\db$ between degree-$k$ concise barcodes.

We introduce two variants of the pullback interleaving/bottleneck distance, mirroring the equivalent definitions of the Gromov-Hausdorff distance (see Section \ref{sec:preliminaries}).
In the first variant, we employ correspondences between metric spaces, instead of tripods, and we denote the resulting distances as $\hatdi^{\Cor}$ and $\hatdb^{\Cor}$, as defined in Definition \ref{def:pb di R}. In the second variant, we define the distances $\hatdi^{\Map}$ and $\hatdb^{\Map}$ utilizing maps between the underlying metric spaces, as specified in Definition \ref{def:pb di M}. These new formulations are beneficial in terms of computational efficiency, which we will discuss in more detail in Section \ref{subsec:multi vector}. 

To simplify our terminology, we will use the following terminology:
\begin{itemize}
    \item `\textbf{pullback interleaving-type distances}' refers to all versions of pullback interleaving distances.
    \item `\textbf{pullback bottleneck-type distances}' refers to all versions of pullback bottleneck distances.
    \item `\textbf{pullback distances}' refers to all pullback interleaving-type distances and pullback bottleneck-type distances.
\end{itemize}

We show that pullback distances are all stable under the Gromov-Hausdorff distance, and they improve upon the stability of the bottleneck distance between the concise barcodes. Moreover, they satisfy the relation below. 

\begin{theorem}\label{thm:all hatdi and hatdb}
The several variants of pullback interleaving and bottleneck distances satisfy the following relations:
\[
\begin{tikzcd}
&
\hatdi^{\Tri} \ar[r,symbol=\leq]
& 
\hatdi^{\Cor} \ar[r,symbol=\leq]
& 
\hatdi^{\Map} \ar[r,symbol=\lneq]
& 
2\cdot \dgh
\\
\sup\limits_k\dbk{k} \ar[r,symbol=\lneq]
&
\sup\limits_k\hatdbk{k}^{\Tri} \ar[r,symbol=\leq]
\ar[u,symbol=\lneq]
& 
\sup\limits_k\hatdbk{k}^{\Cor} \ar[r,symbol=\leq]
\ar[u,symbol=\lneq]
& 
\sup\limits_k\hatdbk{k}^{\Map}
\ar[u,symbol=\lneq]
&
\end{tikzcd}
\]
where `$\lneq$' indicates that (1) `$\leq$' always holds and (2) there exist examples for which the inequality is strict.

Moreover, the above table remains valid if we fix a degree $k$ and replace the last row with  
$\dbk{k} \lneq\hatdbk{k}^{\Tri}\leq\hatdbk{k}^{\Cor}\leq\hatdbk{k}^{\Map}.$
\end{theorem}

Recall from Section \ref{sec:preliminaries} that the Gromov-Hausdorff distance has several equivalent definitions:
\begin{itemize}
    \item using maps: 
    \[
    \dgh(X,Y)=\frac{1}{2}\inf_{\substack{f:X\rightarrow Y\\ g :Y\rightarrow X}}\max\{\dis(f),\dis(g),\codis(f,g)\};
    \]
    \item using correspondences: 
    \[\dgh(X,Y)=\frac{1}{2}\inf_{R\in \mathfrak{R}(X,Y)}\dis(R);\]
    \item using tripods: 
    \[\dgh(X,Y)=\frac{1}{2}\inf_{X\xtwoheadleftarrow{\phi_X}Z\xtwoheadrightarrow{\phi_Y}Y}\dis((Z,\phi_X,\phi_Y)).\]
\end{itemize}

Let $R\subset X\times Y$ be a correspondence between metric spaces $X$ and $Y$, and let $\pi_X$ and $\pi_Y$ be the two projection maps onto $X$ and $Y$, respectively. We equip $R$ with the respective pullback metrics induced by $\pi_X$ and $\pi_Y$, and denote 
\begin{equation}\label{eq:R_X and R_Y}
    R_X:=(R,\pi_X^*d_X) \quad \text{and} \quad R_Y:=(R,\pi_Y^*d_Y).
\end{equation}

\begin{definition} 
\label{def:pb di R}
For two finite metric spaces $X$ and $Y$, we define the \textbf{pullback interleaving distance induced by correspondences (and VR FCCs)} between $X$ and $Y$ to be
\begin{align*}
    \pbdivariant{\hat}{\Cor}{X}{Y}  :=\min\left\{   \di\left(  \fccVR{R_X} ,\fccVR{R_Y}  \right)   \mid R\subset X\times Y\text{ a correspondence}  \right\}   ,
\end{align*}
where $R_X:=(R,\pi_X^*d_X)$ and $R_Y:=(R,\pi_Y^*d_Y)$.

For any degree $k\in \bbZ_{\geq 0}$, we define the \textbf{pullback bottleneck distance induced by correspondences (and degree-$k$ verbose barcodes)} between $X$ and $Y$ to be
\begin{align*}
    \pbdbvariant{k}{\hat}{\Cor}{X}{Y}  :=\min\left\{   \dm\left(  \caB_{\Ver,k}(R_X), \caB_{\Ver,k}(R_Y)  \right)   \mid R\subset X\times Y\text{ a correspondence}  \right\}.
\end{align*}
\end{definition}

For two maps $f:X\to Y$ and $g:Y\to X$, define a multiset arising in the use of the graphs of $f$ and $g$
\begin{equation} \label{eq: G(f,g)}
    G(f,g):=\{(x,f(x))\mid x\in X\}\cup \{(g(y),y)\mid y\in Y\}.
\end{equation}
For simplicity, write $G:=G(f,g)$. Note that $G$ is a correspondence between $X$ and $Y$.
Let $\pi_X$ and $\pi_Y$ be the two projection maps from $G$ onto $X$ and $Y$, respectively. We equip $G$ with the respective pullback metrics induced by $\pi_X$ and $\pi_Y$, and denote 
\begin{equation}\label{eq:G_X and G_Y}
    G_X:=(G,\pi_X^*d_X)\quad \text{and} \quad G_Y:=(G,\pi_Y^*d_Y).
\end{equation}

\begin{definition} 
\label{def:pb di M}
For two finite metric spaces $X$ and $Y$, we define the \textbf{pullback interleaving distance induced by maps (and VR FCCs)} between $X$ and $Y$ to be
\begin{align*}
    \pbdivariant{\hat}{\Map}{X}{Y} :=\min\left\{   \di\left(  \fccVR{G_X} ,\fccVR{G_Y}  \right)   \mid f:X\to Y, g:Y\to X, G=G(f,g) \right\}   ,
\end{align*}
where $G_X:=(G,\pi_X^*d_X)$ and $G_Y:=(G,\pi_Y^*d_Y)$.

For any degree $k\in \bbZ_{\geq 0}$, we define the \textbf{pullback bottleneck distance induced by maps (and degree-$k$ verbose barcodes)} between $X$ and $Y$ to be
\begin{align*}
    \pbdbvariant{k}{\hat}{\Map}{X}{Y}  :=\min\left\{   \dm\left(  \caB_{\Ver,k}(G_X), \caB_{\Ver,k}(G_Y)  \right)   \mid f:X\to Y, g:Y\to X, G=G(f,g) \right\}   ,
\end{align*}
\end{definition}

\begin{proof}[Proof of Theorem \ref{thm:all hatdi and hatdb}]\label{pf:all pb distances}
We first prove that all the following inequalities hold for every degree $k$:
\[
\begin{tikzcd}
&
\hatdi^{\Tri} \ar[r,symbol=\leq,"(1)"]
& 
\hatdi^{\Cor} \ar[r,symbol=\leq,"(2)"]
& 
\hatdi^{\Map} \ar[r,symbol=\leq,"(3)"]
& 
2\cdot \dgh
\\
\dbk{k} \ar[r,symbol=\leq,"(7)"]
&
\hatdbk{k}^{\Tri} \ar[r,symbol=\leq,"(8)"]
\ar[u,symbol=\leq,"(4)\,\,"]
& 
\hatdbk{k}^{\Cor} \ar[r,symbol=\leq,"(9)"]
\ar[u,symbol=\leq,"(5)\,\,"]
& 
\hatdbk{k}^{\Map},
\ar[u,symbol=\leq,"(6)\,\,"]
&
\end{tikzcd}
\]
Inequalities (4) and (7) have been established in Corollary \ref{cor:hatdb-hatdi} and Theorem \ref{thm:hatdb-dgh stability}, respectively.
Inequalities (5) and (6) follow directly from the definitions of the pullback distances and the isometry theorem between the interleaving distance and the bottleneck distance (see Theorem \ref{thm:dm=di}).

For (1), consider any two finite metric spaces $X$ and $Y$. Note that any correspondence $R\subset X\times Y$ induces a tripod of the form 
\[X\xtwoheadleftarrow{\pi_X}R\xtwoheadrightarrow{\pi_Y}Y\] 
where $R_X = (R, \pi_X^*d_X)$ and $R_Y = (R, \pi_Y^*d_Y)$. Therefore, we have
\[\hatdi^{\Tri}\leq
\di \left( \fccVR{R_X}, \fccVR{R_Y}\right). \]
Taking the infimum over all correspondences $R$ yields $\hatdi^{\Tri} \leq \hatdi^{\Cor}$. Inequality (8) can be established similarly.

For (2), note that for any $f:X\to Y$ and $g:Y\to X$,
let $G:=G(f,g)$, and $\pi_X,\pi_Y$ be the projection from $G$ to $X$, $Y$, respectively.
Let $G_X = (G, \pi_X^*d_X)$ and $G_Y = (G, \pi_Y^*d_Y)$. Since $G$ is a correspondence between $X$ and $Y$, we have 
\[\hatdi^{\Cor}\leq
\di  \left( \fccVR{G_X}, \fccVR{G_Y}\right). \]
Since the maps $f:X\to Y$ and $g:Y\to X$ are arbitrary, we conclude that $\hatdi^{\Cor} \leq \hatdi^{\Map}$. Inequality (9) can be established similarly.

For (3), using the above notation $f,g$ and $G$, we apply Proposition \ref{prop:infty-stability} to obtain that
\[\hatdi^{\Map}\leq
\di\left( \fccVR{G_X}, \fccVR{G_Y}\right) \leq \dis(G) =\max\{\dis(f),\dis(g),\codis(f,g)\}. \]
Taking the infimum over all maps $f:X\to Y$ and $g:Y\to X$ yields $\hatdi^{\Map}\leq 2\cdot \dgh.$

Thus, we have proved all inequalities (1)-(9). By taking the supremum over all degrees $k$, we obtain the desired inequalities
\[
\begin{tikzcd}
&
\hatdi^{\Tri} \ar[r,symbol=\leq]
& 
\hatdi^{\Cor} \ar[r,symbol=\leq]
& 
\hatdi^{\Map} \ar[r,symbol=\leq]
& 
2\cdot \dgh
\\
\sup\limits_k\dbk{k} \ar[r,symbol=\leq,"(7')"]
&
\sup\limits_k\hatdbk{k}^{\Tri} \ar[r,symbol=\leq]
\ar[u,symbol=\leq,"(4')"]
& 
\sup\limits_k\hatdbk{k}^{\Cor} \ar[r,symbol=\leq]
\ar[u,symbol=\leq,"(5')"]
& 
\sup\limits_k\hatdbk{k}^{\Map},
\ar[u,symbol=\leq,"(6')"]
&
\end{tikzcd}
\]

To finish proving the proposition, it remains to establish examples such that inequalities (4$'$)-(7$'$) are strict. Such examples will be presented in Proposition \ref{prop:hatdi 5-pt non-zero} and Remark \ref{rmk:XY hatdb=0}.
\end{proof}


\subsubsection{Pseudo-metrics Based on Pullback Distances} \label{subsubsec:metrize}

It is worth noticing that none of $\hatdi^{\Tri}$, $\hatdi^{\Cor}$, and $\hatdi^{\Map}$ satisfy the triangle inequality; see Corollary \ref{cor:triangle fails}.
Nevertheless, we demonstrate below that all of them can be converted into pseudo-metrics. Moreover, utilizing ideas in \cite[Section 6.2]{memoli2017distance}, we show that these new functions retain the favorable properties of the original pullback interleaving-type distances, such as the stability under the Gromov-Hausdorff distance.

Given any non-negative and symmetric function $\omega_X:X\times X\to \bbR$, the following induced function defines a pseudo-metric on $X$: \label{para:metrize}
\[d_X(x,x'):= \inf_{x=x_0,\dots,x_n=x'} \sum_{i=0}^{n-1}\omega_X(x_i,x_{i+1}).\]
Indeed, for any sequences $\alpha: x=x_0,\dots,x_n=x'$ and $\beta: x'=y_0,\dots,y_m=x''$, we have a sequence $x=x_0,\dots,x_n=x'=y_0,\dots,y_m=x''$ between $x$ and $x''$. Thus,
\begin{align*}
    d_X(x,x'')\leq \sum_{i=0}^{n-1}\omega_X(x_i,x_{i+1}) + \sum_{j=0}^{m-1}\omega_X(y_j,y_{j+1}).
\end{align*}
By taking the infimum over all $\alpha$ and $\beta$, we obtain 
\[d_X(x,x'')\leq d_X(x,x')+d_X(x',x'').\]
Additionally, by considering the sequence $x=x_0,x_1=x'$, we see that $d_X\leq \omega_X$.

Denote the collection of all finite metric spaces by $\frX$. Then we apply the above procedure to convert the different variants of the pullback interleaving-type distances into different pseudo-metrics between finite metric spaces, as follows: 
\begin{definition}\label{def:tildi}
For each $\hatdi^{\square}$, where $\square=\Tri, \Cor$ or $\Map$, define 
\[\pbdivariant{\tilde}{\square}{X}{Y}:= \inf_{\substack{X_i\in \frX\\
X=X_0,\dots,X_n=Y}} \sum_{i=0}^{n-1}\pbdivariant{\hat}{\square}{X_i}{X_{i+1}}.\]
\end{definition}

\begin{proposition} \label{prop:tildi ineq}
We have $\tildi^{\Tri}\leq \tildi^{\Cor}\leq \tildi^{\Map}\leq 2\cdot \dgh$.
\end{proposition}

\begin{proof}
The first two inequalities are derived directly from $\hatdi^{\Tri}\leq \hatdi^{\Cor}\leq \hatdi^{\Map}$. 
The last one follows immediately from the facts that $\tildi^{\Map}\leq \hatdi^{\Map}$
 and $\hatdi^{\Map}\leq 2\dgh$ (cf. Theorem \ref{thm:all hatdi and hatdb} ).
\end{proof}

\begin{remark} For pullback bottleneck-type distances, we will see in Section \ref{subsec: hatdb 0} that when $k=0$, $\hatdbk{0}^{\Tri}=\hatdbk{0}^{\Cor}=\hatdbk{0}^{\Map}$ and they all satisfy the triangle inequality. 

It is still an open question whether $\hatdb^{\square}$ induced by positive-degree verbose barcodes satisfies the triangle inequality, where $\square=\Tri, \Cor$ or $\Map$. Even if it fails, we can transform $\hatdb^{\square}$ into a pseudo-metric in a way similar to the above discussion. That is, we define (for $k>0$)
\begin{align*}
    \pbdbvariant{k}{\tilde}{\square}{X}{Y}:= \inf_{\substack{X_i\in  \mathfrak{X}\\ X=X_0,\dots,X_n=Y}} \sum_{i=0}^{n-1}
    \pbdbvariant{k}{\hat}{\square}{X_i}{X_{i+1}}.
\end{align*}
Similarly as before, we note that $\tildb$ satisfies the triangle inequality and the Gromov-Hausdorff stability $\tildb^{\square}\leq 2\cdot \dgh$. 
\end{remark}

\section{Computation of the Pullback Bottleneck Distance} \label{sec:computation}

In Section \ref{sec:pb barcode} we study verbose barcodes under pullbacks, by working out the relation between the verbose barcode of a finite metric space $(X,d_X)$ and the verbose barcode of a pullback space $(Z,\phi_X^*d_X)$ induced by a surjective map $\phi_X:Z\twoheadrightarrow X.$ 
In Section \ref{subsec:computation}, we discuss the computability of pullback distances and, in particular, pullback bottleneck distances.

\subsection{Verbose Barcodes under Pullbacks} \label{sec:pb barcode}

Let $(X,d_X)$ be a finite metric space with $X:=\{x_1,\dots,x_n\}$. Recall from page \ref{para:pb m.s.} that 
for any surjection $\phi: Z\twoheadrightarrow X$, the \textbf{pullback (pseudo-metric) space} (induced by $\phi$) is defined as the pair $(Z,\phi^*d_X)$, where $\phi^*d_X$ is the pullback of the distance function $d_X$. In other words, for any $z_1,z_2\in Z$,
\[(\phi^*d_X)(z_1,z_2):=d_X\left(\phi_X(z_1),\phi_X(z_2) \right).\]
For each $z\in Z$, the point $\phi_X(z)\in X$ is called the \textbf{parent} of $z$.

\begin{definition} [Pullback barcodes]
For any surjective map $\phi_X:Z\twoheadrightarrow X$, we call the degree-$k$ verbose barcode of $(Z,\phi_X^*d_X)$ a \textbf{degree-$k$ pullback barcode} of $X$. 
\end{definition}

\subsubsection{Inductive Formula for Pullback Barcodes 
} \label{sec:pb barcode proof}

We start with the case when the pullback space repeats only one point from the original space. 
For any multiset $A$ and any integer $l\geq 1$, we recall from Equation (\ref{eq:P_l(A)}) that $P_l(A)$ denotes the set consisting of sub-multisets of $A$ each with cardinality $l$.

\begin{proposition} 
\label{prop:pullback_barcode_1}
Assume $X:=\left\{  x_1,\dots,x_n\right\}$ is a pseudo-metric space and $Z=X\sqcup\{z\}$. Suppose $\phi:Z\twoheadrightarrow X$ is such that $z\mapsto x_j$  for some $j=1,\dots,n$ and is the identity otherwise. Then,
\[    \caB_{\Ver,0}(Z)=\caB_{\Ver,0}(X)\sqcup \{(\diam(\{x_j\}))\cdot (1,1)\}= \caB_{\Ver,0}(X)\sqcup \{(0,0)\}, \]
and for $k\geq 1$,
\begin{align*}
    \caB_{\Ver,k}(Z)&= \caB_{\Ver,k}(X)\sqcup \left\{   \diam (\{x_j,x_{i_1},\dots,x_{i_k}\})\cdot(1,1): x_{i_l}\in X-\left\{  x_j\right\}   ,\forall l=1,\dots,k \right\}   \\
    &= \caB_{\Ver,k}(X)\sqcup \left\{   \diam (\{x_j\}\sqcup\beta)\cdot(1,1): \beta\in  P_k(X\setminus \{x_j\})  \right\}  .
\end{align*}
\end{proposition}

\begin{remark}\label{rmk:pullback spaces}
Each finite pullback space (see page \pageref{para:pb m.s.}) $Z$ of $X$ can be regarded as a multiset $X\sqcup \left\{x_{j_1},\dots,x_{j_m}\right\}$ equipped with the metric inherited from $X$ for some $m\geq 0$ and $j_1\leq \dots\leq j_m$. 
Indeed, assume $Z:=X\sqcup\left\{ z_1,\dots,z_m\right\}$ for some auxiliary points $z_1,\dots,z_m$ and consider a surjection $\phi:Z\twoheadrightarrow X$ such that $x\mapsto x$ for $ x\in X$. 
Let $d_Z:=\phi^*d_X$ be the pullback metric on $Z$ under the map $\phi.$ 
If $\phi((z_1,\dots,z_m))=(x_{j_1},\dots,x_{j_m})$, then the points $\{x_{j_1},\dots,x_{j_m}\}$ uniquely determines the map $\phi$ and thus uniquely determine the pullback metric on $Z$. Therefore, $Z$ can be identified with $X\sqcup \left\{x_{j_1},\dots,x_{j_m}\right\}$. 
\end{remark}

Before proving the Proposition \ref{prop:pullback_barcode_1}, we apply it to show the following result:

\pbbarcode*

\begin{proof}
We prove the statement by induction on $m$. 
When $m=1$, the statement follows immediately from Proposition \ref{prop:pullback_barcode_1}. 
Now suppose that the statement is true for $Z':=X\sqcup\left\{  x_{j_{1}},\dots,x_{j_{m-1}}\right\}$, for $m\geq 2$. 
By applying Proposition \ref{prop:pullback_barcode_1} and the induction hypothesis, we obtain:
\begin{align*}
    \caB_{\Ver,k}(Z)
    = \caB_{\Ver,k}(Z') &\sqcup \left\{   \diam (\{x_{j_m}\}\sqcup\beta)\cdot(1,1): \beta\in  P_{k}\left(  (X\setminus \{x_{j_m}\})\sqcup\left\{  x_{j_{1}},\dots,x_{j_{m-1}}\right\}   \right)  \right\}  \\
    = \caB_{\Ver,k}(X) &\sqcup \bigsqcup_{i=0}^{m-2}\left\{   \diam (\{x_{j_{i+1}}\}\sqcup\beta_i)\cdot(1,1): \beta_i\in  P_{k}\left(  (X\setminus \{x_{j_{i+1}}\})\sqcup\left\{  x_{j_{1}},\dots,x_{j_{i}}\right\}   \right)  \right\}  \\
    &\sqcup \left\{   \diam (\{x_{j_{m}}\}\sqcup\beta_{m-1})\cdot(1,1): \beta_{m-1}\in  P_{k}\left(  (X\setminus \{x_{j_m}\})\sqcup\left\{  x_{j_{1}},\dots,x_{j_{m-1}}\right\}  \right)   \right\}  \\
    =  \caB_{\Ver,k}(X) &\sqcup \bigsqcup_{i=0}^{m-1}\left\{  \diam (\{x_{j_{i+1}}\}\sqcup\beta_i)\cdot(1,1): \beta_i\in  P_{k}\left(  (X\setminus \{x_{j_{i+1}}\})\sqcup\left\{  x_{j_{1}},\dots,x_{j_{i}}\right\}   \right)  \right\}. \qedhere
\end{align*}
\end{proof}

\begin{remark} Equation (\ref{eq:pullback barcode}) in Proposition \ref{prop:pullback-barcode} implies the following combinatorial equality: for any $1\leq k\leq n-2,$ the cardinality of $\caB_{\Ver,k}(Z)$ satisfies 
\[{n-1+m \choose k+1}={n-1\choose k+1}+ \sum_{i=0}^{m-1}{i+n-1 \choose k},\]
where the left-hand side follows from Example \ref{ex:card of VR verbose barcodes} and the right hand is given by Proposition \ref{prop:pullback-barcode}.
\end{remark}

\begin{proof} [Proof of Proposition \ref{prop:pullback_barcode_1}]
Fix a degree $k\geq 0$. For notational simplicity, let $\partial:=\partial_{k+1}^Z$ and 
$\ell:=\ell^Z.$
Let
\[A:=\begin{cases}
    \left\{  [z,x_j,x_{i_1},\dots,x_{i_k}]: x_{i_l}\in X-\left\{  x_j\right\}   ,\forall l=1,\dots,k \right\}, &\mbox{$k\geq 1$}\\
    \left\{  [z,x_j] \right\}, &\mbox{$k=0.$}\\
\end{cases} ,\] and notice that $A$ is an orthogonal subset of $\opC_{k+1}\!\left(\Full\left(Z\right)\right)$. 
As in page \pageref{claim1-ultra}, for a $k$-simplex $\gamma$ we denote its $j$-th face by $\face_j(\gamma)$ for $j=0,\dots,k$. For instance, if $\gamma=[z,x_j,x_{i_1},\dots,x_{i_k}]\in A$, then $\face_1(\gamma)=[z,x_{i_1},\dots,x_{i_k}]$ is the simplex obtained by removing the vertex $x_j$.

\begin{claim}\label{claim1}
For any $\gamma\in A$, $\ell(\gamma)=\ell(\face_1(\gamma))=\ell(\partial\gamma).$
\end{claim}

The first equality follows from the fact that $d_Z(z,x_j)=d_X(x_j,x_j)\leq d_X(x_j,x), \forall x\in X$: 
\begin{align*}
\ell(\face_1(\gamma))
&=\ell\left( [z,x_{i_1},\dots,x_{i_k}]\right)\\
&\leq \ell\left( [z,x_j,x_{i_1},\dots,x_{i_k}]\right) \\
&= \max\left\{d_Z(z,x_j),\max_l d_Z(z,x_{i_l}), \max_{l,l'}d_X(x_{i_l},x_{i_{l'}})\right\}\\
&\leq \max\left\{\max_l d_Z(z,x_{i_l}), \max_{l,l'}d_X(x_{i_l},x_{i_{l'}})\right\}=\ell\left( [z,x_{i_1},\dots,x_{i_k}]\right).
\end{align*}
For the second equality, note that for any $l=2,\dots,k+1$, we have
\begin{align*}
\ell\left( \face_l(\gamma)\right)  \leq \ell(\gamma)=\ell(\face_1(\gamma)).
\end{align*}
Incorporating the equality $\ell(\face_1(\gamma))=\ell\left( [z,x_{i_1},\dots,x_{i_k}]\right) = \ell\left( [x_j,x_{i_1},\dots,x_{i_k}]\right)$, we have
\[\ell \left( \partial\gamma\right)=\max\Big\{ \ell\left( [x_j,x_{i_1},\dots,x_{i_k}]\right) ,\max_{l=1,\dots,k+1}\ell\left(\face_l(\gamma)\right)  \Big\}=\ell(\face_1(\gamma)).\]

\begin{claim}\label{claim2}
The set $\partial A$ is orthogonal.
\end{claim}

For any linear combination $c:=\sum_{\gamma\in A} \lambda_{\gamma} \left( \partial \gamma\right)$ of elements in $\partial A$ where the coefficients $\lambda_{\gamma}$ come from the base field $\field$, we want to show that $\ell\left( c\right) =\max_{\lambda_{\gamma}\neq 0 } \ell \left( \partial \gamma\right)$. 
The `$\leq$' follows from the definition of filtration functions. It remains to prove `$\geq $'.

To prove this, write
\[c = \sum_{\gamma\in A} \lambda_{\gamma} \left( \partial \gamma\right)=
\sum_{\gamma\in A} \lambda_{\gamma} \left( \face_1(\gamma)\right)+\ast,\]
where $\ast$ is a linear combination of simplices that have $x_j$ has a vertex. 
Since $x_{i_l}\neq x_j$ for every $l$, $x_j$ is not a vertex of $\face_1(\gamma)$ for any $\gamma$. Therefore,
$\sum_{\gamma\in A} \lambda_{\gamma} \left( \face_1(\gamma)\right)$ is a linear combination of simplices that do not have $x_j$ as a vertex, and thus is orthogonal to the $\ast$ term. Therefore,
\begin{align*}
\ell\left( \sum_{\gamma\in A} \lambda_{\gamma} \left( \partial \gamma\right) \right) 
=\,&\max\left\{    \ell\left( \sum_{\gamma\in A} \lambda_{\gamma} \left( \face_1(\gamma)\right)  \right) , \ell(\ast) \right\}  \\
\geq\,&  \ell\left( \sum_{\gamma\in A} \lambda_{\gamma} \left( \face_1(\gamma)\right)  \right) \\
\geq\,&  \max_{\lambda_{\gamma}\neq 0 } \ell\left(  \face_1(\gamma)\right)  \\
 =\,&  \max_{\lambda_{\gamma}\neq 0 } \ell\left(  \partial\gamma\right). \hspace{1em} \mbox{(by Claim \ref{claim1})}
\end{align*}
Thus, $\partial A$ is an orthogonal subset of $\opC_{k}\!\left(\Full\left(Z\right)\right)$. 

\begin{claim}\label{claim3}
Let $\left(\left(\sigma_1,\dots,\sigma_{m}\right),\left(\alpha_1,\dots,\alpha_r\right)\right)$ be a singular value decomposition (cf. Definition \ref{def:s.v.d.}) of the map $\partial|_{\opC_{k+1}\!\left(\Full\left(X\right)\right)}:\opC_{k+1}\!\left(\Full\left(X\right)\right)\to \kernel \partial_{k}^X$, where 
\[m=\dim\left(\opC_{k+1}\!\left(\Full\left(X\right)\right)\right)={n\choose k+2}\quad  \text{and} \quad r=\dim\left(\im \partial|_{\opC_{k+1}\!\left(\Full\left(X\right)\right)}\right)={n-1\choose k+1}.\]
Then $\left\{  \alpha_1,\dots,\alpha_r\right\}   \sqcup \partial A$ is an orthogonal basis for $\im \partial$. 
\end{claim}

By Example \ref{ex:card of VR verbose barcodes}, we have 
$r+ \card (\partial A) = {n-1 \choose k+1}+{n-1 \choose k} ={n \choose k+1}=
\dim(\im \partial).$
Thus, to show that $\left\{  \alpha_1,\dots,\alpha_r\right\}   \sqcup \partial A$ is an orthogonal basis, it suffices to show that $\left\{  \alpha_1,\dots,\alpha_r\right\}   \sqcup \partial A$ is orthogonal. 
Since both $\left\{  \alpha_1,\dots,\alpha_r\right\}  $ and $ \partial A$ are orthogonal, by Lemma \ref{lem:orthogonal}, it remains to show that $\left\{  \alpha_1,\dots,\alpha_r\right\}  $ and $ \partial A$ are orthogonal to each other. 

Let $\alpha$ and $\gamma$ be non-zero linear combinations of elements in $\left\{  \alpha_1,\dots,\alpha_r\right\}$ and $\partial A$, respectively. We want to prove
\[\ell(\alpha+\gamma) = \max\{\ell(\alpha),\ell(\gamma)\}.\]
When $\ell(\alpha)\neq \ell(\gamma)$, apply Lemma \ref{rmk:property of filtration function}.
When $\ell(\alpha)= \ell(\gamma)$, since `$\leq$' is trivial, we only need to show `$\geq $'.
Because simplices in $\alpha$ do not contain the vertex $z$, if $[z,x_{i_1},\dots,x_{i_k}]$ is in $\gamma$, it must also be in $\alpha+\gamma.$ By Claim \ref{claim1} and Claim \ref{claim2}, we see that $\ell(\gamma)$ is equal to a term in the form of $\ell\left( [z,x_{i_1},\dots,x_{i_k}]\right)$. Therefore,
\[\ell(\alpha+\gamma) \geq \ell([z,x_{i_1},\dots,x_{i_k}])\geq \ell(\gamma) = \max\{\ell(\alpha),\ell(\gamma)\}.\]
This finishes the proof of Claim \ref{claim3}.\\

In what follows, we use subscripts to indicate the degree of $A$. Then $\left\{  \sigma_1,\dots,\sigma_{m}\right\}   \sqcup A_{k+1}\sqcup \partial_{k+2}^Z A_{k+2}$, which is orthogonal. The orthogonality arises because, for any $\gamma\in A_{k+2}$, $\partial_{k+2}^Z\gamma$ has a dominating\footnote{A dominating term is a simplex \( \sigma \) that appears as a summand in the linear combination expressing \( \partial \gamma \) in terms of simplices and satisfies \( \ell(\sigma) = \max_{\lambda_{\gamma} \neq 0} \ell(\partial \gamma) \).} 
term $\face_1(\gamma)$ as established by Claim \ref{claim1}, that is absent from $\left\{\sigma_1,\dots,\sigma_{m}\right\} \sqcup A_{k+1}$. 
This absence can be understood because (1) each $\face_1(\gamma)$ incorporates $z$ as a vertex, in contrast to the simplices in $\left\{  \sigma_1,\dots,\sigma_{m}\right\}$ which are contained in $X$, and (2) unlike the simplices in $A_{k+1}$, no $\face_1(\gamma)$ includes $x_j$.

In addition, $\left\{  \sigma_1,\dots,\sigma_{m}\right\}   \sqcup A_{k+1}\sqcup \partial_{k+2}^Z A_{k+2}$ is an orthogonal basis for $\opC_{k+1}\!\left(\Full\left(Z\right)\right)$, since its cardinality matches the dimension of $\opC_{k+1}\!\left(\Full\left(Z\right)\right)$: ${n\choose k+2}+ {n-1 \choose k}+{n-1 \choose k+1}={n+1 \choose k+2}.$
Thus, we have the following singular value decomposition for $\partial_{k+1}^Z:$
\begin{center}
\begin{tikzcd}[column sep=0em]
    \opC_{k+1}\!\left(\Full\left(Z\right)\right): \ar[d,"\partial_{k+1}^Z"] &\Big( \left\{  \sigma_{r+1},\dots,\sigma_{m}\right\}  \sqcup \partial_{k+2}^Z A_{k+2}, \ar[mapsto,d,"0"] & \left\{  \sigma_1,\dots,\sigma_r\right\}   \sqcup A_{k+1}\Big) \ar[mapsto,d] \\	
    \opC_{k}\!\left(\Full\left(Z\right)\right): &  0 &\Big( \left\{  \alpha_{1},\dots,\alpha_r\right\}  \sqcup \partial_{k+1}^Z A_{k+1}, & \dots \Big) .
\end{tikzcd}
\end{center} 
By the definition of verbose barcodes, we have
\begin{align*}
   \caB_{\Ver,k}(Z)= & \left\{  (\ell (\alpha_i),\ell(\sigma_i)):i=1,\dots,r\right\}   \sqcup 
    \left\{   (\ell(\partial\gamma),\ell(\gamma)):\gamma\in A_{k+1} \right\}   \\
     =& \caB_{\Ver,k}(X)\sqcup \left\{   \diam (\{z,x_{i_1},\dots,x_{i_k}\})\cdot(1,1): x_{i_l}\in X-\left\{  x_j\right\}   ,\forall l=1,\dots,k \right\}   \\
    =& \begin{cases}
    \caB_{\Ver,k}(X)\sqcup \{\diam (\{x_j\})\cdot(1,1)\},&\mbox{$k=0$,}\\
    \caB_{\Ver,k}(X)\sqcup \left\{   \diam (\{x_j,x_{i_1},\dots,x_{i_k}\})\cdot(1,1): x_{i_l}\in X-\left\{  x_j\right\}   ,\forall l=1,\dots,k \right\}  ,&\mbox{$k\geq 1$.}
    \end{cases}\\
    =& \begin{cases}
    \caB_{\Ver,k}(X)\sqcup \{\diam (\{x_j\})\cdot(1,1)\},&\mbox{$k=0$,}\\
    \caB_{\Ver,k}(X)\sqcup \left\{   \diam (\{x_j\}\sqcup\beta)\cdot(1,1): \beta\in  P_k(X\setminus \{x_j\})  \right\}   ,&\mbox{$k\geq 1$.} 
    \end{cases}
\end{align*}
\end{proof}

\subsubsection{Explicit Formula for Pullback Barcodes  }\label{subsec:pb barcodes explicit}

Proposition \ref{prop:pullback-barcode} shows that the pullback barcodes are obtained from the verbose barcodes of the underlying metric space $X$ by adding certain diagonal points. In degree $0$, these extra diagonal points can only be copies of $(0,0)$, and we can determine their exact multiplicity. One may ask whether there exists a more explicit formula for the extra diagonal points and their multiplicity in positive degrees. In this subsection, we answer this question.

By Remark \ref{rmk:pullback spaces}, for a pullback space $Z=X\sqcup \left\{x_{j_1},\dots,x_{j_m}\right\}$, if we order points in $\left\{x_{j_1},\dots,x_{j_m}\right\}$ suitably, we can regard $Z$ as $X\sqcup \{x_1\}^{m_1}\sqcup\{x_2\}^{m_2}\sqcup\dots\sqcup\{x_n\}^{m_n}$, where $m_1,m_2,\dots,m_n\in \bbN$. It follows that pullback spaces $Z$ of $X$ are in one-to-one correspondence with vectors $\vec{m}:=(m_1,m_2,\dots,m_n)\in \bbN^n$. We call $\vec{m}$ the \textbf{pullback vector} associated with $Z$. \label{para:pb vector}

Let $n,k,p\in \bbZ_{\geq 1}$ be such that $p\leq k+1\leq n$. Let $\vec{m}:=(m_1,\dots,m_n)$. 
We introduce the following notation: for any $I_p:=[i_1,\dots,i_p]$ with $1\leq i_1<\dots<i_p\leq n$, letting $\vec{m}(I_p):=(m_{i_1},\dots,m_{i_p})$, we define the following multiplicity function
\label{para:mu_k notation}
\[\mu_k(\vec{m}(I_p)):=\sum_{q=1}^p\,\sum_{\substack{ \omega_1,\dots,\omega_q\geq 1 \\\omega_1+\dots+\omega_q+(p-q)=k+1 }}\, {m_{i_1}+1\choose \omega_1}{m_{i_2}+1\choose \omega_2}\dots {m_{i_{q-1}}+1\choose\omega_{q-1}}{m_{i_q}\choose \omega_q}.\]
See Section \ref{app:mu} for some examples of $\mu_k(\vec{m}(I_p))$.

\explicit*

\begin{proof}
Recall Equation (\ref{eq:pullback barcode}) and its graphical explanation given in Figure \ref{fig:pb barcode}.
Fix a degree $k\geq 0$.
Let $1\leq p\leq k+1$. We first consider the case of $(1,2,\dots,p)$ and count the number of copies of $a:=\diam({ x_{1},\dots, x_{p}})\cdot(1,1)$.

Take any $1\leq q\leq p$. Starting from step $i=m_1+\dots+m_{q-1}$ and continuing until step $i=m_1+\dots+m_{q-1}+m_q-1$, we obtain one copy of $a$ for each copy of a multiset of the following form:
\[A:=\{ \underbrace{ x_{1},\dots,x_{1}}_{\omega_1\geq 1},\underbrace{ x_{2},\dots,x_{2}}_{\omega_2\geq 1}, \dots, \underbrace{ x_{q},\dots,x_{q}}_{\omega_q\geq 1},x_{q+1},\dots,x_p\},\]
where $\omega_1+\dots+\omega_q+(p-q)=k+1.$

During these steps, we are picking points from the red-colored and blue-colored parts below:
\[\text{\myboxred{$x_1,x_2,\dots$}},\, x_q,\,
\text{\myboxred{$\dots,x_n$}},\,
\underbrace{\text{\myboxred{$ x_{1},\dots,x_{1}$}}}_{m_1},\,
\underbrace{\text{\myboxred{$ x_{2},\dots,x_{2}$}}}_{m_2},\,
\text{\myboxred{$\dots$}},\,
\text{\myboxblue{$x_{q},\dots,x_q$}}.\]
Therefore, the number of multisets in the form of $A$ is 
\[\sum_{\substack{ \omega_1,\dots,\omega_q\geq 1 \\\omega_1+\dots+\omega_q+(p-q)=k+1 }}\, {m_1+1\choose \omega_1}{m_2+1\choose \omega_2}\dots {m_{q-1}+1\choose \omega_{q-1}}{m_q\choose \omega_q}.\]
Thus, the total number of copies of $a$ is
\[\sum_{q=1}^p\,\sum_{\substack{ \omega_1,\dots,\omega_q\geq 1 \\\omega_1+\dots+\omega_q+(p-q)=k+1 }}\, {m_1+1\choose \omega_1}{m_2+1\choose \omega_2}\dots {m_{q-1}+1\choose \omega_{q-1}}{m_q\choose \omega_q}=\mu_k(\vec{m}([1,\dots,p])).\]

It is clear that the above result also holds for general $1\leq i_1<\dots<i_p\leq n$. In other words, the total number of copies of $\diam(\{ x_{i_1},\dots, x_{i_p}\})\cdot(1,1)$ is $\mu_k(\vec{m}([i_1,\dots,i_p]))$.
\end{proof}

\subsection{Discussion on the Computation of Pullback Distances} \label{subsec:computation}


In Section \ref{subsec:multi vector}, we reformulate all pullback distances with pullback vectors, and discuss the computability of the pullback distances through these reformulations. 

In Section \ref{subsec: hatdb 0}, we prove Proposition \ref{prop:hatdb degree 0}, which provides a precise formula for computing the pullback bottleneck distance (in all three settings) in degree $0$. 
This formula dictates that when computing $\hatdbk{0}$ bars in barcodes should only be matched with other bars or with the origin $(0,0)$, distinguishing it from the standard bottleneck distance where bars can be matched to any point along the diagonal.

\subsubsection{Reformulations of Pullback Distances Using Pullback Vectors  } \label{subsec:multi vector}

Given any vector $\vec{m}=(m_1,\dots,m_n)\in \bbN^n$, we construct a pullback space 
\[X(\vec{m}):=X\sqcup\bigsqcup_{j=1}^n\{ x_j\}^{m_j}\] 
equipped with the pseudo metric induced from $X$. 
In Proposition \ref{prop:pb distances using vectors}, we reformulate the pullback distances in terms of pullback vectors.

In order to state the proposition, we introduce the following notation:
\begin{itemize}
    \item For any \(\vec{m} \in \mathbb{N}^{n}\) and \(\vec{m}' \in \mathbb{N}^{n'}\), let \(\caX(\vec{m},\vec{m}')\) denote the set of all \( n \times n' \) binary matrices $M$ such that the sum of the $i$-th row of $M$ equals \( m_i \) and the sum of the $j$-th column of $M$ equals \( m'_j \), for all $i$ and $j$.
    \item Let \(\caX_{\mathrm{row}}\) be the set of \( n \times n' \) binary matrices in which each row contains exactly one entry equal to \(1\).  
    \item Let \(\caX_{\mathrm{col}}\) be the set of \( n \times n' \) binary matrices in which each column contains exactly one entry equal to \(1\).
    \item For any two matrices of the same size \( M \) and \( M' \), define \( M \vee M' \) as the matrix obtained by taking the element-wise maximum of \( M \) and \( M' \).
    By $\caX_{\mathrm{row}}\vee \caX_{\mathrm{col}}$, denote the set of all matrices \( M \vee M' \) with \( M \in \caX_{\mathrm{row}} \) and \( M' \in \caX_{\mathrm{col}} \). 
    \item Given any \(\vec{m} \in \mathbb{N}^{n}\), define \(\vec{m}+1 := (m_1+1,m_2+1,\dots, m_n+1)\).
\end{itemize}

\begin{restatable}[Pullback distances reformulation]
{proposition}{pbdistancevectors}\label{prop:pb distances using vectors}
Let $X$ and $Y$ be two finite metric spaces with cardinality $n$ and $n'$, respectively. 
Then, all pullback distances can be written in the form of 
\begin{align*}
    \pbdivariant{\hat}{\square}{X}{Y}=&\inf\limits_{(\vec{m},\vec{m}')\in \frM_\square}\di(\fccVR{X(\vec{m})}, \fccVR{Y(\vec{m}')}) 
    \\
    \pbdbvariant{k}{\hat}{\square}{X}{Y}=&\inf\limits_{(\vec{m},\vec{m}')\in \frM_\square}\dm(\caB_{\Ver,k}(X(\vec{m})), \caB_{\Ver,k}(Y(\vec{m}'))),
\end{align*}
where
\begin{itemize}
    \item [(1)] for $\square=\Tri$, $\frM_{\Tri} :=\left\{(\vec{m},\vec{m}')\in \bbN^{n}\times \bbN^{n'} \mid \|  \vec{m}+1 \| _1=\|  \vec{m}'+1\|_1 \right\}$;
    \item [(2)] for $\square=\Cor$, 
    $\frM_{\Cor} :=\left\{(\vec{m},\vec{m}')\in \bbN^{n}\times \bbN^{n'}\mid 
    \caX(\vec{m}+1,\vec{m}'+1) \neq \emptyset
    \right\}$;
    \item [(3)] for $\square=\Map$, $\frM_{\Map} :=\left\{(\vec{m},\vec{m}')\in \bbN^{n}\times \bbN^{n'}\mid 
    \caX(\vec{m}+1,\vec{m}'+1) \cap( \caX_{\mathrm{row}}\vee \caX_{\mathrm{col}}) \neq \emptyset\right\}$.
\end{itemize}
\end{restatable}

\begin{remark}
In the definition of $\frM_{\Cor}$, the condition $\caX(\vec{m}+1,\vec{m}'+1) \neq \emptyset$ can be characterized by direct constraints on the vectors $\vec{m},\vec{m}'$ (see (\ref{eq:gale-eyser})), following from the Gale--Ryser Theorem \cite{gale1957theorem,ryser1957combinatorial}, as we now describe.

For any \(\vec{m} = (m_1, \dots, m_n) \in \mathbb{N}^n\), denote by
\(
{}^\downarrow\vec{m} : = \left(m_{(1)}, \dots, m_{(n)}\right)
\)
the vector obtained by rearranging the entries of \(\vec{m}\) in nonincreasing order, i.e., \(m_{(1)} \ge m_{(2)} \ge \cdots \ge m_{(n)}\). 

Define the \emph{conjugate} of a nonincreasingly ordered vector \(\vec{m}\in \mathbb{N}^n\) as 
\(
\vec{m}^* := (m^*_{1}, m^*_{2}, \dots, m^*_{n}),
\)
where 
\[
m^*_{k} := \card(\{ i \in \{1, \dots, n\} \mid m_{i} \ge k \}),
\]
if $1\leq k\leq m_1$. If $k> m_{1}$, we set $m^*_{k} := 0$.

Given \(\vec{m} \in \mathbb{N}^{n}\) and \(\vec{m}' \in \mathbb{N}^{n'}\), we say that \(\vec{m}\) is \emph{majored} by \(\vec{m}'\), written \(\vec{m} \prec \vec{m}'\), if  
\[
\sum_{i=1}^k m_i \leq \sum_{i=1}^k m_i' \quad \text{for all } k,
\]
where we assume \(m_i=0\) for $i>n$ and \( m'_i = 0 \) for \( i > n' \). 

By the Gale--Ryser Theorem, we have
\begin{equation}\label{eq:gale-eyser}
    \caX(\vec{m}+1,\vec{m}'+1) \neq \emptyset \quad \Longleftrightarrow \quad \| \vec{m}+1 \| _1=\|  \vec{m}'+1\|_1 \text{ and }{}^\downarrow\vec{m}+1 \prec ({}^\downarrow\vec{m}'+1)^*.
\end{equation}

Unlike \(\frM_{\Tri}\) and \(\frM_{\Cor}\), it remains an open question whether \(\frM_{\Map}\) admits a formulation in terms of direct constraints on \(\vec{m}\) and \(\vec{m}'\).
\end{remark}

\color{black}

\begin{remark}[Analysis of computational complexity of computing pullback distances.]
Let $X$ and $Y$ be two finite metric spaces with cardinality $n$ and $n'$, respectively. Without loss of generality, assume $n\geq n'$. By Proposition \ref{prop:pb distances using vectors}, the brute-force algorithms for computing pullback distances between $X$ and $Y$ have the following complexity upper bounds (view $k$ as a constant): \label{para:complextiy analysis}
\begin{itemize}
    \item[(a)] for $\hatdbk{k}^{\Cor}$: $O\left(n^{2n}\cdot (2n)^{3(k+1)}\right) = O(n^{2n+3k+3})$;
    \item[(b)] for $\hatdi^{\Cor}$: $O\left(n^{2n}\cdot \sum_{k=0}^{n-2}(2n)^{3(k+1)}\right) = O\left(n^{5n-3}\right)$;
    \item[(c)] for $\hatdbk{k}^{\Map}$: $O\left(16^n\cdot (2n)^{3(k+1)}\right) = O(16^nn^{3k+3})$;
    \item[(d)] for $\hatdi^{\Map}$: $O\left(16^n\cdot \sum_{k=0}^{n-2}(2n)^{3(k+1)}\right) = O\left(16^nn^{3n-3}\right)$.
\end{itemize}
Here $O(\cdot)$ is the big $O$ notation.
To establish complexity bounds for computing pullback bottleneck distances, we begin by separately estimating the cardinality of \(\frM_{\square}\) and the complexity of computing the distance \(\dm\).  

For any non-negative integer \( i \), the number of vectors with \( n \) non-negative integer entries summing to \( i \) equals the number of ways to distribute \( i \) indistinguishable balls into \( n \) distinguishable boxes. This is given by  
${ n+i-1\choose n-1 }(={ n+i-1\choose i })$, which we will use to estimate the cardinality of \(\frM_{\square}\). 

Note that for \((\vec{m},\vec{m}') \in \frM_{\Cor}\), we have \( \|\vec{m}\|_\infty \leq n' - 1 \) and \( \|\vec{m}'\|_\infty \leq n - 1 \). This follows because each \( m_i + 1 \) equals a row sum in an \( n \times n' \) binary matrix and is therefore at most \( n' \), and similarly, \( m'_j + 1 \leq n \).  
Consequently, we obtain the bound  
\begin{equation}\label{eq:card(M_Cor)}
    \card(\frM_{\Cor}) \leq\, n^{n'}(n')^n = O(n^{2n}).
\end{equation}

For any \((\vec{m},\vec{m}') \in \frM_{\Map} \subset \frM_{\Cor}\), we have \( \|\vec{m}\|_\infty \leq n' - 1 \) and \( \|\vec{m}'\|_\infty \leq n - 1 \), along with the additional constraints \( \|\vec{m}\|_1 \leq n' \) and \( \|\vec{m}'\|_1 \leq n \). 
The latter follows from the fact that \( \|\vec{m} + 1\|_1 = \|\vec{m}\|_1 + n \) equals the total number of 1s in a matrix of the form \( F \vee F' \), where \( F \in \caX_{\mathrm{row}} \) and \( F' \in \caX_{\mathrm{col}} \). 
Since each row of \( F \) and each column of \( F' \) contain at most one entry equal to 1, the total number of 1s in \( F \vee F' \) does not exceed \( n' + n \). This implies \( \|\vec{m}\|_1 \leq n' \), and similarly, \( \|\vec{m}'\|_1 \leq n \).
Thus, we obtain the following bound on \( \card(\frM_{\Map}) \):
\begin{align}
\card(\frM_{\Map}) 
    \leq\, & \sum_{i=0}^{n'}{ n+i-1\choose i }\cdot { n'+(n+i-n')-1\choose (n+i-n')} \notag \\
    =\,&\sum_{i=0}^{n'}{ n+i-1\choose i }\cdot { n+i-1\choose n'-1} \notag \\  
    =\, & O\left(n{ 2n-1\choose n}^2\right) \notag \\
    =\, & O\left(n\left(\frac{4^n}{\sqrt{2\pi n}}\right)^2\right) = O(16^n). \quad \text{(by Stirling's approximation)} \label{eq:card(M_Map)}
\end{align}

In order to bound the complexity of computing the matching distance
\[
\dm(\caB_{\Ver,k}(X(\vec{m})), \caB_{\Ver,k}(Y(\vec{m}'))),
\]
we first observe the following bound on the cardinality of the involved verbose barcode 
\begin{equation}\label{eq:card(B_ver)}
    \card(\caB_{\Ver,k}(X(\vec{m}))) = {\|  \vec{m}\|_1+n-1\choose k+1} = O\left((\|  \vec{m}\|_1+n)^{k+1}\right) = O\left(n^{k+1}\right),
\end{equation}
and similarly for $\card(\caB_{\Ver,k}(Y(\vec{m}')))$.
Combining Equations (\ref{eq:card(M_Cor)}), (\ref{eq:card(M_Map)}), and (\ref{eq:card(B_ver)}), along with the fact that the matching distance can be computed in cubic time \cite[Section 4.2]{burkard2012assignment} relative to the input size, we complete the proof of the complexity bound for Items (a) and (c), i.e., for $\hatdbk{k}^{\Map}$ and $\hatdbk{k}^{Cor}$, respectively.

For $\hatdi^{\Map}$ and $\hatdi^{\Cor}$, the idea is similar except that one needs to add up the complexity of computing the matching distance in each degree.

For tripods, we estimate the complexity of computing pullback distances by bounding the \(\ell_\infty\)-norm of the pullback vectors, which in turn provides an upper bound on the distance computation.  
Fix a positive integer \( N \) and define  
\[
\frM_{\Tri}^N = \left\{(\vec{m},\vec{m}')\in \bbN^{n}\times \bbN^{n'} \mid \|  \vec{m}+1 \| _1=\|  \vec{m}'+1\|_1 ,\|\vec{m}\|_\infty, \|\vec{m}'\|_\infty\leq N\right\}.
\]  
Then, the cardinality satisfies $\card(\frM_{\Tri}^N) \leq n^{N}(n')^N = O(n^{2N}).$
\end{remark}

Reformulating pullback distances via pullback vectors offers a potential advantage of more efficient computation of these distances when approached using brute-force. 
This is because the set \(\frM_{\Tri}\) is ``smaller'' than the set of tripods\footnote{We use ``smaller'' informally here, since their cardinalities could still be the same.}:
each tripod only produces one element in $\frM_{\Tri}$, but each pair $(\vec{m},\vec{m}')\in \frM_{\Tri}$ induces multiple tripods, as explained in the proof of Proposition \ref{prop:pb distances using vectors} on page \pageref{pf:pb using vectors}. 
The same phenomenon applies to correspondences and maps. In other words, the sets $\frM_{\Tri}, \frM_{\Cor}$, and $\frM_{\Map}$ are ``smaller'' than the respective sets of tripods, correspondences, and maps.
To elaborate on this point, consider the following comparisons: 
\begin{itemize}
    \item The number of pairs of maps $(f,g)$ between $X$ and $Y$ is
    \[n^{n'}(n')^n=O(n^{2n})>O(16^n) = \card(\frM_{\Map}).\]
    \item The number of correspondences $R$ between $X$ and $Y$ is at least
    \[2^{nn'}-n2^{(n-1)n'}-n'2^{(n'-1)n}=O(2^{n^2})>O(n^{2n}) = \card(\frM_{\Cor}).\]
\end{itemize}

\begin{proof} [Proof of Proposition \ref{prop:pb distances using vectors}]\label{pf:pb using vectors}

It suffices to establish the results for the pullback interleaving distances, as the case of pullback bottleneck distances follows from a similar argument.  

To proceed, let \(\{x_1, \dots, x_n\}\) and \(\{y_1, \dots, y_{n'}\}\) denote the underlying point sets of \(X\) and \(Y\), respectively.

Case (1): given any tripod 
$X\xtwoheadleftarrow{\phi_X}Z\xtwoheadrightarrow{\phi_Y}Y$, let 
\[m_i:=\card(\phi_X^{-1}({x_i}))-1\quad  \text{and} \quad  m_j':=\card(\phi_Y^{-1}({y_j}))-1\] 
for every $x_i\in X$ and $y_j\in Y$. These vectors $\vec{m}:=(m_1,\dots,m_n)$ and $\vec{m}':=(m_1',\dots,m_{n'}')$ satisfy \( \| \vec{m} + 1 \|_1 = \card(Z) = \| \vec{m}' + 1 \|_1 \) and induce isometries (see page \pageref{para:isometry}) \( X(\vec{m}) \cong Z_X \) and \( Y(\vec{m}') \cong Z_Y \), leading to
\[
\hatdi^{\Tri} \leq \min\limits_{(\vec{m},\vec{m}')\in \frM_{\Tri}} \di(\fccVR{X(\vec{m})}, \fccVR{Y(\vec{m}')}).
\]

Conversely, given any \((\vec{m},\vec{m}')) \in \frM_{\Tri}\) and for each bijection $f:X(\vec{m}) \to Y(\vec{m}')$, we construct a tripod $X\xtwoheadleftarrow{\pi_X}Z:=X(\vec{m})\xtwoheadrightarrow{\pi_Y\circ f}Y$ such that $Z_X$ and $Z_Y$ are isometric to $X(\vec{m})$ and $Y(\vec{m}')$, respectively. Here $\pi_X:X(\vec{m})
\to X$ and $\pi_Y:Y(\vec{m}')
\to Y$ are the projection maps.

Case (2): for the `$\leq$' direction, it suffices to show that for any correspondence \( R \subset X \times Y \), there exists \((\vec{m}, \vec{m}') \in \frM_{\Cor}\) such that \( R_X \) and \( R_Y \) (defined in Equation (\ref{eq:R_X and R_Y})) are isometric to \( X(\vec{m}) \) and \( Y(\vec{m}') \), respectively. A sufficient condition for \( R_X \cong X(\vec{m}) \) and \( R_Y \cong Y(\vec{m}') \) is that  
\begin{equation}\label{eq:cardinality conditions}
    m_i + 1 = \card(\pi_X^{-1}({x_i})) \quad \text{and} \quad m_j' + 1 = \card(\pi_Y^{-1}({y_j})),
\end{equation}
where \( \pi_X \) and \( \pi_Y \) are the projections from \( R \) to \( X \) and \( Y \), respectively.
Conversely, it suffices to show that every such pair \((\vec{m}, \vec{m}')\) determines a corresponding \( R \) satisfying Condition (\ref{eq:cardinality conditions}).

Given any correspondence \( R \subset X \times Y \). Define  
\[
m_i := \card(\pi_X^{-1}({x_i})) - 1, \quad \text{and} \quad  
m_j' := \card(\pi_Y^{-1}({y_j})) - 1.
\]  
Next, define the binary matrix \( M \) by setting \( M_{ij} = 1 \) if and only if \( (x_i, y_j) \in R \). 
Then \( M \) satisfies the prescribed row and column sums.
As a consequence, it follows that \( M \in \caX(\vec{m} + 1, \vec{m}' + 1) \), ensuring that this set is nonempty. Consequently, \((\vec{m}, \vec{m}') \in \frM_{\Cor}\) and satisfies Condition (\ref{eq:cardinality conditions}).

Conversely, given any \((\vec{m}, \vec{m}') \in \frM_{\Cor},\) there exists a corresponding binary matrix $M \in \caX(\vec{m} + 1, \vec{m}' + 1)$, which induces a correspondence $R$ satisfying Condition (\ref{eq:cardinality conditions}).
    
Case (3): Given two maps \( f: X \to Y \) and \( g: Y \to X \), let \( G := G(f, g) \) denote the union of their graphs, as defined in Equation~(\ref{eq: G(f,g)}).
For the `$\leq$' direction, it suffices to show that for any pair of maps \( f: X \to Y \) and \( g: Y \to X \), there exists \((\vec{m}, \vec{m}') \in \frM_{\Cor}\) such that \( G_X \) and \( G_Y \) (defined in Equation (\ref{eq:G_X and G_Y})) are isometric to \( X(\vec{m}) \) and \( Y(\vec{m}') \), respectively. A sufficient condition for \( G_X \cong X(\vec{m}) \) and \( G_Y \cong Y(\vec{m}') \) is  
\begin{equation}\label{eq:cardinality conditions for maps}
    m_i + 1 = \card(\pi_X^{-1}({x_i})) \quad \text{and} \quad m_j' + 1 = \card(\pi_Y^{-1}({y_j})),
\end{equation}  
where \( \pi_X \) and \( \pi_Y \) denote the projections from \( G \) onto \( X \) and \( Y \), respectively.  
Conversely, it suffices to show that every such pair \((\vec{m}, \vec{m}')\) determines a pair of maps \( f: X \to Y \) and \( g: Y \to X \) such that the union of their graphs satisfies Condition (\ref{eq:cardinality conditions for maps}).  

Given any pair of maps \( f: X \to Y \) and \( g: Y \to X \), let \( G := G(f,g) \) be the union of the graphs of \( f \) and \( g \), inducing a correspondence between \( X \) and \( Y \). 
We define
\[
m_i := \card(\pi_X^{-1}({x_i})) - 1, \quad \text{and} \quad  
m_j' := \card(\pi_Y^{-1}({y_j})) - 1.
\]  
Next, define the binary matrices \( F \) and \( F' \) by setting \( F_{ij} = 1 \) if and only if \( f(x_i) = y_j \) and \( F'_{ij} = 1 \) if and only if \( g(y_j) = x_i \). Then, define  
\(
M := F \vee F' \in \caX_{\mathrm{row}} \vee \caX_{\mathrm{col}}.
\)
Since \( M \) satisfies the prescribed row and column sums, it follows that \( M \in \caX(\vec{m} + 1, \vec{m}' + 1) \). Consequently, \((\vec{m}, \vec{m}') \in \frM_{\Map}\) and satisfies Condition (\ref{eq:cardinality conditions for maps}).  

Conversely, given any \((\vec{m}, \vec{m}') \in \frM_{\Map}\), there exists a corresponding binary matrix \( M \in \caX(\vec{m} + 1, \vec{m}' + 1) \cap (\caX_{\mathrm{row}} \vee \caX_{\mathrm{col}}) \). Suppose \( M = F \vee F' \) for some \( F \in \caX_{\mathrm{row}} \) and \( F' \in \caX_{\mathrm{col}} \). Then, \( F \) induces a map \( f: X \to Y \) defined by \( f(x_i) = y_j \) if and only if \( F_{ij} = 1 \), and similarly, \( F' \) induces a map \( g: Y \to X \).  
It follows that the union of the graphs of \( f \) and \( g \) satisfies Condition (\ref{eq:cardinality conditions for maps}), completing the proof.
\end{proof}

\subsubsection{The Case of Degree Zero} \label{subsec: hatdb 0}

In this subsection, we prove Proposition \ref{prop:hatdb degree 0}. Recall from Section \ref{sssec:dm} the matching distance $\dm$: for any $A,B\subset \overline{\caH}^\infty$,
\[\dm (A,B)=\min\left\{   \max_{a\in A} \infdistance{a}{\phi(a)}:A\xrightarrow{\phi} B \text{ a bijection}\right\}  .\]

Before discussing how to compute the degree-$0$ pullback bottleneck distance, let us identify a special property of the matching distance: 
\begin{fact} \label{fact:1-dim dm}
Let $A:=\{ a_1\geq \dots\geq a_n\}$ and $B:=\{b_1\geq \dots\geq b_n\}$ be two multisets of $n$ real numbers each. Then, 
\[\dm(A,B)=\min_{\text{bij }f:A\to B}\max_{i}|a_i-f(a_i)| = \max_{i}|a_i-b_i|.\]
\end{fact}
\begin{proof} For any pair of real numbers $a\geq a'$ and $b\geq b'$, notice that their differences satisfy the so-called \emph{bottleneck Monge property} (see \cite[Section 4.1]{burkard1996perspectives}):
\begin{equation}\label{eq:monge}
    \max\{ |a-b'|, |a'-b| \}\geq  \max\{ |a-b|, |a'-b'| \}.
\end{equation}

Consider any bijection $f:A\to B$ and assume there exist $i<j$ and $i'<j'$ such that $f(a_i)=b_{j'}$ and $ f(a_j)=b_{i'}$. We define a new bijection $\tilde{f}:A\to B$ such that $\tilde{f}(a_i)=b_{i'}$, $ \tilde{f}(a_j)=b_{j'}$ and $\tilde{f}=f$ otherwise. It follows from Equation (\ref{eq:monge}) that $\dis(\tilde{f})\leq \dis(f)$. Repeat this process which stops when we obtain a bijection $g:a_i\mapsto b_i$ for every $i$. Thus, $g$ is the optimal bijection.
\end{proof}

\hatdbfordegreezero*

\begin{proof}
By Proposition \ref{prop:pb distances using vectors}, $\pbdbk{0}{X}{Y}$
can be reformulated using pullback vectors as: 
\begin{equation}\label{eq:reformulation degree 0}
    \pbdbk{0}{X}{Y}=\min\limits_{(\vec{m},\vec{m}')\in \frM_{\Tri}}\dm(\caB_{\Ver,0}(X(\vec{m})), \caB_{\Ver,0}(Y(\vec{m}'))),
\end{equation}
where $\frM_{\Tri} :=\left\{(\vec{m},\vec{m}')\in \bbN^{n}\times \bbN^{n'} \mid \|  \vec{m}+1 \| _1=\|  \vec{m}'+1\|_1 \right\}$.
By Proposition \ref{prop:pullback-barcode}, we have
\begin{equation}\label{eq:barc 0}
    \caB_{\Ver,0}(X(\vec{m}))=\caB_{\Ver,0}(X)\sqcup \{(0,0)\}^{\|  \vec{m} \| _1}\quad  \text{and} \quad \caB_{\Ver,0}(Y(\vec{m}'))=\caB_{\Ver,0}(Y)\sqcup \{(0,0)\}^{\|  \vec{m}' \| _1}.
\end{equation}

Combining Equation (\ref{eq:reformulation degree 0}) and Equation (\ref{eq:barc 0}), we have 
\[\pbdbk{0}{X}{Y}=\min\limits_{\substack{m, m' \in \mathbb{N} \\ m+n = m'+n'}}\dm\left(\caB_{\Ver,0}(X)\sqcup \{(0,0)\}^m,\caB_{\Ver,0}(Y)\sqcup \{(0,0)\}^{m'}\right).\]

By the given assumption, we have $\caB_{\Ver,0}(X)=\{(0,a_1),\dots,(0,a_{n-1})\}\sqcup\{(0,\infty)\}$ and $\caB_{\Ver,0}(Y)=\{(0,b_1),\dots,(0,b_{n'-1})\}\sqcup\{(0,\infty)\}$. 
Together with Fact \ref{fact:1-dim dm}, we have
\begin{align}
& \pbdbk{0}{X}{Y}\notag\\
=\, & \min\limits_{\substack{m, m' \in \mathbb{N} \\ m+n = m'+n'}}\dm  \Big(\{(0,\infty), (0,a_1),\dots,(0,a_{n-1})\}\sqcup \{(0,0)\}^m , \{(0,\infty),(0,b_1),\dots,(0,b_{n'-1})\}\sqcup \{(0,0)\}^{m'} \Big)\notag\\
=\, & \min\limits_{\substack{m, m' \in \mathbb{N} \\ m+n = m'+n'}}\dm  \Big(\{a_1,\dots,a_{n-1}, \underbrace{0,\dots, 0}_{m}\} , \{b_1,\dots,b_{n'-1}, \underbrace{0,\dots, 0}_{m'}\} \Big)\notag\\
=\, & \dm  \Big(\{a_1,\dots,a_{n-1}, \underbrace{0,\dots, 0}_{n-n'}\} , \{b_1,\dots,b_{n'-1}\} \Big)\notag\\
=\, & \max\left\{ \max_{1\leq i\leq n-1}|a_i-b_i|, \max_{n\leq i\leq n'-1} b_i \right\}.\notag \qedhere
\end{align}
\end{proof}

\begin{example}\label{ex:twopointspace} 
Let $X_{\epsilon}:=(\left\{  0,1\right\}   ,d^{\epsilon})$ be a metric space consisting of two points, where $d^{\epsilon}(0,1)=1+\epsilon$. It is not hard to verify that $\caB_{\Con,0}(X_{\epsilon})=\left\{  (0,1+\epsilon),(0,+\infty)\right\}   $, and thus,\[\tfrac{1}{2}\db(\caB_{\Con,0}(X_{\epsilon}),\caB_{\Con,0}(X_{0}))= \min\left\{  \tfrac{\epsilon}{2},\tfrac{1+\epsilon}{4}\right\}  \leq \tfrac{\epsilon}{2}=\dgh(X_{\epsilon},X_{0}).\]where the inequality is strict when $\epsilon>1.$
However, for any $\epsilon\geq 0$, 
\[\tfrac{1}{2}\pbdbk{0}{X_\epsilon}{X_0}=\frac{\epsilon}{2}=\dgh(X_{\epsilon},X_{0}).\]
Thus, in this example, $\hatdb$ between degree-$0$ barcodes have a stronger distinguishing power than $\db$.
\end{example}

Proposition \ref{prop:hatdb degree 0} implies the following corollary:

\begin{corollary}\label{cor:hatdb zero triangle}
\begin{itemize}
    \item[(1)] $\hatdbk{0}^{\Tri}=\hatdbk{0}^{\Cor}=\hatdbk{0}^{\Map}$.
    \item[(2)] The pullback bottleneck distance between degree-$0$ verbose barcodes satisfies the triangle inequality. 
\end{itemize}
\end{corollary}

\begin{proof} Let $X$ and $Y$ be finite metric spaces with cardinality $n_X$ and $n_Y$, respectively. As before, consider the death times of finite-length bars in the degree-$0$ barcodes of $X$ and $Y$ as $a_1\geq \dots\geq a_{n_X-1}$ and $b_1\geq \dots\geq b_{n_Y-1}$, respectively. For any $N\geq \max\{n_X,n_Y\}$, define 
\[\alpha^N:=\{a_1, \dots, a_{n_X-1},\underbrace{0,\dots, 0}_{N-n_X}\} \text{ and }\beta^N:=\{a_1, \dots, b_{n_Y-1},\underbrace{0,\dots, 0}_{N-n_Y}\}.\]
We have seen in Proposition \ref{prop:hatdb degree 0} that
\[\pbdbvariant{0}{\hat}{\Tri}{X}{Y}= \infdistance{\alpha^N}{\beta^N},\,\forall N\geq \max\{n_X,n_Y\}.\]
Via a similar discussion, we have
\[\pbdbvariant{0}{\hat}{\Map}{X}{Y}= \infdistance{\alpha^{n_X+n_Y}}{\beta^{n_X+n_Y}},\]
and thus $\pbdbvariant{0}{\hat}{\Map}{X}{Y}$ is equal to $\pbdbvariant{0}{\hat}{\Tri}{X}{Y}$.
Because $\hatdbk{0}^{\Tri}\leq\hatdbk{0}^{\Cor}\leq\hatdbk{0}^{\Map}$, we must have that all three of them are equal. So Item (1) holds.

Let $Z$ be a finite metric space of $n_Z$ many points. Assume that the death times of finite-length bars in the degree-$0$ barcodes of $Z$ are $c_1\geq \dots\geq c_{n_Z-1}$. For any $N\geq \max\{n_X,n_Y,n_Z\}$, define $\gamma^N:=\{c_1, \dots, c_{n_Z-1},\underbrace{0,\dots, 0}_{N-n_Z}\}.$
Then, we have
\[\pbdbk{0}{X}{Z}  = \infdistance{\alpha^N}{\gamma^N}\leq \infdistance{\alpha^N}{\beta^N} + \infdistance{\beta^N}{\gamma^N} =\pbdbk{0}{X}{Y}+\pbdbk{0}{Y}{Z}.\]
This proves Item (2).
\end{proof}

\subsubsection{An Important Example of Ultra-Metric Spaces}
\label{subsub:non triangle}

We demonstrate that all three pullback interleaving-type distances violate the triangle inequality by considering the three five-point ultra-metric spaces depicted in Figure \ref{fig:5-point-ultra}.

\begin{figure}[ht]
    \centering
    \includegraphics[scale=0.29]{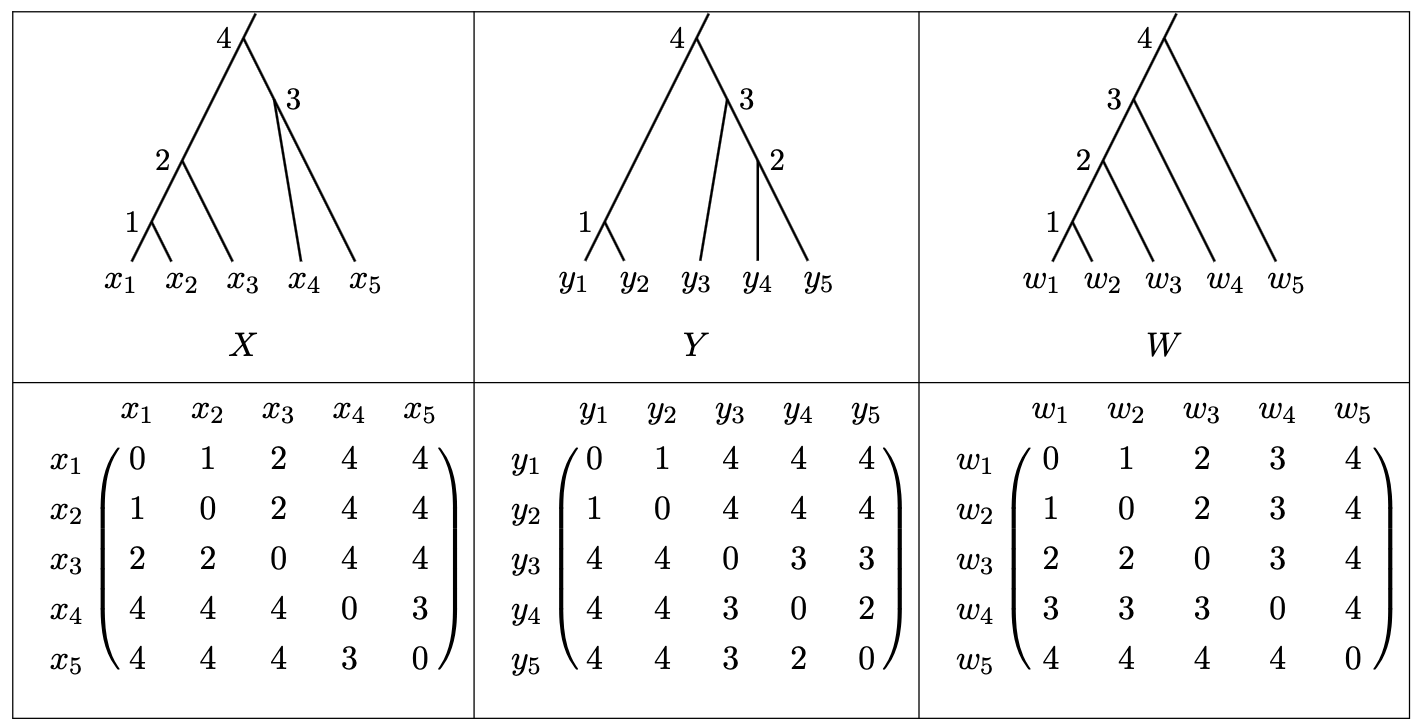}
    \caption{\textbf{First row:} three $5$-point ultra-metric spaces denoted as $X$, $Y$ and $W$, respectively. \textbf{Second row:} distance matrices of the ultra-metric spaces $X$, $Y$ and $W$, respectively. In each matrix, the $(i,j)$-th element is the distance between the $i$-th and $j$-th elements in the corresponding metric space.   }
    \label{fig:5-point-ultra}
\end{figure}

\begin{proposition}\label{prop:hatdi 5-pt non-zero}
Let $X$, $Y$ and $W$ be given as in Figure \ref{fig:5-point-ultra}. Then,
\begin{itemize}
    \item [(1)] for the pair $(X,W)$ or $(W,Y)$, $\hatdi^{\Tri}=\hatdi^{\Cor}=\hatdi^{\Map}=0$;
    \item [(2)] for the pair $(X,Y)$, $\hatdi^{\Tri}=\hatdi^{\Cor}=\hatdi^{\Map}=1$.
\end{itemize}
\end{proposition}

Proposition \ref{prop:hatdi 5-pt non-zero} implies the following corollary:

\begin{corollary}\label{cor:triangle fails}
The triangle inequality does not hold for $\hatdi^{\Tri}$, $\hatdi^{\Cor}$, and $\hatdi^{\Map}$.
\end{corollary}

\begin{proof}
Consider $X$, $Y$, and $W$ as depicted in Figure \ref{fig:5-point-ultra}. Proposition \ref{prop:hatdi 5-pt non-zero} implies that 
\[\pbdivariant{\hat}{\Tri}{X}{Y}=1
    > 0+0=
    \pbdivariant{\hat}{\Tri}{X}{W}+\pbdivariant{\hat}{\Tri}{W}{Y}.\]
The same is true for $\hatdi^{\Cor}$ and $\hatdi^{\Map}$.
\end{proof}

\paragraph{Proof of Proposition \ref{prop:hatdi 5-pt non-zero} (1).} 

By Theorem \ref{thm:all hatdi and hatdb}, we always have $\hatdi^{\Tri}\leq \hatdi^{\Cor}\leq \hatdi^{\Map}$. Thus, to prove the statement, it suffices to show that $\hatdi^{\Map}=0$. By Proposition \ref{prop:pb distances using vectors}, this is equivalent to finding pullback vectors $(\vec{m},\vec{m}')\in \frM_{\Map}$ such that $\di$ between the VR FCCs of the corresponding pullback spaces is $0$. 

For $(X,W)$, consider 
\[\vec{m}=(0,0,0,2,3) \text{ and }\vec{m}'=(0,0,0,3,2).\]
Recall the definition of $\mathfrak{M}_{\text{Map}}$ from Proposition \ref{prop:pb distances using vectors}. 
Note that \((\vec{m},\vec{m}') \in \mathfrak{M}_{\text{Map}}\), since there exists \( F \vee F' \) in \( \caX(\vec{m}+1, \vec{m}'+1) \), given by
\[
F = \begin{pmatrix}
    0 & 0 & 0 & 0 & 1 \\
    0 & 0 & 0 & 1 & 0 \\
    0 & 0 & 0 & 1 & 0 \\
    0 & 0 & 0 & 1 & 0 \\
    0 & 0 & 0 & 0 & 1
\end{pmatrix} \in \caX_{\mathrm{row}} 
\quad \text{and} \quad 
F' = \begin{pmatrix}
    0 & 0 & 0 & 0 & 0 \\
    0 & 0 & 0 & 0 & 0 \\
    0 & 0 & 0 & 0 & 0 \\
    0 & 0 & 1 & 0 & 1 \\
    1 & 1 & 0 & 1 & 0
\end{pmatrix} \in \caX_{\mathrm{col}} .
\]

Let $Z:=X(\vec{m})$ and $Z':=W(\vec{m}')$ be the pullback spaces of $X$ and $W$, respectively. 
We claim that 
\[\di\left(\fccVR{Z} ,\fccVR{Z'} \right)=\max_{k=0,\dots,\card(Z)-2}\dm
    \left(  \caB_{\Ver,k}(Z), \caB_{\Ver,k}(Z')\right) = 0.\]

Since $(\vec{m},\vec{m}')\in \frM_{\Map}$, we deduce from the above claim that $\hatdi^{\Map}$ between $X$ and $W$ is also zero. 
The claim follows by computing the verbose barcodes of $Z$ and $Z'$ via Theorem \ref{thm:verbose of ultra} and Proposition \ref{prop:explicit barc k}. For brevity, we have omitted the details. Interested readers may refer to \cite[Section 5.7 \& Section A.1]{zhou2023}. 

For $(Y,W)$, consider 
\[\vec{m}=(0,0,0,1,4) \text{ and }\vec{m}'=(0,0,4,0,1),\]
Note that \((\vec{m},\vec{m}') \in \mathfrak{M}_{\text{Map}}\), since there exists \( F \vee F' \) in \( \caX(\vec{m}+1, \vec{m}'+1) \), given by
\[
F = \begin{pmatrix}
    0 & 0 & 1 & 0 & 0 \\
    0 & 0 & 1 & 0 & 0 \\
    0 & 0 & 1 & 0 & 0 \\
    0 & 0 & 0 & 0 & 1 \\
    0 & 0 & 1 & 0 & 0
\end{pmatrix} \in \caX_{\mathrm{row}} 
\quad \text{and} \quad 
F' = \begin{pmatrix}
    0 & 0 & 0 & 0 & 0 \\
    0 & 0 & 0 & 0 & 0 \\
    0 & 0 & 0 & 0 & 0 \\
    0 & 0 & 1 & 0 & 0 \\
    1 & 1 & 0 & 1 & 1
\end{pmatrix} \in \caX_{\mathrm{col}} .
\]

Apply a similar argument as in the previous case to deduce that $\hatdi^{\Map}$ between $Y$ and $W$ is zero.

\paragraph{Proof of Proposition \ref{prop:hatdi 5-pt non-zero} (2).} 

For the pair $(X,Y)$, to show $\hatdi^{\Tri}=\hatdi^{\Cor}=\hatdi^{\Map}=1$, it suffices to prove 
\begin{itemize}
    \item [(a)] $\pbdivariant{\hat}{\Map}{X}{Y}\leq 1$; and 
    \item [(b)] $\pbdivariant{\hat}{\Tri}{X}{Y}\geq 1$. 
\end{itemize}

Item (a) can be obtained by considering pullback vectors
\[\vec{m}=\vec{m}'=(0,0,0,1,4).\]
With a similar argument as in the proof of Proposition \ref{prop:hatdi 5-pt non-zero} (1), we have $(\vec{m},\vec{m}')\in \frM_{\Map}$.
Let $Z:=X(\vec{m})$ and $Z':=Y(\vec{m}')$ be the pullback spaces of $X$ and $Y$, respectively. We can compute the verbose barcodes of $Z$ and $Z'$ from Theorem \ref{thm:verbose of ultra} and Proposition \ref{prop:explicit barc k}, and check that 
\[\pbdivariant{\hat}{\Map}{X}{Y}\leq \di(\fccVR{Z},\fccVR{Z'}) = \max_{k=0,1,\dots,8}\dm(\caB_{k}(Z),\caB_{k}(Z'))=1.\]

For Item (b), note that the statement is equivalent to showing that 
\begin{equation}\label{eq:non-zero}
    \pbdivariant{\hat}{\Tri}{X}{Y}\neq 0.
\end{equation}
The proof of Equation (\ref{eq:non-zero}) is not included in this section due to its technical nature. Interested readers can find more details in Section \ref{app:proof of non-zero}. To summarize, the proof involves solving \emph{Diophantine equations}, which are polynomial equations in two or more unknowns with integer coefficients, such that the only solutions of interest are the integer ones.
The reason for this is that $\pbdivariant{\hat}{\Tri}{X}{Y}= 0$ if and only if there exist pullback vectors $(\vec{m},\vec{m}')$ such that $X(\vec{m})$ and $Y(\vec{m}')$ have the same verbose barcodes in all degrees. And the latter is equivalent to the existence of solutions for equations given by matching the multiplicities of bars in the verbose barcodes.

\begin{remark}\label{rmk:XY hatdb=0}
For the pair $(X,Y)$ and for any $k=1,2,3$, we claim that $\hatdbk{k}^{\Map}$ is zero. It then follows that $\hatdbk{k}^{\Tri}$ and $\hatdbk{k}^{\Cor}$ are also zero. The claim can be proved by considering the following pullback vectors:
\begin{itemize}
    \item For $k=1$: $\vec{m}=(1, 1, 0, 0, 3), \vec{m}'=(0, 3, 0, 1, 1)$. 
    Note that \((\vec{m},\vec{m}') \in \mathfrak{M}_{\text{Map}}\), since there exists \( F \vee F' \) in \( \caX(\vec{m}+1, \vec{m}'+1) \), given by
\[
F = \begin{pmatrix}
    0 & 1 & 0 & 0 & 0 \\
    0 & 1 & 0 & 0 & 0 \\
    0 & 1 & 0 & 0 & 0 \\
    0 & 0 & 0 & 1 & 0 \\
    0 & 0 & 0 & 0 & 1
\end{pmatrix} \in \caX_{\mathrm{row}} 
\quad \text{and} \quad 
F' = \begin{pmatrix}
    0 & 0 & 0 & 0 & 1 \\
    0 & 0 & 1 & 0 & 0 \\
    0 & 0 & 0 & 0 & 0 \\
    0 & 0 & 0 & 0 & 0 \\
    1 & 1 & 0 & 1 & 0
\end{pmatrix} \in \caX_{\mathrm{col}} .
\]
    \item For $k=2$ and $3$: $\vec{m}=(0, 0, 0, 1, 4),  \vec{m}'=(0, 1, 4, 0, 0)$. 
    With a similar argument as in the proof of Proposition \ref{prop:hatdi 5-pt non-zero} (1), we see that $(\vec{m},\vec{m}')$ is in $\frM_{\Map}$.
    \item For $k>3$: $\vec{m}=\vec{m}'=(1,1,1,1,1)$. 
    Note that \((\vec{m},\vec{m}') \in \mathfrak{M}_{\text{Map}}\), since there exists \( F \vee F' \) in \( \caX(\vec{m}+1, \vec{m}'+1) \), given by
\[
F = \begin{pmatrix}
    1 & 0 & 0 & 0 & 0 \\
    0 & 1 & 0 & 0 & 0 \\
    0 & 0 & 1 & 0 & 0 \\
    0 & 0 & 0 & 1 & 0 \\
    0 & 0 & 0 & 0 & 1
\end{pmatrix} \in \caX_{\mathrm{row}} 
\quad \text{and} \quad 
F' = \begin{pmatrix}
    0 & 1 & 0 & 0 & 0 \\
    0 & 0 & 1 & 0 & 0 \\
    0 & 0 & 0 & 1 & 0 \\
    0 & 0 & 0 & 0 & 1 \\
    1 & 0 & 0 & 0 & 0
\end{pmatrix} \in \caX_{\mathrm{col}} .
\]
\end{itemize}
\end{remark}

\newpage

\bibliographystyle{plainurl}
\bibliography{refsfcc.bib} 

\appendix

\section{Appendix}
\subsection{Understanding the Multiplicity Function \texorpdfstring{$\mu_k(\vec{m}(I_p))$}{mu}}
\label{app:mu}
We present some examples to help understand the notation $\mu_k(\vec{m}(I_p))$ introduced in Section \ref{subsec:pb barcodes explicit}.

First consider the case when $k=1$. There are only three possible cases:
\begin{itemize}
    \item $p=1,$ in which situation we must have $q=1$ and $\omega_1=2$;
    \item $p=2$ and $q=1$, in which situation we must have $\omega_1=1$;
    \item $p=2$ and $q=2$, in which situation we must have $\omega_1=\omega_2=1$.
\end{itemize}
Thus, we have
\begin{align}
    \mu_1(\vec{m}([i]))=&\,\sum_{ \omega_1= 2}\, {m_{i}\choose \omega_1}={m_{i}\choose 2}. \notag\\
    \mu_1(\vec{m}([i_1,i_2]))
    =&\,
    \sum_{\omega_1=1}\, {m_{i_1}\choose 1}+
    \sum_{\omega_1=\omega_2=1}{m_{i_1+1}\choose \omega_1}{m_{i_2}\choose \omega_2}  = m_{i_1}m_{i_2}+m_{i_1}+m_{i_2}.\label{eq:mu_1,2}
\end{align}
From the above, it is clear that when  $k=1$, Proposition \ref{prop:explicit barc k} reduces to the following:
\begin{align*}
    \caB_{\Ver,1}(Z)
= \caB_{\Ver,1}(X) \sqcup \bigsqcup_{1\leq p<q\leq n}\left\{  d_X(x_p,x_q)\cdot(1,1)\right\}^{(m_pm_q+m_p+m_q)}
\sqcup \left\{ (0,0)\right\}^{\sum_p{m_p \choose 2}}.
\end{align*}

For $k=2$, we have
\begin{align}
    \mu_2(\vec{m}([i]))=&\,\sum_{ \omega_1= 3}\, {m_{i}\choose \omega_1}={m_{i}\choose 3}. \notag\\
    \mu_2(\vec{m}([i_1,i_2]))
    =&\,{m_{i_1}\choose 2}+ {m_{i_1}+1\choose 1}{m_{i_2}\choose 2} + {m_{i_1}+1\choose 2}{m_{i_2}\choose 1}\notag\\
    =&\,(m_{i_1}+1)(m_{i_2}+1) \tfrac{m_{i_1}+m_{i_2}-2}{2}+1.\label{eq:mu_2,2}
\end{align}

One more example of $\mu_k(\vec{m}(I_p))$ is when $p=k+1$. In this case, we have
\begin{align}
\label{eq:mu_k,k+1}
    \mu_k(\vec{m}([i_1,\dots,i_{k+1}])
    =& \sum_{q=1}^{k+1}\,\sum_{\substack{ \omega_1,\dots,\omega_q\geq 1 \\\omega_1+\dots+\omega_q=q }}\, {m_{i_1}+1\choose \omega_1}{m_{i_2}+1\choose \omega_2}\dots {m_{i_{q-1}}+1\choose\omega_{q-1}}{m_{i_q}\choose \omega_q}\notag\\
    =&\sum_{q=1}^{k+1}(m_{i_1}+1)(m_{i_2}+1)\dots (m_{i_{q-1}}+1)m_{i_{q}}.
\end{align}

\begin{remark} \label{rmk:recursive mu_k} The functions $\{\mu_k(\cdot)\}_k$ satisfy the following recursive relation:
    \begin{align*}
        &\mu_{k}(\vec{m}([i_1,\dots,i_p]))-\mu_{k-1}(\vec{m}([i_1,\dots,i_{p-1}]))\\
        =&\left(\sum_{q=1}^p\,\sum_{\substack{ \omega_1,\dots,\omega_q\geq 1 \\\omega_1+\dots+\omega_q+(p-q)=k+1 }}-
        \sum_{q=1}^{p-1}\,\sum_{\substack{ \omega_1,\dots,\omega_q\geq 1 \\\omega_1+\dots+\omega_q+(p-1-q)=k }}\right) {m_{i_1}+1\choose \omega_1}\dots {m_{i_{q-1}}+1\choose\omega_{q-1}}{m_{i_q}\choose \omega_q}\\
        =& \sum_{q=p}\,\sum_{\substack{ \omega_1,\dots,\omega_q\geq 1 \\\omega_1+\dots+\omega_q=k+1 }} {m_{i_1}+1\choose \omega_1}\dots {m_{i_{q-1}}+1\choose\omega_{q-1}}{m_{i_q}\choose \omega_q}\\
        =& \sum_{\substack{ \omega_1,\dots,\omega_p\geq 1 \\\omega_1+\dots+\omega_p=k+1 }} {m_{i_1}+1\choose \omega_1}\dots {m_{i_{p-1}}+1\choose\omega_{p-1}}{m_{i_p}\choose \omega_p}.
    \end{align*}
\end{remark}

\subsection{Proof of Equation (\ref{eq:non-zero})} \label{app:proof of non-zero}

We establish an important Lemma that provides a necessary condition for pullback interleaving-type distances to be zero (see Lemma \ref{lem:ultra-permutation}). This lemma allows us to restrict our attention to pullback vectors $(\vec{m},\vec{m}')$ that are a certain type of permutation of one another.

For any integer $n\geq 2$, let $\rmS_{n-1}$ be the group of all permutations of the set ${1,\dots,n-1}$. 
Let $\sigma\in \rmS_{n-1}$ be a permutation with the property that $\sigma(1)=1$. Define $X_\sigma$ to be the ultra-metric space such that $u_X(x_i,x_{i+1}):=\sigma(i)$ for all $i=1,\dots,n-1$, and $u_X(x_i,x_j):=\max_{l=i,\dots,j-1}\sigma(l)$ for all $1\leq i<j\leq n$.\label{para:caU_n}
Denote
\[\caU_n:=\left\{ (X_\sigma,u_X)\mid \sigma\in \rmS_{n-1}, \, \sigma(1)=1 \right\}.\]

\begin{lemma} \label{lem:ultra-permutation}
Let $X,X'$ be in $\caU_n$, and $\vec{m},\vec{m}'\in\bbN^n$. Let $Z:=X(\vec{m})$ and $Z':=X'(\vec{m}')$ be the pullback spaces of $X$ and $X'$, respectively. Then, 
\begin{itemize}
    \item[(1)] $\caB_{\Ver,k}(Z)$ and $\caB_{\Ver,k}(Z')$ have the same number of $(1,1)$ for both $k=1,2$, and 
    \item[(2)] $\caB_{\Ver,k}(Z)$ and $\caB_{\Ver,k}(Z')$ have the same number of $(0,0)$ for all $k=0,1,\dots,n-3$,
\end{itemize}
iff 
\begin{itemize}
    \item [(I)] $(m_1,m_2)$ and $(m_1',m_2')$ differ by a permutation, and
    \item [(II)] $(m_3,\dots, m_n)$ and $(m_3',\dots, m_n')$ differ by a permutation.
\end{itemize}
Moreover, if these conditions are satisfied, the multiplicity of $(0,0)$ (or $(1,1)$ respectively) in $\caB_{\Ver,k}(Z)$ and $\caB_{\Ver,k}(Z')$ matches for $k = 3, \dots, n + \|\vec{m}\|_1 - 2$.
\end{lemma}

\begin{proof}
\label{pf:permutaton lemma}
We first prove that Items (1) and (2) imply Items (I) and (II). 

It follows from Proposition \ref{prop:explicit barc k} and Equations (\ref{eq:mu_1,2}) and (\ref{eq:mu_2,2}) that the multiplicities of $(1,1)$ in $\caB_{\Ver,k}(Z)$ for $k=1,2$ are:
\begin{itemize}
    \item $k=1$: $\mu_1(\vec{m}([1,2]))=(m_1+1)(m_2+1)-1$;
    \item $k=2$: $\mu_2(\vec{m}([1,2])) = (m_{1}+1)(m_{2}+1) \tfrac{m_{1}+m_{2}-2}{2}+1.$
\end{itemize}
Similarly, we have the multiplicities of $(1,1)$ in $\caB_{\Ver,k}(Z')$ for $k=1,2$:
\begin{itemize}
    \item $k=1$: $\mu_1(\vec{m}'([1,2]))=(m_1'+1)(m_2'+1)-1$;
    \item $k=2$: $\mu_2(\vec{m}'([1,2])) =(m_{1}'+1)(m_{2}'+1) \tfrac{m_{1}'+m_{2}'-2}{2}+1$.
\end{itemize}
Thus, for Item (1) to hold, the following system of equations is satisfied
\begin{align*}
    (m_1+1)(m_2+1)-1&=(m_1'+1)(m_2'+1)-1\\
    (m_{1}+1)(m_{2}+1) \tfrac{m_{1}+m_{2}-2}{2}+1&=(m_{1}'+1)(m_{2}'+1) \tfrac{m_{1}'+m_{2}'-2}{2}+1.
\end{align*}
Because $m_1+1,m_2+1>0$ are non-zero, the above equations can be simplified to
\begin{align*}
    (m_1+1)(m_2+1)&=(m_1'+1)(m_2'+1)\\
    m_1+m_2&=m_1'+m_2'.
\end{align*}
As a consequence, we can infer that $(m_1+1,m_2+1)$ and $(m_1'+1,m_2'+1)$ are permutations of each other, which implies that $(m_1,m_2)$ and $(m_1',m_2')$ are also permutations of each other.

By Proposition \ref{prop:pullback-barcode} and Proposition \ref{prop:explicit barc k}, for any $k=0,\dots,n-3$, the multiplicities of $(0,0)$ in the verbose barcodes $\caB_{\Ver,k}(Z)$ and $\caB_{\Ver,k}(Z')$ are $\sum_p{m_p \choose k+1}$ and $\sum_p{m_p' \choose k+1}$, respectively. 
Thus, for Item (2) to hold, $\vec{m}$ and $\vec{m}'$ need to satisfy the following $(n-2)$ equations:
\[\sum_p{m_p \choose 1} =\sum_p{m_p' \choose 1},\, 
\dots,\,
\sum_p{m_p \choose n-2} =\sum_p{m_p' \choose n-2}.\]
It is clear that the above system of equations is equivalent to the system of equations
\[\| \vec{m}\|_1 =\| \vec{m}'\|_1,\, 
\dots,\,\| \vec{m}\|_{n-2} =\| \vec{m}'\|_{n-2},\]
where $\|\cdot\|_{k+1}$ denotes the $(k+1)$-norm of a vector.

Let $\vec{z}:=(m_3 ,\dots,m_n )\in  \bbN^{n-2}$ and $\vec{z}':=(m_3' ,\dots,m_n' )\in  \bbN^{n-2}$. Because $(m_1,m_2)$ and $(m_1',m_2')$ are permutations of each other, we have 
\[\| \vec{z}\|_{k+1}^{k+1}=\| \vec{m}\|_{k+1}^{k+1}-\|(m_1 ,m_2 )\|_{k+1}^{k+1}
=\| \vec{m}'\|_{k+1}^{k+1}-\|(m_1' ,m_2' )\|_{k+1}^{k+1}
=\| \vec{z}'\|_{k+1}^{k+1}\] 
for any $k$, i.e. the following $(n-2)$ equations hold:
\begin{equation}\label{eq:norms}
    \| \vec{z}\|_1 =\| \vec{z}'\|_1,\, 
\| \vec{z}\|_2 =\| \vec{z}'\|_2,\, 
\dots,\,\| \vec{z}\|_{n-2} =\| \vec{z}'\|_{n-2}.
\end{equation}

Let $f(x):=\prod_{i=1}^{n-2}(x-z_i)-\prod_{i=1}^{n-2}(x-z_i')$. Then the system of equations (\ref{eq:norms}) guarantees that $f(x)\equiv 0$.
This follows from Newton's identities, which say that power sums $\left\{\sum_{j=1}^{n-2}z_j^{k}\right\}_{k=1}^{n-2}$ and symmetric polynomials $\left\{\left(z_1+\dots+z_{n-2}\right),\left(\sum_{1\leq j_1<j_2\leq n-2}z_{j_1}z_{j_2}\right),\dots,\left(z_1\dots z_{n-2}\right)\right\}$ determine each other. Because these symmetric polynomials determine $f(x)$, the power sums also determine $f(x)$.

Thus, $\vec{z}$ and $\vec{z}'$ differ by a permutation,
which implies that $(m_3,\dots, m_n)$ and $(m_3',\dots, m_n')$ differ by a permutation. Therefore, we have proved that Items (I) and (II) hold.

For the other direction, it is straightforward to verify that Items (I) and (II) imply Items (1) and (2).\\

Moreover, if Items (I) and (II) are satisfied, the following is true for any \(k = 3, \dots, n + \|\vec{m}\|_1 - 2\):
The multiplicities of $(0,0)$ and $(1,1)$ in $\caB_{\Ver,k}(Z)$ and $\caB_{\Ver,k}(Z')$ match. 
This equivalence is due to the preservation of the multiplicity of $(0,0)$ in $\caB_{\Ver,k}(Z)$ under permutations and the multiplicity of $(1,1)$ being a symmetric function in $m_1$ and $m_2$.
\end{proof}

Let $X$ and $Y$ be given as in Figure \ref{fig:5-point-ultra}. We apply Lemma \ref{lem:ultra-permutation} to show $\pbdivariant{\hat}{\Tri}{X}{Y}\neq 0$, i.e. Equation (\ref{eq:non-zero}).

\begin{proof}[Proof of `$\pbdivariant{\hat}{\Tri}{X}{Y}\neq 0$']

It suffices to show that there exist no pullback vectors $\vec{m}\in \bbN^5$ for $X$ and $\vec{m}'\in\bbN^5$ for $Y$ such that
\[\max_{0\leq k\leq 3}\dm(\caB_{\Ver,k}(X(\vec{m})),\caB_{\Ver,k}(Y(\vec{m}')))=0.\]
In other words,
\[\caB_{\Ver,k}(X(\vec{m}))=\caB_{\Ver,k}(Y(\vec{m}')),\,\forall k=0,\dots,3.\]
We prove these equalities by contradiction. Assume such $\vec{m},\vec{m}'\in \bbN^5$ exist. Then they must satisfy a certain system of Diophantine equations. 

Because of Lemma \ref{lem:ultra-permutation}, the pullback vectors $\vec{m},\vec{m}'\in \bbN^5$ satisfies the following properties:
\begin{itemize}
    \item[(I)] $(m_1,m_2)$ and $(m_1',m_2')$ differ by a permutation, and \item[(II)] $(m_3,m_4, m_5)$ and $(m_3',m_4', m_5')$ differ by a permutation.
\end{itemize}
Let $Z:=X(\vec{m})$ and $Z':=Y(\vec{m}')$ be the pullback spaces of $X$ and $Y$, respectively. 

Conditions (I) and (II) guarantee that $ \caB_{\Ver,k}(Z)$ and $ \caB_{\Ver,k}(Z')$ contain the same number of $(0,0)$ and $(1,1)$ for all $k=0,\dots,n-2$. 
Given that the total cardinalities of $\caB_{\Ver,k}(Z)$ and $\caB_{\Ver,k}(Z')$ are equal, matching the multiplicities of $(2,2)$ and $(3,3)$ in both barcodes implies that the multiplicity of $(4,4)$ will also match.
Thus, it remains to match the multiplicities of $(2,2)$ and $(3,3)$ in $\caB_{\Ver,k}(Z)$ and $\caB_{\Ver,k}(Z')$. 

For degree $1$, the multiplicity of $(a,a)$ in $ \caB_{\Ver,1}(Z)$ is:
\begin{itemize}
    \item $a=2$: $\mu_1(\vec{m}([1,3]))+\mu_1(\vec{m}([2,3]))+1$;
    \item $a=3$: $\mu_1(\vec{m}([4,5]))$.
\end{itemize}
And the multiplicity of $(a,a)$ in $ \caB_{\Ver,1}(Z')$ is:
\begin{itemize}
    \item $a=2$: $\mu_1(\vec{m}'([4,5]))$;
    \item $a=3$: $\mu_1(\vec{m}'([3,4]))+\mu_1(\vec{m}'([3,5]))+1$.
\end{itemize}
Applying Equation (\ref{eq:mu_1,2}), we want $\vec{m}$ and $\vec{m}'$ to satisfy:
\begin{align}
    (m_{1}+1)(m_{3}+1)-1 +(m_{2}+1)(m_{3}+1) =&(m_{4}'+1)(m_{5}'+1)-1\notag\\
    (m_3'+1)(m_4'+1)-1 +(m_3'+1)(m_5'+1)
    =&(m_4+1)(m_5+1)-1.\notag
\end{align}
To simplify the equations, we define 
\[\vec{\zeta}:=(m_1+1,\dots,m_5+1)\text{ and }\vec{\zeta}':=(m_1'+1,\dots,m_5'+1).\]
Then the above equations can be rewritten as:
\begin{align}
    \zeta_1\zeta_3 +\zeta_2\zeta_3 =&\zeta_4'\zeta_5'\label{eq:1,2}\\
    \zeta_3'\zeta_4' +\zeta_3'\zeta_5'
    =&\zeta_4\zeta_5.\label{eq:1,3}
\end{align}

For degree $2$, it follows from Proposition \ref{prop:explicit barc k} that the multiplicity of $(a,a)$ in $ \caB_{\Ver,2}(Z)$ is:
\begin{itemize}
    \item $a=2$: $\mu_2(\vec{m}([1,3]))+\mu_2(\vec{m}([2,3]))+\mu_2(\vec{m}([1,2,3]))$; 
    \item $a=3$: $\mu_2(\vec{m}([4,5]))$.
\end{itemize}
And the multiplicity of $(a,a)$ in $ \caB_{\Ver,2}(Z')$ is:
\begin{itemize}
    \item $a=2$: $\mu_2(\vec{m}'([4,5]))$;
    \item $a=3$: $\mu_2(\vec{m}'([3,4]))+\mu_2(\vec{m}'([3,5]))+\mu_2(\vec{m}'([3,4,5]))$.
\end{itemize}
Applying Equation (\ref{eq:mu_2,2}) and (\ref{eq:mu_k,k+1}) and substituting $\vec{\zeta}$ and $\vec{\zeta}'$ for $\vec{m}$ and $\vec{m}'$, we obtain:
\begin{align}
    {\zeta_1\zeta_3}(\tfrac{\zeta_1+\zeta_3}{2}-2) + {\zeta_2\zeta_3}(\tfrac{\zeta_2+\zeta_3}{2}-2)+1+ \zeta_1\zeta_2\zeta_3-1=&{\zeta_4'\zeta_5'}(\tfrac{\zeta_4'+\zeta_5'}{2}-2)\label{eq:2,2}\\
    {\zeta_3'\zeta_4'}(\tfrac{\zeta_3'+\zeta_4'}{2}-2) + {\zeta_3'\zeta_5'}(\tfrac{\zeta_3'+\zeta_5'}{2}-2)+1+ \zeta_3'\zeta_4'\zeta_5'-1=&{\zeta_4\zeta_5}(\tfrac{\zeta_4+\zeta_5}{2}-2).\label{eq:2,3}
\end{align}

Let us assume 
\[\alpha:=\zeta_3, \beta:=\zeta_4, \gamma:=\zeta_5.\]
By Equation (\ref{eq:1,3}) and Item (II), we have
\begin{align*}
\zeta_4\zeta_5=\zeta_3'(\zeta_4'+\zeta_5')&\implies \zeta_4\zeta_5+\zeta_4'\zeta_5'=\zeta_3'(\zeta_4'+\zeta_5')+\zeta_4'\zeta_5'\\
&\implies \beta\gamma+\zeta_4'\zeta_5' = \alpha\beta+\beta\gamma+\alpha\gamma\\
&\implies \zeta_4'\zeta_5' = \alpha\beta+\alpha\gamma.
\end{align*}
There are three possibilities for the multiset $\{\zeta_4',\zeta_5'\}$: $\{\alpha,\beta\},\{\alpha,\gamma\},\{\beta,\gamma\}$. Because $\alpha,\beta,\gamma>0$, the first two choices will yield a contradiction. Thus, $\{\zeta_4',\zeta_5'\}=\{\beta,\gamma\}$. This implies 
\begin{center}
    $\zeta_3'=\alpha=\zeta_3$ and $\beta\gamma=\alpha\beta+\alpha\gamma.$
\end{center} 
Going back to Equation (\ref{eq:1,2}), we obtain that
\begin{align*}
(\zeta_1+\zeta_2)\zeta_3=\zeta_4'\zeta_5'\implies (\zeta_1+\zeta_2)\alpha=\beta\gamma
\implies \zeta_1+\zeta_2=\frac{\beta\gamma}{\alpha}=\beta+\gamma,
\end{align*}
where we applied the fact that all variables involved are non-zero.

Assume $\xi:=\zeta_1$. So far, we have shown that $\vec{\zeta}$ and $\vec{\zeta}'$ must be of the following forms:
\begin{itemize}
    \item $\vec{\zeta} = (\xi, \beta+\gamma-\xi,\alpha,\beta,\gamma)$ for some $\xi,\alpha,\beta,\gamma\in \bbZ_{\geq 1}$ satisfying $\beta\gamma=\alpha\beta+\alpha\gamma$;
    \item $\vec{\zeta}' $ is of one of the following four: $\vec{\zeta}, (\beta+\gamma-\xi,\xi,\alpha,\beta,\gamma),(\xi, \beta+\gamma-\xi,\alpha,\gamma,\beta)$ or $(\beta+\gamma-\xi,\xi,\alpha,\gamma,\beta)$.
\end{itemize}

Applying Equation (\ref{eq:2,3}), we obtain:
\begin{align*}
0=&{\zeta_3'\zeta_4'}(\tfrac{\zeta_3'+\zeta_4'}{2}-2) + {\zeta_3'\zeta_5'}(\tfrac{\zeta_3'+\zeta_5'}{2}-2)+1+ \zeta_3'\zeta_4'\zeta_5'-1-{\zeta_4\zeta_5}(\tfrac{\zeta_4+\zeta_5}{2}-2)\\
=&{\alpha \zeta_4'}(\tfrac{\alpha+\zeta_4'}{2}-2) + {\alpha \zeta_5'}(\tfrac{\alpha+\zeta_5'}{2}-2)+1+ \alpha \beta\gamma-1-{\beta\gamma}(\tfrac{\beta+\gamma}{2}-2)\\
=&\tfrac{(\alpha-4)\alpha(\beta+\gamma)+\alpha(\beta^2+\gamma^2)}{2}+ \alpha \beta\gamma -\tfrac{{\beta\gamma}(\beta+\gamma)}{2}+2\beta\gamma
\\
=&\tfrac{(\alpha-4)\beta\gamma+\alpha(\beta^2+\gamma^2)}{2} + \alpha \beta\gamma -\tfrac{{\beta\gamma}(\beta+\gamma)}{2}+2\beta\gamma
\\
=&\tfrac{\alpha(3\beta\gamma+\beta^2+\gamma^2)}{2} -\tfrac{{\beta\gamma}(\beta+\gamma)}{2}\\
=&\tfrac{\alpha(\beta\gamma+(\beta+\gamma)^2)-\beta\gamma(\beta+\gamma)}{2} \\
=&\tfrac{\alpha\beta\gamma}{2} .
\end{align*}
contradicts the fact that $\alpha,\beta,\gamma>0.$

Therefore, we have proved that there exists no $\vec{\zeta},\vec{\zeta}'\in \bbZ_{\geq1}^5$, i.e. there are no $\vec{m},\vec{m}'\in \bbN^5$, such that \[\caB_{\Ver,k}(X(\vec{m}))=\caB_{\Ver,k}(Y(\vec{m}')),\,\forall k=0,\dots,3.\]
Consequently, we conclude that $\hatdi^{\Tri} \neq 0$, thus completing the proof. 

As an additional remark, it is noteworthy that not all conditions were utilized in our argument. Specifically, we did not utilize Equation (\ref{eq:2,2}) and equations that match the multiplicities of $(2,2)$ and $(3,3)$ in degree $3$.
\end{proof}

\end{document}